\documentclass[12pt]{article}
\usepackage{amsmath,mathdots}
\usepackage{amssymb}
\usepackage{latexsym}
\usepackage{amsthm}
\usepackage{amsmath,amssymb,amsfonts,amsthm}
\usepackage{mdframed}
\usepackage{graphicx}
\usepackage{stmaryrd}
\usepackage{gensymb}
\usepackage{bm}
\usepackage{tikz}  
\usepackage{longtable}
\usepackage{stmaryrd}
\usepackage{multirow}
\usepackage{float}
\usepackage{array}
\usepackage{graphicx}
\usepackage{tikz}
\usetikzlibrary{matrix,arrows}
\usetikzlibrary{positioning}
\usetikzlibrary{fit}
\usetikzlibrary{patterns}

\usepackage[colorlinks,
linkcolor=blue,
anchorcolor=blue,
citecolor=blue
]{hyperref}
\usepackage[margin=2.9cm]{geometry}
\usepackage[normalem]{ulem}

\newtheorem{thm}{Theorem}[section]

\newtheorem{lem}[thm]{Lemma}

\newtheorem{conj}{Conjecture}

\usepackage{graphicx}
\usepackage{tikz}
\usetikzlibrary{matrix,arrows}
\usetikzlibrary{positioning}
\usetikzlibrary{fit}
\usetikzlibrary{patterns}

\newcommand{\pattern}[4]{										
	\raisebox{0.6ex}{
		\begin{tikzpicture}[scale=0.35, baseline=(current bounding box.center), #1]
		\foreach \x/\y in {#4}		\fill[gray!50] (\x,\y) rectangle +(1,1);
		\draw (0.01,0.01) grid (#2+0.99,#2+0.99);
		\foreach \x/\y in {#3}		\filldraw (\x,\y) circle (6pt);
		\end{tikzpicture}}
}

\newcommand{\boks}[2]{({#1, #2})}   

\definecolor{red}{rgb}{1,0,0}
{}

 \allowdisplaybreaks

\begin{document}

\begin{center}
{\large Equidistribution of mesh patterns of short length}
\end{center}

\begin{center}
Xinyu Su$^{a}$, Sergey Kitaev$^{b}$ and Jiahao Zhang$^{c}$
\\[6pt]

$^{a,c}$College of Mathematical Sciences \& Institute of Mathematics and Interdisciplinary Sciences, Tianjin Normal University, \\ Tianjin  300387, P. R. China\\[6pt]

$^{b}$Department of Mathematics and Statistics, University of Strathclyde, \\ 26 Richmond Street, Glasgow G1 1XH, UK\\[6pt]

Email:  $^{a}${\tt  suxinyu2220@163.com},
           $^{b}${\tt sergey.kitaev@strath.ac.uk},
           $^{c}${\tt jiahaozhang1215@163.com}
\end{center}

\noindent\textbf{Abstract.}

We study the equidistribution of mesh patterns of length~2. We show that the number of equidistribution equivalence classes lies between 105 and 108, and conjecture that it is exactly 105. As a consequence, we obtain an upper bound of 49 Wilf-classes, improving the previously known bound of 56 due to Hilmarsson et al., and reducing the problem to three remaining conjectural equivalences (with the actual number conjectured to be 46).

Our approach combines bijective constructions, generating functions, recurrence relations, and structural symmetries. We establish several new equidistribution results, including four previously unknown distribution classes, connect numerous patterns to known distributions in the literature, and resolve seven open pattern-avoidance enumeration problems posed by Hilmarsson et al.

This work provides a near-complete classification of mesh patterns of length~2 and unifies several previously isolated results within a coherent framework. \\

\noindent {\bf Keywords:}  mesh pattern, distribution, avoidance, bijection, generating function \\

\noindent {\bf AMS Subject Classifications:}  05A05, 05A15.

\section{Introduction}\label{intro}

Denote by $S_n$ the set of permutations of $[n] := \{1,2,\ldots,n\}$, and let $\varepsilon$ be the empty permutation. We call permutations of length $n$ $n$-permutations. A \emph{pattern} is a permutation. A permutation $\pi = \pi_1 \pi_2 \cdots \pi_n \in S_n$ avoids a pattern $p = p_1 p_2 \cdots p_k \in S_k$ if there is no subsequence $\pi_{i_1} \pi_{i_2} \cdots \pi_{i_k}$ such that $\pi_{i_j} < \pi_{i_m}$ if and only if $p_j < p_m$. For example, the permutation $32154$ avoids the pattern $231$.

The notion of a \textit{mesh pattern}, generalizing several classes of patterns, was introduced by Br\"and\'en and Claesson \cite{BrCl} to provide explicit expansions for certain permutation statistics as, possibly infinite, linear combinations of (classical) permutation patterns. 
A pair $(\tau,R)$, where $\tau$ is a permutation of length $k$ and $R$ is a subset of $[0,k] \times [0,k]$, where
$[0,k]$ denotes the interval of the integers from $0$ to $k$, is a
mesh pattern of length $k$.
Let $\boks{i}{j}$ denote the box whose corners have coordinates $(i,j), (i,j+1),
(i+1,j+1)$, and $(i+1,j)$. Let the horizontal lines represent the values,  and the vertical lines denote the positions in the pattern. Mesh patterns can be drawn by shading the boxes in $R$. The picture 
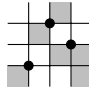
\[
\pattern{scale=0.8}{3}{1/1,2/3,3/2}{0/0,1/2, 2/1,2/3,3/0,3/1}
\]
represents the mesh pattern with $\tau=132$ and $R = \{\boks{0}{0},\boks{1}{2},\boks{2}{1},\boks{2}{3},\boks{3}{0},\boks{3}{1}\}$. 

A subsequence $\pi' = \pi_{i_1}\pi_{i_2}\cdots \pi_{i_k}$ of a permutation $\pi = \pi_1\pi_2 \cdots \pi_n$ is an occurrence of a mesh pattern $(\tau,R)$ if (a) $\pi'$ is order-isomorphic to $\tau$, and (b) the shaded squares specified by $R$ contain no elements of $\pi$ other than those in $\pi'$.  

Let $S_n(p)$ denote the set of $p$-avoiding $n$-permutations, and let $s_n(p) = |S_n(p)|$. For two patterns $p_1$ and $p_2$, if $s_n(p_1) = s_n(p_2)$ for all $n \ge 0$, then $p_1$ and $p_2$ are \emph{Wilf-equivalent}. Wilf-equivalence classes are also referred to here as \emph{Wilf-classes}. Define $s_{n,k}(p)$ as the number of $n$-permutations containing exactly $k$ occurrences of a pattern $p$. If $s_{n,k}(p_1) = s_{n,k}(p_2)$ for all $n,k \ge 0$, then $p_1$ and $p_2$ are \emph{equidistributed} (or \emph{equivalent}), which we denote by $p_1 \sim p_2$. The patterns $p_1$ and $p_2$ are said to be {\em jointly equidistributed} if, for all $n \ge 0$, the number of $n$-permutations with $k$ occurrences of $p_1$ and $\ell$ occurrences of $p_2$ is equal to the number of $n$-permutations with $\ell$ occurrences of $p_1$ and $k$ occurrences of $p_2$.

Hilmarsson et al.~\cite{Hilmarsson2015Wilf} initiated a systematic study of Wilf-equivalences and the enumeration of permutations avoiding mesh patterns of length 2. As a result of their work, they proved that there are at most 56 Wilf-classes for these patterns and conjectured that there are in fact only 46 such classes. In Table~\ref{tab-Wilf-conjectures}, we present the conjectured Wilf-equivalences from \cite{Hilmarsson2015Wilf}. As a result of our study, we strengthen the results of Hilmarsson et al.\ by proving the following theorem and conjecturing that all previously conjectured Wilf-equivalences are indeed equidistributions of patterns.
 
\begin{thm}
The number of Wilf-equivalence classes for mesh patterns of length~$2$ is at most $49$.
\end{thm}
 
 \begin{proof} 
From~\cite{Hilmarsson2015Wilf}, the number of Wilf-classes is known to be at most 56. Using our numbering of pattern classes, some of the conjectured Wilf-classes in \cite{Hilmarsson2015Wilf} are Classes 36, 46, 47, 56, and 70. These classes have now been proved to be equidistributed (a stronger result than Wilf-equivalence): Classes 36 and 46 are established in Theorems~\ref{thm:class36} and~\ref{thm:class46} of this paper, Classes 47 and 70 in \cite{HanZeng2021}, and Class 56 in \cite{KZ2019}. Since only three pairs of patterns with conjectured Wilf-equivalences (in fact, conjectured equidistributions) remain, namely, Classes 54, 69, and 71 in Table~\ref{tab-Wilf-conjectures}, the number of Wilf-classes is at most~49. 
\end{proof}

\begin{table}[!t]
    \renewcommand{\arraystretch}{1.6}
    \centering
    \begin{tabular}{|c|c|c||c|c|c|}
        \hline
        \footnotesize{nr.} & \footnotesize{repr. $p$} & \footnotesize{ref.} & \footnotesize{nr.} & \footnotesize{repr. $p$} & \footnotesize{ref.} \\
        \hline

        \multirow{3}{*}{\footnotesize{47}} & $\pattern{scale=0.5}{2}{1/1,2/2}{0/0,0/1,1/2,2/1,2/2}~(48)$ & \multirow{3}*{\shortstack{\cite{HanZeng2021} \\ proved \\ equidistribution}} & 
        \multirow{2}{*}{\footnotesize{69}} & $\pattern{scale=0.5}{2}{1/1,2/2}{0/1,1/1,1/2,2/0}~(57)$ & \multirow{2}*{\shortstack{\cite{KZ2019} \\ conjectured \\ equidistribution}} \\ 
        \cline{2-2} \cline{5-5} 

        & $\pattern{scale=0.5}{2}{1/1,2/2}{0/0,0/1,1/1,1/2,2/0}~(49)$ & & & 
        $\pattern{scale=0.5}{2}{1/1,2/2}{0/1,1/0,1/1,2/2}~(58)$ & \\ 
        \cline{2-2} \cline{4-6} 

        & $\pattern{scale=0.5}{2}{1/1,2/2}{0/0,0/1,1/1,1/2,2/2}~(50)$ & & 
        \multirow{2}{*}{\footnotesize{46}} & $\pattern{scale=0.5}{2}{1/1,2/2}{0/0,0/1,1/1,1/2,2/1}~(59)$ & \multirow{2}*{\shortstack{Thm~\ref{thm:class46} \\ proved \\ equidistribution}} \\ 
        \cline{1-3} \cline{5-5} 

        \multirow{2}{*}{\footnotesize{54}} & $\pattern{scale=0.5}{2}{1/1,2/2}{0/0,0/1,1/1,2/0,2/2}~(51)$ & \multirow{2}*{\shortstack{conjectured \\ equidistribution}} & & 
        $\pattern{scale=0.5}{2}{1/1,2/2}{0/0,0/1,1/1,2/1,2/2}~(60)$ & \\ 
        \cline{2-2} \cline{4-6} 

        & $\pattern{scale=0.5}{2}{1/1,2/2}{0/1,0/2,1/1,2/0,2/2}~(52)$ & & 
        \multirow{2}{*}{\footnotesize{71}} & $\pattern{scale=0.5}{2}{1/1,2/2}{0/0,0/1,1/2,2/0}~(61)$ & \multirow{2}*{\shortstack{\cite{KZ2019} \\ conjectured \\  equidistribution}} \\ 
        \cline{1-3} \cline{5-5} 

        \multirow{2}{*}{\footnotesize{70}} & $\pattern{scale=0.5}{2}{1/1,2/2}{0/0,0/1,1/2,2/1}~(53)$ & \multirow{2}*{\shortstack{\cite{HanZeng2021} \\ proved \\ equidistribution}} & & 
        $\pattern{scale=0.5}{2}{1/1,2/2}{0/0,0/1,1/0,2/2}~(62)$ & \\ 
        \cline{2-2} \cline{4-6} 

        & $\pattern{scale=0.5}{2}{1/1,2/2}{0/0,0/1,1/1,2/2}~(54)$ & & 
        \multirow{3}{*}{\footnotesize{56}} & $\pattern{scale=0.5}{2}{1/1,2/2}{0/0,0/1,1/2,2/0,2/1}~(63)$ & \multirow{3}*{\shortstack{\cite{KZ2019} \\ proved \\  equidistribution}} \\ 
        \cline{1-3} \cline{5-5} 

        \multirow{2}{*}{\footnotesize{36}} & $\pattern{scale=0.5}{2}{1/1,2/2}{0/0,0/1,1/1,1/2,2/0,2/1}~(55)$ & \multirow{2}*{\shortstack{Thm~\ref{thm:class36} \\ proved \\ equidistribution}} & & 
        $\pattern{scale=0.5}{2}{1/1,2/2}{0/1,0/2,1/1,1/2,2/0}~(64)$ & \\ 
        \cline{2-2} \cline{5-5}

        & $\pattern{scale=0.5}{2}{1/1,2/2}{0/0,0/1,1/1,1/2,2/1,2/2}~(56)$ & & & 
        $\pattern{scale=0.5}{2}{1/1,2/2}{0/0,0/1,1/0,1/1,2/2}~(65)$ & \\ 
        \hline
    \end{tabular}
\caption{Conjectured Wilf-equivalences in Tables 11 and 12 of \cite{Hilmarsson2015Wilf}, which have now been proved as equidistributions or remain conjectured equidistributions. Next to each pattern in the table, we give the original number of the pattern class in \cite{Hilmarsson2015Wilf}.}    \label{tab-Wilf-conjectures}
\end{table}

Additionally, as consequences of our enumerative results in this paper, we solve seven open pattern-avoidance enumeration problems from \cite{Hilmarsson2015Wilf}. In particular, Theorems~\ref{thm:class36}, \ref{thm:class47}, and \ref{thm:class54} resolve patterns labelled in \cite{Hilmarsson2015Wilf} as follows: patterns 55 and 56 (corresponding to our Class~36), patterns 48--50 (Class~47), and patterns 51 and 52 (Class~54).

The results in \cite{Hilmarsson2015Wilf} were extended to a systematic study of distributions of mesh patterns of length~2 by Kitaev and Zhang in \cite{KZ2019}. In particular, the authors of \cite{KZ2019} obtained 27 distribution results for these patterns, thereby resolving some of the avoidance problems posed in \cite{Hilmarsson2015Wilf}. Some of the equidistribution conjectures from \cite{KZ2019} were settled by Han and Zeng in \cite{HanZeng2021}, while two cases remain open. We summarize the known equidistribution results in Table~\ref{tab-known}. The study of distributions of patterns of longer lengths can be found in \cite{FFKL2025,KZZ,LK2025,LZ2026,LZ2026-2}, with the main focus on the equidistribution of mesh patterns of length~3.

As the main result of this paper, we provide a nearly complete classification of equidistribution equivalence classes for mesh patterns of length~2, showing that the number of these classes is at most 108 and at least 105. We conjecture that this number is in fact 105. Therefore, apart from two known equidistribution conjectures, we propose one additional conjecture (see Table~\ref{tab-conject}). Resolving these three conjectures would settle the open problem that the number of Wilf-classes for these patterns is 46; see Table~\ref{tab-Wilf-conjectures}.

To obtain these results, we construct four involved bijections, use several simpler bijections and symmetries, apply generating functions and recurrence relations, and exploit known results from the literature. Our new enumerative contributions include four previously unknown distributions (Classes 15, 25, 73, and 80), as well as non-trivial connections between many patterns and known distributions from \cite{KZ2019,KZZ}.

The paper is organized as follows. In Section~\ref{prelim-sec}, we provide definitions, preliminary observations, and notation, together with known enumerative results, some of which are used later in the paper, while others are included to make the exposition self-contained; the known enumerative results are summarized in Table~\ref{tab:patterns_ref}. In Section~\ref{sec-trivial}, we present straightforward equidistributions, which are summarized in Table~\ref{tab-easy}. Equidistributions explained via generating functions and recurrence relations are discussed in Sections~\ref{gf-sec} and~\ref{recurrence-sec}, respectively. In Section~\ref{bijections-sec}, we present four involved bijections explaining additional equidistributions. 

In Section~\ref{avoiders-enum-sec}, we enumerate avoidance classes corresponding to some equivalence classes arising from trivial bijections; these are presented in Table~\ref{tab-remaining} in Appendix~\ref{trivial-appendix}. This section also contains two additional enumerative results not related to Table~\ref{tab-remaining}. Finally, in Section~\ref{conclud-sec}, we provide concluding remarks, and in Appendix~\ref{all-distr-appendix}, we present the distribution of each class of patterns on $n$-permutations for $1 \leq n \leq 7$. This appendix both justifies the lower bound of 105 equivalence classes and provides data for researchers wishing to further investigate distributions of mesh patterns of length~2.

\section{Preliminaries}\label{prelim-sec}

For a permutation $\pi = \pi_1 \pi_2 \cdots \pi_n \in S_n$, the entry $\pi_i$ is called a \emph{left-to-right maximum} (resp., \emph{left-to-right minimum}) if $\pi_i > \pi_j$ (resp., $\pi_i < \pi_j$) for all $j < i$. Similarly, $\pi_i$ is called a \emph{right-to-left maximum} (resp., \emph{right-to-left minimum}) if $\pi_i > \pi_j$ (resp., $\pi_i < \pi_j$) for all $j > i$. In particular, $\pi_1$ is both a left-to-right maximum and a left-to-right minimum, while $\pi_n$ is both a right-to-left maximum and a right-to-left minimum. Let
\[
F(x) := \sum_{n \ge 0} n! \, x^n.
\]
This function appears frequently in our enumerative results. Throughout the paper, we use $A(x)$ (resp., $F(x,q)$) to denote the ordinary generating function for permutations avoiding the pattern in question (resp., the generating function for the distribution of occurrences of the pattern). Hence, if $p(\pi)$ denotes the number of occurrences of a pattern $p$ in a permutation $\pi$, then
$$F(x,q)=\sum_{n\geq 0}x^n\sum_{\pi\in S_n}q^{p(\pi)},$$
and $A(x)=F(x,0)$. We abbreviate ``generating function'' as g.f.

We establish many equidistribution results using three fundamental symmetry operations on mesh patterns: \emph{complement} (vertical reflection), \emph{reverse} (horizontal reflection), and \emph{inverse} (reflection in the line $y=x$). For a mesh pattern $p=(\tau,R)$ of length $m$, define
\[
\begin{aligned}
p^c &= (\tau^c, R^c), 
& \tau^c(i) &= m+1-\tau(i), 
& R^c &= \{(x,\,m-y) : (x,y)\in R\},\\[2mm]
p^r &= (\tau^r, R^r), 
& \tau^r(i) &= m+1-\tau(m+1-i), 
& R^r &= \{(m-x,\,y) : (x,y)\in R\},\\[2mm]
p^{-1} &= (\tau^{-1}, R^{-1}), 
& \tau^{-1}(i) &= j \text{ where } \tau(j)=i, 
& R^{-1} &= \{(y,\,x) : (x,y)\in R\}.
\end{aligned}
\]

It is easy to see that if a pattern $p_1$ is obtained from a pattern $p_2$ by applying any of these three operations, then $p_1 \sim p_2$. In Tables~\ref{tab-known}, \ref{tab-easy}, \ref{tab-no-easy}, \ref{tab-conject}, and \ref{tab-remaining}, equivalences arising from these operations are indicated by placing the corresponding patterns in the same cell. We refer to such equivalences as \emph{trivial} and do not discuss them further.

\begin{table}[!t]
    \renewcommand{\arraystretch}{1.4}
    \begin{center} 
       \begin{tabular}{|c|c|c|c||c|c|c|c|}
    \hline  
    \multicolumn{1}{|c|}{\footnotesize{nr.}} & 
    \multicolumn{1}{c|}{\footnotesize{trivial}} & 
    \multicolumn{1}{c|}{\footnotesize{trivial}} & 
    \multicolumn{1}{c||}{\footnotesize{ref.}} & 
    \multicolumn{1}{c|}{\footnotesize{nr.}} & 
    \multicolumn{1}{c|}{\footnotesize{trivial}} & 
    \multicolumn{1}{c|}{\footnotesize{trivial}} & 
    \multicolumn{1}{c|}{\footnotesize{ref.}} \\
    \hline

            \footnotesize{7} 
            & \hspace{-4mm} $\pattern{scale=0.5}{2}{1/1,2/2}{0/2,2/2,2/1} 
            \pattern{scale=0.5}{2}{1/1,2/2}{0/2,0/0,1/0} 
            \pattern{scale=0.5}{2}{1/1,2/2}{1/2,2/2,2/0} 
            \pattern{scale=0.5}{2}{1/1,2/2}{0/1,0/0,2/0}$
            &  \hspace{-4mm} $\pattern{scale=0.5}{2}{1/1,2/2}{1/2,2/2,2/1} 
            \pattern{scale=0.5}{2}{1/1,2/2}{0/1,0/0,1/0}$
            & \cite{HanZeng2021}
            & \footnotesize{61} 
            & \hspace{-4mm}  $\pattern{scale=0.5}{2}{1/1,2/2}{0/2,1/2,0/1,1/1,2/1} 
            \pattern{scale=0.5}{2}{1/1,2/2}{0/2,1/2,0/1,1/1,1/0} 
            \pattern{scale=0.5}{2}{1/1,2/2}{1/2,2/1,1/1,1/0,2/0} 
            \pattern{scale=0.5}{2}{1/1,2/2}{0/1,1/1,2/1,1/0,2/0}$
            & \hspace{-4mm} $\pattern{scale=0.5}{2}{1/1,2/2}{1/2,0/1,1/1,2/1,1/0}$
            & \cite{KZ2019} \\
            \hline

            \multirow{2}{*}{\footnotesize{47}} 
            & \hspace{-4mm}  $\pattern{scale=0.5}{2}{1/1,2/2}{1/2,2/2,0/1,1/1,0/0} 
            \pattern{scale=0.5}{2}{1/1,2/2}{2/2,1/1,2/1,0/0,1/0}$
            & \hspace{-4mm}  $\pattern{scale=0.5}{2}{1/1,2/2}{1/2,2/2,0/1,2/1,0/0} 
            \pattern{scale=0.5}{2}{1/1,2/2}{1/2,2/2,0/1,0/0,1/0} 
            \pattern{scale=0.5}{2}{1/1,2/2}{1/2,2/2,2/1,0/0,1/0} 
            \pattern{scale=0.5}{2}{1/1,2/2}{2/2,0/1,2/1,0/0,1/0}$
            &\cite{HanZeng2021}
            & \footnotesize{67} 
            & \hspace{-4mm}  $\pattern{scale=0.5}{2}{1/1,2/2}{1/2,2/2,0/1,2/1} 
            \pattern{scale=0.5}{2}{1/1,2/2}{1/2,2/2,2/1,1/0} 
            \pattern{scale=0.5}{2}{1/1,2/2}{1/2,0/1,0/0,1/0} 
            \pattern{scale=0.5}{2}{1/1,2/2}{0/1,2/1,0/0,1/0}$
            & \hspace{-4mm}  $\pattern{scale=0.5}{2}{1/1,2/2}{0/2,2/2,2/1,1/1} 
            \pattern{scale=0.5}{2}{1/1,2/2}{0/2,1/1,0/0,1/0} 
            \pattern{scale=0.5}{2}{1/1,2/2}{1/2,2/2,1/1,2/0} 
            \pattern{scale=0.5}{2}{1/1,2/2}{0/1,1/1,0/0,2/0}$
            & \cite{HanZeng2021} \\
            \cline{5-8}

            & \hspace{-4mm}  $\pattern{scale=0.5}{2}{1/1,2/2}{0/2,2/2,1/1,2/1,1/0} 
            \pattern{scale=0.5}{2}{1/1,2/2}{0/2,1/1,2/1,0/0,1/0} 
            \pattern{scale=0.5}{2}{1/1,2/2}{1/2,2/2,0/1,1/1,2/0} 
            \pattern{scale=0.5}{2}{1/1,2/2}{1/2,0/1,1/1,0/0,2/0}$
            & 
            & \cite{KZ2019}
            & \multirow{2}{*}{\footnotesize{68}}
            & \hspace{-4mm}  $\pattern{scale=0.5}{2}{1/1,2/2}{0/2,1/2,0/1,1/1} 
            \pattern{scale=0.5}{2}{1/1,2/2}{1/1,2/1,1/0,2/0}$
            & \hspace{-4mm}  $\pattern{scale=0.5}{2}{1/1,2/2}{1/2,2/2,1/1,2/1} 
            \pattern{scale=0.5}{2}{1/1,2/2}{0/1,1/1,0/0,1/0}$
            & \multirow{2}{*}{\cite{KZ2019}} \\
            \cline{1-4}

            \footnotesize{48} 
            & \hspace{-4mm}  $\pattern{scale=0.5}{2}{1/1,2/2}{0/2,2/2,0/1,1/1,2/1} 
            \pattern{scale=0.5}{2}{1/1,2/2}{0/2,1/2,1/1,0/0,1/0} 
            \pattern{scale=0.5}{2}{1/1,2/2}{0/1,1/1,2/1,0/0,2/0} 
            \pattern{scale=0.5}{2}{1/1,2/2}{1/2,2/2,1/1,1/0,2/0}$
            & \hspace{-4mm}  $\pattern{scale=0.5}{2}{1/1,2/2}{1/2,2/2,0/1,1/1,2/1} 
            \pattern{scale=0.5}{2}{1/1,2/2}{1/2,2/2,1/1,2/1,1/0} 
            \pattern{scale=0.5}{2}{1/1,2/2}{1/2,0/1,1/1,0/0,1/0} 
            \pattern{scale=0.5}{2}{1/1,2/2}{0/1,1/1,2/1,0/0,1/0}$
            & \cite{HanZeng2021}
            & 
            & \hspace{-4mm}  $\pattern{scale=0.5}{2}{1/1,2/2}{1/2,0/1,1/1,2/1} 
            \pattern{scale=0.5}{2}{1/1,2/2}{1/2,0/1,1/1,1/0} 
            \pattern{scale=0.5}{2}{1/1,2/2}{1/2,1/1,2/1,1/0} 
            \pattern{scale=0.5}{2}{1/1,2/2}{0/1,1/1,2/1,1/0}$
            & 
            & \\
            \hline

            \multirow{2}{*}{\footnotesize{56}} 
            & \hspace{-4mm}  $\pattern{scale=0.5}{2}{1/1,2/2}{1/2,2/2,1/1,2/1,0/0} 
            \pattern{scale=0.5}{2}{1/1,2/2}{2/2,0/1,1/1,0/0,1/0}$
            & \hspace{-4mm}  $\pattern{scale=0.5}{2}{1/1,2/2}{0/2,2/2,0/1,2/1,1/0} 
            \pattern{scale=0.5}{2}{1/1,2/2}{0/2,1/2,2/1,0/0,1/0} 
            \pattern{scale=0.5}{2}{1/1,2/2}{1/2,0/1,2/1,0/0,2/0} 
            \pattern{scale=0.5}{2}{1/1,2/2}{1/2,2/2,0/1,1/0,2/0}$
            & \multirow{1}{*}{\cite{KZ2019}}
            & \footnotesize{70} 
            & \hspace{-4mm}  $\pattern{scale=0.5}{2}{1/1,2/2}{1/2,2/2,1/1,0/0} 
            \pattern{scale=0.5}{2}{1/1,2/2}{2/2,0/1,1/1,0/0} 
            \pattern{scale=0.5}{2}{1/1,2/2}{2/2,1/1,2/1,0/0} 
            \pattern{scale=0.5}{2}{1/1,2/2}{2/2,1/1,0/0,1/0}$
            & \hspace{-4mm}  $\pattern{scale=0.5}{2}{1/1,2/2}{1/2,0/1,2/1,0/0} 
            \pattern{scale=0.5}{2}{1/1,2/2}{1/2,2/2,0/1,1/0} 
            \pattern{scale=0.5}{2}{1/1,2/2}{2/2,0/1,2/1,1/0} 
            \pattern{scale=0.5}{2}{1/1,2/2}{1/2,2/1,0/0,1/0}$
            & \cite{HanZeng2021} \\
            \cline{5-8}

            & \hspace{-4mm}  $\pattern{scale=0.5}{2}{1/1,2/2}{0/2,1/2,0/1,1/1,2/0} 
            \pattern{scale=0.5}{2}{1/1,2/2}{0/2,1/1,2/1,1/0,2/0}$ 
            & 
            & \cite{LK2025}
            & 
            & 
            & 
            & \\
            \hline

        \end{tabular}
    \end{center} 
    \vspace{-0.5cm}
    \caption{Known equidistributions.}
    \label{tab-known}
\end{table}

Throughout the paper, we frequently use the insertion of elements into an $(n-1)$-permutation $\pi' = \pi'_1 \cdots \pi'_{n-1}$ to obtain an $n$-permutation $\pi = \pi_1 \cdots \pi_n$. The insertion of the new largest element is straightforward. Inserting the new smallest element amounts to replacing each $\pi'_i$ by $\pi'_i + 1$ for $1 \le i \le n-1$. More generally, when inserting a new element $x$ either on the left or on the right, we increase by $1$ all entries of $\pi'$ that are greater than or equal to $x$.

\begin{table}[!t]
\footnotesize
    \renewcommand{\arraystretch}{1.5}
    \begin{center} 
        \begin{tabular}{|c|l|l|l|l|c|}
            \hline  
                nr. & \hfil trivial & \hfil trivial & \hfil trivial & \hfil trivial & ref. \\
            \hline  
            
            \footnotesize{1} 
            & $\pattern{scale=0.5}{2}{1/1,2/2}{0/2,1/2,2/2,0/1,2/1,0/0,1/0} 
            \pattern{scale=0.5}{2}{1/1,2/2}{1/2,2/2,0/1,2/1,0/0,1/0,2/0}$ 
            & $\pattern{scale=0.5}{2}{1/1,2/2}{0/2,1/2,2/2,1/1,2/1,0/0,1/0} 
            \pattern{scale=0.5}{2}{1/1,2/2}{0/2,2/2,0/1,1/1,2/1,0/0,1/0} 
            \pattern{scale=0.5}{2}{1/1,2/2}{1/2,2/2,0/1,1/1,0/0,2/1,2/0} 
            \pattern{scale=0.5}{2}{1/1,2/2}{1/2,2/2,0/1,1/1,0/0,1/0,2/0}$
            & $\pattern{scale=0.5}{2}{1/1,2/2}{0/2,1/2,2/2,1/1,0/1,0/0,2/0} 
            \pattern{scale=0.5}{2}{1/1,2/2}{0/2,2/2,1/1,2/1,0/0,1/0,2/0}$
            & $\pattern{scale=0.5}{2}{1/1,2/2}{0/2,1/2,0/1,1/1,2/1,0/0,2/0} 
            \pattern{scale=0.5}{2}{1/1,2/2}{0/2,1/2,1/1,0/1,2/2,1/0,2/0} 
            \pattern{scale=0.5}{2}{1/1,2/2}{0/2,2/2,0/1,1/1,2/1,1/0,2/0} 
            \pattern{scale=0.5}{2}{1/1,2/2}{0/2,1/2,1/1,2/1,0/0,1/0,2/0}$
            & Thm~\ref{thm:class1} \\[0.5mm]
            \hline

            \multirow{2}{*}{\footnotesize{2}} 
            & $\pattern{scale=0.5}{2}{1/1,2/2}{0/2,1/2,2/2,0/1,1/1,2/1,0/0} 
            \pattern{scale=0.5}{2}{1/1,2/2}{0/2,1/2,2/2,0/1,1/1,0/0,1/0} 
            \pattern{scale=0.5}{2}{1/1,2/2}{1/2,2/2,1/1,2/1,0/0,1/0,2/0} 
            \pattern{scale=0.5}{2}{1/1,2/2}{2/2,0/1,1/1,2/1,0/0,1/0,2/0}$
            & $\pattern{scale=0.5}{2}{1/1,2/2}{0/2,1/2,2/2,0/1,1/1,2/1,1/0} 
            \pattern{scale=0.5}{2}{1/1,2/2}{0/2,1/2,0/1,1/1,2/1,0/0,1/0} 
            \pattern{scale=0.5}{2}{1/1,2/2}{1/2,2/2,0/1,1/1,2/1,1/0,2/0} 
            \pattern{scale=0.5}{2}{1/1,2/2}{1/2,0/1,1/1,2/1,0/0,1/0,2/0}$
            & $\pattern{scale=0.5}{2}{1/1,2/2}{0/2,1/2,2/2,0/1,1/1,2/1,2/0} 
            \pattern{scale=0.5}{2}{1/1,2/2}{0/2,1/2,2/2,1/1,2/1,1/0,2/0} 
            \pattern{scale=0.5}{2}{1/1,2/2}{0/2,1/2,0/1,1/1,0/0,1/0,2/0} 
            \pattern{scale=0.5}{2}{1/1,2/2}{0/2,0/1,1/1,2/1,0/0,1/0,2/0}$
            & $\pattern{scale=0.5}{2}{1/1,2/2}{0/2,1/2,2/2,0/1,2/1,0/0,2/0} 
            \pattern{scale=0.5}{2}{1/1,2/2}{0/2,1/2,2/2,0/1,0/0,1/0,2/0} 
            \pattern{scale=0.5}{2}{1/1,2/2}{0/2,1/2,2/2,2/1,0/0,1/0,2/0} 
            \pattern{scale=0.5}{2}{1/1,2/2}{0/2,2/2,0/1,2/1,0/0,1/0,2/0}$
            & \multirow{2}{*}{Thm~\ref{thm:class2}} \\[0.5mm]
            \cline{2-2} 
            & $\pattern{scale=0.5}{2}{1/1,2/2}{0/2,2/2,0/1,1/1,2/1,0/0,2/0} 
            \pattern{scale=0.5}{2}{1/1,2/2}{0/2,1/2,2/2,1/1,0/0,1/0,2/0}$
            & & & & \\[0.5mm]
            \hline
            
            \footnotesize{25} 
            & $\pattern{scale=0.5}{2}{1/1,2/2}{0/2,1/2,2/2,2/1} 
            \pattern{scale=0.5}{2}{1/1,2/2}{0/2,0/1,0/0,1/0} 
            \pattern{scale=0.5}{2}{1/1,2/2}{1/2,2/2,2/1,2/0} 
            \pattern{scale=0.5}{2}{1/1,2/2}{0/1,0/0,1/0,2/0}$
            & $\pattern{scale=0.5}{2}{1/1,2/2}{0/2,1/2,2/2,2/0} 
            \pattern{scale=0.5}{2}{1/1,2/2}{0/2,2/2,2/1,2/0} 
            \pattern{scale=0.5}{2}{1/1,2/2}{0/2,0/1,0/0,2/0} 
            \pattern{scale=0.5}{2}{1/1,2/2}{0/2,0/0,1/0,2/0}$
            & & & Thm~\ref{thm:class25} \\[0.5mm]
            \hline

            \footnotesize{31} 
            & $\pattern{scale=0.5}{2}{1/1,2/2}{0/2,1/2,2/2,2/1,0/0,1/0} 
            \pattern{scale=0.5}{2}{1/1,2/2}{0/2,2/2,0/1,2/1,0/0,1/0} 
            \pattern{scale=0.5}{2}{1/1,2/2}{1/2,2/2,0/1,2/1,0/0,2/0} 
            \pattern{scale=0.5}{2}{1/1,2/2}{1/2,2/2,0/1,0/0,1/0,2/0}$
            & $\pattern{scale=0.5}{2}{1/1,2/2}{0/2,1/2,2/2,0/1,1/1,2/0} 
            \pattern{scale=0.5}{2}{1/1,2/2}{0/2,1/2,0/1,1/1,0/0,2/0} 
            \pattern{scale=0.5}{2}{1/1,2/2}{0/2,2/2,1/1,2/1,1/0,2/0} 
            \pattern{scale=0.5}{2}{1/1,2/2}{0/2,1/1,2/1,0/0,1/0,2/0}$
            & & & Thm~\ref{thm:class31} \\[0.5mm]
            \hline

            \multirow{2}{*}{\footnotesize{32}} 
            & $\pattern{scale=0.5}{2}{1/1,2/2}{0/2,1/2,2/2,0/1,1/1,0/0} 
            \pattern{scale=0.5}{2}{1/1,2/2}{2/2,1/1,2/1,0/0,1/0,2/0}$
            & $\pattern{scale=0.5}{2}{1/1,2/2}{0/2,1/2,2/2,0/1,2/1,0/0} 
            \pattern{scale=0.5}{2}{1/1,2/2}{0/2,1/2,2/2,0/1,0/0,1/0} 
            \pattern{scale=0.5}{2}{1/1,2/2}{2/2,1/2,2/1,0/0,1/0,2/0} 
            \pattern{scale=0.5}{2}{1/1,2/2}{2/2,0/1,2/1,0/0,1/0,2/0}$
            & $\pattern{scale=0.5}{2}{1/1,2/2}{0/2,1/2,0/1,1/1,2/1,0/0} 
            \pattern{scale=0.5}{2}{1/1,2/2}{0/2,1/2,2/2,0/1,1/1,1/0} 
            \pattern{scale=0.5}{2}{1/1,2/2}{2/2,0/1,1/1,2/1,1/0,2/0} 
            \pattern{scale=0.5}{2}{1/1,2/2}{1/2,1/1,2/1,0/0,1/0,2/0}$
            & $\pattern{scale=0.5}{2}{1/1,2/2}{0/2,2/2,0/1,1/1,2/1,0/0} 
            \pattern{scale=0.5}{2}{1/1,2/2}{0/2,1/2,2/2,1/1,0/0,1/0} 
            \pattern{scale=0.5}{2}{1/1,2/2}{2/2,0/1,1/1,2/1,0/0,2/0} 
            \pattern{scale=0.5}{2}{1/1,2/2}{1/2,2/2,1/1,0/0,1/0,2/0}$
            & \multirow{2}{*}{Thm~\ref{thm:class32}} \\[0.5mm]
            \cline{2-4}
            & $\pattern{scale=0.5}{2}{1/1,2/2}{0/2,1/2,2/2,1/1,2/1,1/0} 
            \pattern{scale=0.5}{2}{1/1,2/2}{0/2,0/1,1/1,2/1,0/0,1/0} 
            \pattern{scale=0.5}{2}{1/1,2/2}{1/2,2/2,0/1,1/1,2/1,2/0} 
            \pattern{scale=0.5}{2}{1/1,2/2}{1/2,0/1,1/1,0/0,1/0,2/0}$
            & $\pattern{scale=0.5}{2}{1/1,2/2}{0/2,2/2,0/1,1/1,2/1,2/0} 
            \pattern{scale=0.5}{2}{1/1,2/2}{0/2,0/1,1/1,2/1,0/0,2/0} 
            \pattern{scale=0.5}{2}{1/1,2/2}{0/2,1/2,2/2,1/1,1/0,2/0} 
            \pattern{scale=0.5}{2}{1/1,2/2}{0/2,1/2,1/1,0/0,1/0,2/0}$
            & $\pattern{scale=0.5}{2}{1/1,2/2}{0/2,1/2,2/2,0/1,0/0,2/0} 
            \pattern{scale=0.5}{2}{1/1,2/2}{0/2,2/2,2/1,0/0,1/0,2/0}$
            & & \\[0.5mm]
            \hline

            \footnotesize{33} 
            & $\pattern{scale=0.5}{2}{1/1,2/2}{0/2,1/2,2/2,0/1,1/1,2/1} 
            \pattern{scale=0.5}{2}{1/1,2/2}{0/2,1/2,0/1,1/1,0/0,1/0} 
            \pattern{scale=0.5}{2}{1/1,2/2}{1/2,2/2,1/1,2/1,1/0,2/0} 
            \pattern{scale=0.5}{2}{1/1,2/2}{0/1,1/1,2/1,0/0,1/0,2/0}$
            & $\pattern{scale=0.5}{2}{1/1,2/2}{0/2,2/2,0/1,2/1,0/0,2/0} 
            \pattern{scale=0.5}{2}{1/1,2/2}{0/2,1/2,2/2,0/0,1/0,2/0}$
            & & & Thm~\ref{thm:class33} \\[0.5mm]
            \hline

            \footnotesize{45} 
            & $\pattern{scale=0.5}{2}{1/1,2/2}{0/2,1/2,2/2,0/1,1/1} 
            \pattern{scale=0.5}{2}{1/1,2/2}{0/2,1/2,0/1,1/1,0/0} 
            \pattern{scale=0.5}{2}{1/1,2/2}{2/2,1/1,2/1,1/0,2/0} 
            \pattern{scale=0.5}{2}{1/1,2/2}{1/1,2/1,0/0,1/0,2/0}$
            & $\pattern{scale=0.5}{2}{1/1,2/2}{0/2,1/2,2/2,0/1,0/0} 
            \pattern{scale=0.5}{2}{1/1,2/2}{2/2,2/1,0/0,1/0,2/0}$
            & $\pattern{scale=0.5}{2}{1/1,2/2}{0/2,2/2,0/1,2/1,0/0} 
            \pattern{scale=0.5}{2}{1/1,2/2}{0/2,1/2,2/2,0/0,1/0} 
            \pattern{scale=0.5}{2}{1/1,2/2}{2/2,0/1,2/1,0/0,2/0} 
            \pattern{scale=0.5}{2}{1/1,2/2}{1/2,2/2,0/0,1/0,2/0}$
            & $\pattern{scale=0.5}{2}{1/1,2/2}{0/2,0/1,1/1,2/1,0/0} 
            \pattern{scale=0.5}{2}{1/1,2/2}{0/2,1/2,2/2,1/1,1/0} 
            \pattern{scale=0.5}{2}{1/1,2/2}{2/2,0/1,1/1,2/1,2/0} 
            \pattern{scale=0.5}{2}{1/1,2/2}{1/2,1/1,0/0,1/0,2/0}$
            & Thm~\ref{thm:class45} \\[0.5mm]
            \hline

            \footnotesize{53} 
            & $\pattern{scale=0.5}{2}{1/1,2/2}{0/2,1/2,2/2,0/1,1/1,2/1,0/0,1/0} 
            \pattern{scale=0.5}{2}{1/1,2/2}{1/2,2/2,0/1,1/1,2/1,0/0,1/0,2/0}$
            & $\pattern{scale=0.5}{2}{1/1,2/2}{0/2,1/2,2/2,0/1,1/1,2/1,0/0,2/0} 
            \pattern{scale=0.5}{2}{1/1,2/2}{0/2,1/2,2/2,0/1,1/1,0/0,1/0,2/0} 
            \pattern{scale=0.5}{2}{1/1,2/2}{0/2,1/2,2/2,2/1,0/0,1/0,2/0,1/1} 
            \pattern{scale=0.5}{2}{1/1,2/2}{0/2,2/2,0/1,2/1,0/0,1/0,2/0,1/1}$
            & $\pattern{scale=0.5}{2}{1/1,2/2}{0/2,1/2,2/2,0/1,1/1,2/1,1/0,2/0} 
            \pattern{scale=0.5}{2}{1/1,2/2}{0/2,1/2,2/2,0/1,1/1,2/1,0/0,1/0}$
            & $\pattern{scale=0.5}{2}{1/1,2/2}{0/2,2/2,0/1,2/1,0/0,1/0,2/0,1/2}$
            & Thm~\ref{thm:class53} \\[0.5mm]
            \hline

            \footnotesize{60} 
            & $\pattern{scale=0.5}{2}{1/1,2/2}{0/2,1/2,2/2,1/1,0/0} 
            \pattern{scale=0.5}{2}{1/1,2/2}{0/2,2/2,0/1,1/1,0/0} 
            \pattern{scale=0.5}{2}{1/1,2/2}{2/2,1/1,2/1,0/0,2/0} 
            \pattern{scale=0.5}{2}{1/1,2/2}{2/2,1/1,0/0,1/0,2/0}$
            & $\pattern{scale=0.5}{2}{1/1,2/2}{0/2,1/2,0/1,2/1,0/0} 
            \pattern{scale=0.5}{2}{1/1,2/2}{0/2,1/2,2/2,0/1,1/0} 
            \pattern{scale=0.5}{2}{1/1,2/2}{2/2,0/1,2/1,1/0,2/0} 
            \pattern{scale=0.5}{2}{1/1,2/2}{1/2,2/1,0/0,1/0,2/0}$
            & & & Thm~\ref{thm:class60} \\[0.5mm]
            \hline

    \footnotesize{62} 
    & $\pattern{scale=0.5}{2}{1/1,2/2}{0/2,1/2,2/2,2/1,0/0} 
    \pattern{scale=0.5}{2}{1/1,2/2}{0/2,2/2,0/1,0/0,1/0} 
    \pattern{scale=0.5}{2}{1/1,2/2}{1/2,2/2,2/1,0/0,2/0} 
    \pattern{scale=0.5}{2}{1/1,2/2}{2/2,0/1,0/0,1/0,2/0}$
    & $\pattern{scale=0.5}{2}{1/1,2/2}{0/2,1/2,2/2,2/1,1/0} 
    \pattern{scale=0.5}{2}{1/1,2/2}{0/2,0/1,2/1,0/0,1/0} 
    \pattern{scale=0.5}{2}{1/1,2/2}{1/2,2/2,0/1,2/1,2/0} 
    \pattern{scale=0.5}{2}{1/1,2/2}{1/2,0/1,0/0,1/0,2/0}$
    & $\pattern{scale=0.5}{2}{1/1,2/2}{0/2,1/2,2/2,0/1,2/0} 
    \pattern{scale=0.5}{2}{1/1,2/2}{0/2,1/2,0/1,0/0,2/0} 
    \pattern{scale=0.5}{2}{1/1,2/2}{0/2,2/2,2/1,1/0,2/0} 
    \pattern{scale=0.5}{2}{1/1,2/2}{0/2,2/1,0/0,1/0,2/0}$
    & $\pattern{scale=0.5}{2}{1/1,2/2}{0/2,1/2,2/2,1/1,2/0} 
    \pattern{scale=0.5}{2}{1/1,2/2}{0/2,2/2,1/1,2/1,2/0} 
    \pattern{scale=0.5}{2}{1/1,2/2}{0/2,0/1,1/1,0/0,2/0} 
    \pattern{scale=0.5}{2}{1/1,2/2}{0/2,1/1,0/0,1/0,2/0}$
    & Thm~\ref{thm:class62} \\[0.5mm]
    \hline

    \footnotesize{63} 
    & $\pattern{scale=0.5}{2}{1/1,2/2}{0/2,1/2,2/2,0/1,2/1} 
    \pattern{scale=0.5}{2}{1/1,2/2}{0/2,1/2,0/1,0/0,1/0} 
    \pattern{scale=0.5}{2}{1/1,2/2}{1/2,2/2,2/1,1/0,2/0} 
    \pattern{scale=0.5}{2}{1/1,2/2}{0/1,2/1,0/0,1/0,2/0}$
    & $\pattern{scale=0.5}{2}{1/1,2/2}{0/2,1/2,2/2,1/1,2/1} 
    \pattern{scale=0.5}{2}{1/1,2/2}{0/2,0/1,1/1,0/0,1/0} 
    \pattern{scale=0.5}{2}{1/1,2/2}{1/2,2/2,1/1,2/1,2/0} 
    \pattern{scale=0.5}{2}{1/1,2/2}{0/1,1/1,0/0,1/0,2/0}$
    & $\pattern{scale=0.5}{2}{1/1,2/2}{0/2,2/2,0/1,2/1,2/0} 
    \pattern{scale=0.5}{2}{1/1,2/2}{0/2,0/1,2/1,0/0,2/0} 
    \pattern{scale=0.5}{2}{1/1,2/2}{0/2,1/2,2/2,1/0,2/0} 
    \pattern{scale=0.5}{2}{1/1,2/2}{0/2,1/2,0/0,1/0,2/0}$
    & $\pattern{scale=0.5}{2}{1/1,2/2}{0/2,1/2,2/2,0/0,2/0} 
    \pattern{scale=0.5}{2}{1/1,2/2}{0/2,2/2,0/1,0/0,2/0} 
    \pattern{scale=0.5}{2}{1/1,2/2}{0/2,2/2,2/1,0/0,2/0} 
    \pattern{scale=0.5}{2}{1/1,2/2}{0/2,2/2,0/0,1/0,2/0}$
    & Thm~\ref{thm:class63} \\[0.5mm]
    \hline

            \multirow{3}{*}{\footnotesize{68}}
            & $\pattern{scale=0.5}{2}{1/1,2/2}{0/2,1/2,2/2,0/1} 
            \pattern{scale=0.5}{2}{1/1,2/2}{0/2,1/2,0/1,0/0} 
            \pattern{scale=0.5}{2}{1/1,2/2}{2/2,2/1,1/0,2/0} 
            \pattern{scale=0.5}{2}{1/1,2/2}{2/1,0/0,1/0,2/0}$
            & $\pattern{scale=0.5}{2}{1/1,2/2}{0/2,1/2,2/2,1/1} 
            \pattern{scale=0.5}{2}{1/1,2/2}{0/2,0/1,1/1,0/0} 
            \pattern{scale=0.5}{2}{1/1,2/2}{2/2,1/1,2/1,2/0} 
            \pattern{scale=0.5}{2}{1/1,2/2}{1/1,0/0,1/0,2/0}$
            & $\pattern{scale=0.5}{2}{1/1,2/2}{0/2,1/2,0/1,1/1} 
            \pattern{scale=0.5}{2}{1/1,2/2}{1/1,2/1,1/0,2/0}$
            & $\pattern{scale=0.5}{2}{1/1,2/2}{0/2,2/2,0/1,2/1} 
            \pattern{scale=0.5}{2}{1/1,2/2}{0/2,1/2,0/0,1/0} 
            \pattern{scale=0.5}{2}{1/1,2/2}{0/1,2/1,0/0,2/0} 
            \pattern{scale=0.5}{2}{1/1,2/2}{1/2,2/2,1/0,2/0}$
            & \multirow{3}{*}{Thm~\ref{thm:class68}} \\[0.5mm]
            \cline{2-5} 
            & $\pattern{scale=0.5}{2}{1/1,2/2}{1/2,2/2,1/1,2/1} 
            \pattern{scale=0.5}{2}{1/1,2/2}{0/1,1/1,0/0,1/0}$
            & $\pattern{scale=0.5}{2}{1/1,2/2}{0/2,0/1,1/1,2/1} 
            \pattern{scale=0.5}{2}{1/1,2/2}{0/2,1/2,1/1,1/0} 
            \pattern{scale=0.5}{2}{1/1,2/2}{0/1,1/1,2/1,2/0} 
            \pattern{scale=0.5}{2}{1/1,2/2}{1/2,1/1,1/0,2/0}$
            & $\pattern{scale=0.5}{2}{1/1,2/2}{1/2,0/1,1/1,2/1} 
            \pattern{scale=0.5}{2}{1/1,2/2}{1/2,0/1,1/1,1/0} 
            \pattern{scale=0.5}{2}{1/1,2/2}{1/2,1/1,2/1,1/0} 
            \pattern{scale=0.5}{2}{1/1,2/2}{0/1,1/1,2/1,1/0}$
            & $\pattern{scale=0.5}{2}{1/1,2/2}{0/2,1/2,2/2,0/0} 
            \pattern{scale=0.5}{2}{1/1,2/2}{0/2,2/2,0/1,0/0} 
            \pattern{scale=0.5}{2}{1/1,2/2}{2/2,2/1,0/0,2/0} 
            \pattern{scale=0.5}{2}{1/1,2/2}{2/2,0/0,1/0,2/0}$
            & \\[0.5mm]
            \cline{2-2} 
            & $\pattern{scale=0.5}{2}{1/1,2/2}{0/2,0/1,2/1,0/0} 
            \pattern{scale=0.5}{2}{1/1,2/2}{0/2,1/2,2/2,1/0} 
            \pattern{scale=0.5}{2}{1/1,2/2}{2/2,0/1,2/1,2/0} 
            \pattern{scale=0.5}{2}{1/1,2/2}{1/2,0/0,1/0,2/0}$
            & & & & \\[0.5mm]
            \hline

            \footnotesize{77} 
            & $\pattern{scale=0.5}{2}{1/1,2/2}{0/2,1/2,2/2,1/1,2/1,0/0,2/0} 
            \pattern{scale=0.5}{2}{1/1,2/2}{0/2,2/2,0/1,1/1,0/0,1/0,2/0}$
            & $\pattern{scale=0.5}{2}{1/1,2/2}{0/2,1/2,2/2,0/1,2/1,1/0,2/0} 
            \pattern{scale=0.5}{2}{1/1,2/2}{0/2,1/2,0/1,2/1,0/0,1/0,2/0}$
            & & & Thm~\ref{thm:class77} \\[0.5mm]
            \hline

            \footnotesize{78} 
            & $\pattern{scale=0.5}{2}{1/1,2/2}{0/2,1/2,2/2,0/1,2/1,2/0} 
            \pattern{scale=0.5}{2}{1/1,2/2}{0/2,1/2,2/2,2/1,1/0,2/0} 
            \pattern{scale=0.5}{2}{1/1,2/2}{0/2,1/2,0/1,0/0,1/0,2/0} 
            \pattern{scale=0.5}{2}{1/1,2/2}{0/2,0/1,2/1,0/0,1/0,2/0}$
            & $\pattern{scale=0.5}{2}{1/1,2/2}{0/2,1/2,2/2,1/1,2/1,2/0} 
            \pattern{scale=0.5}{2}{1/1,2/2}{0/2,0/1,1/1,0/0,1/0,2/0}$
            & $\pattern{scale=0.5}{2}{1/1,2/2}{0/2,1/2,2/2,2/1,0/0,2/0} 
            \pattern{scale=0.5}{2}{1/1,2/2}{0/2,2/2,0/1,0/0,1/0,2/0}$
            & & Thm~\ref{thm:class78} \\[0.5mm]
            \hline

            \multirow{2}{*}{\footnotesize{80}}
            & $\pattern{scale=0.5}{2}{1/1,2/2}{0/2,1/2,2/2,1/1,2/1,0/0} 
            \pattern{scale=0.5}{2}{1/1,2/2}{0/2,2/2,0/1,1/1,0/0,1/0} 
            \pattern{scale=0.5}{2}{1/1,2/2}{1/2,2/2,1/1,2/1,0/0,2/0} 
            \pattern{scale=0.5}{2}{1/1,2/2}{2/2,0/1,1/1,0/0,1/0,2/0}$
            & $\pattern{scale=0.5}{2}{1/1,2/2}{0/2,1/2,2/2,0/1,2/1,1/0} 
            \pattern{scale=0.5}{2}{1/1,2/2}{0/2,1/2,0/1,2/1,0/0,1/0} 
            \pattern{scale=0.5}{2}{1/1,2/2}{1/2,2/2,0/1,2/1,1/0,2/0} 
            \pattern{scale=0.5}{2}{1/1,2/2}{1/2,0/1,2/1,0/0,1/0,2/0}$
            & $\pattern{scale=0.5}{2}{1/1,2/2}{0/2,1/2,0/1,1/1,2/1,2/0} 
            \pattern{scale=0.5}{2}{1/1,2/2}{0/2,1/2,0/1,1/1,1/0,2/0} 
            \pattern{scale=0.5}{2}{1/1,2/2}{0/2,1/2,1/1,2/1,1/0,2/0} 
            \pattern{scale=0.5}{2}{1/1,2/2}{0/2,0/1,1/1,2/1,1/0,2/0}$
            & $\pattern{scale=0.5}{2}{1/1,2/2}{0/2,1/2,2/2,1/1,0/0,2/0} 
            \pattern{scale=0.5}{2}{1/1,2/2}{0/2,2/2,0/1,1/1,0/0,2/0} 
            \pattern{scale=0.5}{2}{1/1,2/2}{0/2,2/2,1/1,2/1,0/0,2/0} 
            \pattern{scale=0.5}{2}{1/1,2/2}{0/2,2/2,1/1,0/0,1/0,2/0}$
            & \multirow{2}{*}{Thm~\ref{thm:class80}} \\[0.5mm]
            \cline{2-2} 
            & $\pattern{scale=0.5}{2}{1/1,2/2}{0/2,1/2,0/1,2/1,0/0,2/0} 
            \pattern{scale=0.5}{2}{1/1,2/2}{0/2,1/2,2/2,0/1,1/0,2/0} 
            \pattern{scale=0.5}{2}{1/1,2/2}{0/2,2/2,0/1,2/1,1/0,2/0} 
            \pattern{scale=0.5}{2}{1/1,2/2}{0/2,1/2,2/1,0/0,1/0,2/0}$
            & & & & \\[0.5mm]
            \hline

        \end{tabular}
    \end{center} 
    \vspace{-0.5cm}
    \caption{Easy explainable equidistributions.}
    \label{tab-easy}
\end{table}

\begin{table}[!t]
\footnotesize
    \renewcommand{\arraystretch}{1.5}
    \begin{center} 
        \begin{tabular}{|c|l|l|l|c|}
            \hline  
                nr. & \hfil trivial & \hfil trivial & \hfil trivial & ref. \\
            \hline  

            \footnotesize{15} 
            & $\pattern{scale=0.5}{2}{1/1,2/2}{1/2,2/2} 
            \pattern{scale=0.5}{2}{1/1,2/2}{2/2,2/1} 
            \pattern{scale=0.5}{2}{1/1,2/2}{0/1,0/0} 
            \pattern{scale=0.5}{2}{1/1,2/2}{0/0,1/0}$
            & $\pattern{scale=0.5}{2}{1/1,2/2}{1/2,1/1} 
            \pattern{scale=0.5}{2}{1/1,2/2}{0/1,1/1} 
            \pattern{scale=0.5}{2}{1/1,2/2}{1/1,2/1} 
            \pattern{scale=0.5}{2}{1/1,2/2}{1/1,1/0}$
            & & Thm~\ref{thm:class15} \\[0.5mm]
            \hline

            \footnotesize{36} 
            & $\pattern{scale=0.5}{2}{1/1,2/2}{1/2,2/2,0/1,1/1,2/1,0/0} 
            \pattern{scale=0.5}{2}{1/1,2/2}{1/2,2/2,0/1,1/1,0/0,1/0} 
            \pattern{scale=0.5}{2}{1/1,2/2}{1/2,2/2,1/1,2/1,0/0,1/0} 
            \pattern{scale=0.5}{2}{1/1,2/2}{2/2,0/1,1/1,2/1,0/0,1/0}$
            & $\pattern{scale=0.5}{2}{1/1,2/2}{0/2,2/2,0/1,1/1,2/1,1/0} 
            \pattern{scale=0.5}{2}{1/1,2/2}{0/2,1/2,1/1,2/1,0/0,1/0} 
            \pattern{scale=0.5}{2}{1/1,2/2}{1/2,0/1,1/1,2/1,0/0,2/0} 
            \pattern{scale=0.5}{2}{1/1,2/2}{1/2,2/2,0/1,1/1,1/0,2/0}$
            & & Thm~\ref{thm:class36} \\[0.5mm]
            \hline

            \footnotesize{46} 
            & $\pattern{scale=0.5}{2}{1/1,2/2}{1/2,0/1,1/1,2/1,0/0} 
            \pattern{scale=0.5}{2}{1/1,2/2}{1/2,2/2,0/1,1/1,1/0} 
            \pattern{scale=0.5}{2}{1/1,2/2}{2/2,0/1,1/1,2/1,1/0} 
            \pattern{scale=0.5}{2}{1/1,2/2}{1/2,1/1,2/1,0/0,1/0}$
            & $\pattern{scale=0.5}{2}{1/1,2/2}{2/2,0/1,1/1,2/1,0/0} 
            \pattern{scale=0.5}{2}{1/1,2/2}{1/2,2/2,1/1,0/0,1/0}$
            & & Thm~\ref{thm:class46} \\[0.5mm]
            \hline

            \footnotesize{49} 
            & $\pattern{scale=0.5}{2}{1/1,2/2}{1/2,2/2,0/1,1/1} 
            \pattern{scale=0.5}{2}{1/1,2/2}{1/2,0/1,1/1,0/0} 
            \pattern{scale=0.5}{2}{1/1,2/2}{2/2,1/1,2/1,1/0} 
            \pattern{scale=0.5}{2}{1/1,2/2}{1/1,2/1,0/0,1/0}$
            & $\pattern{scale=0.5}{2}{1/1,2/2}{2/2,0/1,2/1,0/0} 
            \pattern{scale=0.5}{2}{1/1,2/2}{1/2,2/2,0/0,1/0}$
            & & Thm~\ref{thm:class49} \\[0.5mm]
            \hline

            \footnotesize{73} 
            & $\pattern{scale=0.5}{2}{1/1,2/2}{1/2,2/2,1/1} 
            \pattern{scale=0.5}{2}{1/1,2/2}{2/2,1/1,2/1} 
            \pattern{scale=0.5}{2}{1/1,2/2}{0/1,1/1,0/0} 
            \pattern{scale=0.5}{2}{1/1,2/2}{1/1,0/0,1/0}$
            & $\pattern{scale=0.5}{2}{1/1,2/2}{2/2,0/1,2/1} 
            \pattern{scale=0.5}{2}{1/1,2/2}{0/1,2/1,0/0} 
            \pattern{scale=0.5}{2}{1/1,2/2}{1/2,2/2,1/0} 
            \pattern{scale=0.5}{2}{1/1,2/2}{1/2,0/0,1/0}$
            & & Thm~\ref{thm:class73} \\[0.5mm]
            \hline

            \footnotesize{75} 
            & $\pattern{scale=0.5}{2}{1/1,2/2}{1/2,2/2,0/1} 
            \pattern{scale=0.5}{2}{1/1,2/2}{1/2,0/1,0/0} 
            \pattern{scale=0.5}{2}{1/1,2/2}{2/2,2/1,1/0} 
            \pattern{scale=0.5}{2}{1/1,2/2}{2/1,1/0,0/0}$
            & $\pattern{scale=0.5}{2}{1/1,2/2}{2/2,0/1,1/1} 
            \pattern{scale=0.5}{2}{1/1,2/2}{1/2,1/1,0/0} 
            \pattern{scale=0.5}{2}{1/1,2/2}{1/1,2/1,0/0} 
            \pattern{scale=0.5}{2}{1/1,2/2}{2/2,1/1,1/0}$
            & $\pattern{scale=0.5}{2}{1/1,2/2}{1/2,2/2,0/0} 
            \pattern{scale=0.5}{2}{1/1,2/2}{2/2,0/1,0/0} 
            \pattern{scale=0.5}{2}{1/1,2/2}{2/2,2/1,0/0} 
            \pattern{scale=0.5}{2}{1/1,2/2}{2/2,0/0,1/0}$
            & Thm~\ref{thm:class75} \\[0.5mm]
            \hline

        \end{tabular}
    \end{center} 
    \vspace{-0.5cm}
\caption{Other equidistribution results proved in this paper.}\label{tab-no-easy}
\end{table}

\begin{table}[!t]
    \renewcommand{\arraystretch}{1.5}
    \begin{center} 
        \begin{tabular}{|c|l|l||c|l|l|}
            \hline  
            \multicolumn{1}{|c|}{\footnotesize{nr.}} & \multicolumn{1}{c|}{\footnotesize{trivial}} & \multicolumn{1}{c||}{\footnotesize{trivial}} & \multicolumn{1}{c|}{\footnotesize{nr.}} & \multicolumn{1}{c|}{\footnotesize{trivial}} & \multicolumn{1}{c|}{\footnotesize{trivial}} \\
            \hline  

            \footnotesize{54} 
            & $\pattern{scale=0.5}{2}{1/1,2/2}{0/2,2/2,1/1,2/1,0/0} 
            \pattern{scale=0.5}{2}{1/1,2/2}{0/2,2/2,1/1,0/0,1/0} 
            \pattern{scale=0.5}{2}{1/1,2/2}{1/2,2/2,1/1,0/0,2/0} 
            \pattern{scale=0.5}{2}{1/1,2/2}{2/2,0/1,1/1,0/0,2/0}$
            & $\pattern{scale=0.5}{2}{1/1,2/2}{0/2,2/2,0/1,1/1,2/0} 
            \pattern{scale=0.5}{2}{1/1,2/2}{0/2,1/2,1/1,0/0,2/0} 
            \pattern{scale=0.5}{2}{1/1,2/2}{0/2,1/1,2/1,0/0,2/0} 
            \pattern{scale=0.5}{2}{1/1,2/2}{0/2,2/2,1/1,1/0,2/0}$
            & \footnotesize{69} 
            & $\pattern{scale=0.5}{2}{1/1,2/2}{1/2,1/1,2/1,0/0} 
            \pattern{scale=0.5}{2}{1/1,2/2}{2/2,0/1,1/1,1/0}$
            & $\pattern{scale=0.5}{2}{1/1,2/2}{0/2,1/1,2/1,1/0} 
            \pattern{scale=0.5}{2}{1/1,2/2}{1/2,0/1,1/1,2/0}$ \\
            \hline

            \footnotesize{71} 
            & $\pattern{scale=0.5}{2}{1/1,2/2}{1/2,2/2,2/1,0/0} 
            \pattern{scale=0.5}{2}{1/1,2/2}{2/2,0/1,0/0,1/0}$
            & $\pattern{scale=0.5}{2}{1/1,2/2}{0/2,2/2,2/1,1/0} 
            \pattern{scale=0.5}{2}{1/1,2/2}{0/2,2/1,0/0,2/0} 
            \pattern{scale=0.5}{2}{1/1,2/2}{1/2,2/2,0/1,2/0} 
            \pattern{scale=0.5}{2}{1/1,2/2}{1/2,0/1,0/0,2/0}$
            & 
            & 
            & \\
            \hline
            
        \end{tabular}
    \end{center} 
    \vspace{-0.5cm}
\caption{Conjectured equidistributions. The equidistributions in Classes 69 and 71 were conjectured in \cite{KZ2019}.}
    \label{tab-conject}
\end{table}

\begin{table}[!t]
    \centering
    \scriptsize 
    \renewcommand{\arraystretch}{1.5} 
    \setlength{\tabcolsep}{3pt} 
    
    \begin{center}
    
    \begin{tabular}{
    |>{\centering\arraybackslash}p{0.4cm}|
    >{\centering\arraybackslash}p{0.8cm}|
    >{\centering\arraybackslash}p{1.5cm}||
    >{\centering\arraybackslash}p{0.4cm}|
    >{\centering\arraybackslash}p{0.8cm}|
    >{\centering\arraybackslash}p{1.5cm}||
    >{\centering\arraybackslash}p{0.4cm}|
    >{\centering\arraybackslash}p{0.8cm}|
    >{\centering\arraybackslash}p{1.5cm}||
    >{\centering\arraybackslash}p{0.4cm}|
    >{\centering\arraybackslash}p{0.8cm}|
    >{\centering\arraybackslash}p{1.5cm}|
    }
    \hline
    nr. & enum. & ref. & nr. & enum. & ref. & nr. & enum. & ref. & nr. & enum. & ref. \\ \hline
    
    1  & D & Thm~\ref{thm:class1} & 2  & D & Thm~\ref{thm:class2} & 3  & D & \cite{KZ2019} & 4  & A & Thm~\ref{thm:avoidance-(n-1)!} \\ \hline
    5  & ? &  & 6  & A & Thm~\ref{thm:avoidance-(n-1)!}  & 7  & A & Thm~\ref{thm:avoidance-(n-1)!}  & 8  & D & \cite{KZ2019}\\ \hline
    9  & A & Thm~\ref{thm:avoidance-(n-1)!}  & 10 & A & Thm~\ref{thm:avoidance-(n-1)!}  & 11 & A & Thm~\ref{thm:avoidance-(n-1)!}  & 12 & A & Thm~\ref{thm:avoidance-(n-1)!}    \\ \hline
    13 & ? &   & 14 & ? &  & 15 & D & Thm~\ref{thm:class15} & 16 & A & Thm~\ref{thm:avoidance-1} \\ \hline
    17 & A & Thm~\ref{thm:avoidance-1}  & 18 & A & Thm~\ref{thm:avoidance-1} & 19 & A & Thm~\ref{thm:avoidance-19-21-80-97-103} & 20 & D & \cite{KZ2019}\\ \hline
    21 & A & Thm~\ref{thm:avoidance-19-21-80-97-103} & 22 & ? &  & 23 & ? & & 24 & A & Thm~\ref{thm:avoidance-24-61-86-105} \\ \hline
    25 & D & Thm~\ref{thm:class25} & 26 & ? & & 27 & ? & & 28 & A & Thm~\ref{thm:avoidance-(n!/2)} \\ \hline
    29 & ? &  & 30 & ? & & 31 & D & Thm~\ref{thm:class31} & 32 & D & Thm~\ref{thm:class32} \\ \hline
    33 & D & Thm~\ref{thm:class33} & 34 & D & \cite{KZZ} & 35 & D & \cite{KZ2019}& 36 & D & Thm~\ref{thm:class36} \\ \hline
    37 & D & \cite{KZ2019} & 38 & A & Thm~\ref{thm:avoidance-1}  & 39 & A & Thm~\ref{thm:avoidance-1}  & 40 & ? & \\ \hline
    41 & ? &  & 42 & A & Thm~\ref{thm:avoidance-(n-1)!}  & 43 & A & Thm~\ref{thm:avoidance-(n-1)!}  & 44 & D & \cite{KZ2019} \\ \hline
    45 & D & Thm~\ref{thm:class45} & 46 & A & Thm~\ref{thm:class46} & 47 & A & Thm~\ref{thm:class47} & 48 & A & Thm~\ref{thm:avoidance-48-67}  \\ \hline
    49 & D & Thm~\ref{thm:class49} & 50 & A & Thm~\ref{thm:avoidance-1}  & 51 & A & Thm~\ref{thm:avoidance-1} & 52 & D & \cite{KZ2019} \\ \hline
    53 & D & \cite{KZ2019} & 54 & A & Thm~\ref{thm:class54} & 55 & ? &  & 56 & A & Thm~\ref{thm-pat-56}\\ \hline
    57 & ? &  & 58 & ? &  & 59 & ? &  & 60 & ? & \\ \hline
    61 & D & Thm~\ref{dis-patterns-61} & 62 & A & Thm~\ref{thm:class62}& 63 & D & Thm~\ref{thm:class63} & 64 & A & Thm~\ref{thm:avoidance-(n-1)!}   \\ \hline
    65 & A & Thm~\ref{thm:avoidance-(n-1)!}  & 66 & A & Thm~\ref{thm:avoidance-(n-1)!}  & 67 & A & Thm~\ref{thm:avoidance-48-67} & 68 & D & Thm~\ref{thm:class68} \\ \hline
    69 & ? &  & 70 & ? & & 71 & ? &  & 72 & A & Thm~\ref{thm:avoidance-1}  \\ \hline
    73 & D & Thm~\ref{thm:class73} & 74 & D & Thm~\ref{dis-pattern-74} & 75 & A & Thm~\ref{thm:avoidance-1}  & 76 & D & \cite{KZ2019} \\ \hline
    77 & D & \cite{KZ2019} & 78 & D & \cite{KZZ}& 79 & D & \cite{KZ2019}& 80 & D & Thm~\ref{thm:class80}\\ \hline
    81 & D & \cite{KZ2019}& 82 & ? &  & 83 & ? &  & 84 & A & Thm~\ref{thm:avoidance-32-62-84} \\ \hline
    85 & A & Thm~\ref{thm:avoidance-(n!/2)} & 86 & A & Thm~\ref{thm:avoidance-24-61-86-105} & 87 & D & \cite{StanleyEC1} & 88 & ? &  \\ \hline
    89 & A & Thm~\ref{thm:avoidance-(n-1)!}  & 90 & A & Thm~\ref{thm:avoidance-(n-1)!}  & 91 & ? &  & 92 & A & Thm~\ref{thm:avoidance-(n-1)!}  \\ \hline
    93 & A & Thm~\ref{thm:avoidance-(n-1)!}  & 94 & A & Thm~\ref{thm:avoidance-1} & 95 & A & Thm~\ref{thm:avoidance-1}  & 96 & D & \cite{KZ2019} \\ \hline
    97 & A & Thm~\ref{thm:avoidance-19-21-80-97-103} & 98 & ? &  & 99 & ? &  & 100& A & Thm~\ref{thm:avoidance-(n-1)!}  \\ \hline
    101& A & Thm~\ref{thm:avoidance-(n-1)!}  & 102& A & Thm~\ref{thm:avoidance-1}  & 103& D & \cite{KZ2019} & 104& A & Thm~\ref{thm:avoidance-(n!/2)} \\ \hline
    105& A & Thm~\ref{thm:avoidance-24-61-86-105} &    &   &   &    &   &   &    &   &   \\ \hline
    \end{tabular}
\caption{Known enumerative results for all equivalence classes. ``A'' denotes known avoidance enumeration, and ``D'' denotes known distribution. For the reader's convenience, non-trivial known distributions from the literature are also presented in this paper. We note that, by the Shading Lemma in~\cite{Hilmarsson2015Wilf}, the following sets of classes have identical avoidance sets: $\{5,99\}$, $\{23,58\}$, $\{29,41,88\}$, $\{30,60\}$, and $\{59,82\}$. Hence, the table presents 19 open avoidance problems and 72 open distribution problems.}\label{tab:patterns_ref}

\end{center}\end{table}

The following theorem is used frequently in our paper. 

\begin{thm}[\cite{KZ2019}]\label{short} For the pattern $\pattern{scale=0.6}{1}{1/1}{0/1,1/0}$ (equivalently, $\pattern{scale=0.6}{1}{1/1}{0/0,1/1}$), we have
$$A(x)=\frac{F(x)}{1+xF(x)};\quad F(x,q)=\frac{F(x)}{1+x(1-q)F(x)}.$$
\end{thm}

The following enumerative results are known.

\begin{thm}[\cite{KZ2019}]\label{thm-pat-56}
For the patterns in Class 56 (represented by $\pattern{scale = 0.6}{2}{1/1,2/2}{0/1,1/2,0/0,2/1,2/0}$, $\pattern{scale = 0.6}{2}{1/1,2/2}{0/1,1/2,2/0,0/2,1/1}$, $\pattern{scale = 0.6}{2}{1/1,2/2}{0/1,1/0,0/0,1/1,2/2}$), 
	\begin{align*}
	A(x) = \frac{2F(x)-1}{F(x)};
	\ \ \ \ \ 
	F(x,q) = \frac{(2-q)F(x)+q-1}{(1-q)F(x)+q}.
	\end{align*}
\end{thm}

\begin{thm}[\cite{KZ2019}]\label{dis-patterns-61}
	For any pattern $p$ in Class 61 (represented by $\pattern{scale = 0.6}{2}{1/1,2/2}{0/1,1/1,1/2,1/0,1/2,2/1}$ and  $\pattern{scale = 0.6}{2}{1/1,2/2}{0/1,0/2,1/0,1/1,1/2}$), 
	\begin{align*}
	s_{n,k}(p) = & \, s_{n-1,k-1}(p)+ (k+1)s_{n-1,k+1}(p) + (n-k-1)s_{n-1,k}(p) 
	\end{align*}
	with  the initial conditions	$s_{1, 0}(p) = 1$, $s_{2, 0}(p) = 1$, and $s_{2, 1}(p) = 1$.
\end{thm}

\begin{thm}[\cite{KZ2019}]\label{dis-patterns-37}
	The pattern $p=\pattern{scale=0.6}{2}{1/1,2/2}{0/1,1/2,0/0,1/0,1/1,2/1}$ in Class $37$ satisfies  
	\begin{align*}
		s_{n,k}(p) = & \, (k+1)s_{n-1,k+1}(p)+(n-k)s_{n-1,k}(p)  -s_{n-2,k}(p)+s_{n-2,k-1}(p)
\end{align*}
	with  the initial conditions	$s_{1, 0}(p) = 1$, $s_{2, 0}(p) = 1$, and $s_{2, 1}(p) = 1$.
\end{thm}

\begin{thm}[\cite{KZ2019}]\label{dis-patterns-81}
	The pattern $p=\pattern{scale=0.6}{2}{1/1,2/2}{0/1,1/2,1/0,1/1,2/1,0/2}$ in Class $81$ satisfies 
	\begin{align*}
	s_{n,k}(p) & =   (k + 1) s_{n - 1, k + 1}(p) + (n - k - 1) s_{n - 1, k}(p)   + s_{n - 1, k - 1}(p)  \notag \\[5pt]
	& \quad + (k + 1) s_{n - 2, k + 1}(p) + (n - 2 k - 2) s_{n - 2, k}(p) - (n - k - 1) s_{n - 2, k - 1}(p)
	\end{align*}
with  the initial conditions	$s_{1, 0}(p) = 1$, $s_{2, 0}(p) = 1$, and $s_{2, 1}(p) = 1$.
\end{thm}	

A {\em descent} in a permutation $\pi = \pi_1 \cdots \pi_n$ is a position $i$ such that $\pi_i > \pi_{i+1}$. The distribution of descents is clearly equivalent to the distribution of the pattern
$\pattern{scale=0.6}{2}{1/1,2/2}{1/0,1/1,1/2}$,
and is given by the {\em Eulerian numbers}. We state this classical result in the following theorem.

\begin{thm}[\cite{StanleyEC1}]\label{dis-pattern-74}
For $n \geq 1$ and $0 \leq k \leq n-1$, the number of permutations in $S_n$ with exactly $k$ occurrences of
any pattern $p$ in Class $74$ is given by
\[
s_{n,k}(p)
=
\sum_{j=0}^{k+1}
(-1)^j
\binom{n+1}{j}
(k+1-j)^n.
\]
\end{thm}

\section{Easy explainable equidistributions}\label{sec-trivial}

In this section, we explain the equidistributions presented in Table~\ref{tab-easy} and discuss the corresponding enumerative results. We note that the proof of the following lemma in fact establishes the joint equidistribution of the patterns in question.

\begin{lem}\label{two-equidistr-lemmas} There are two classes of equidistributed mesh patterns, shown in Figure~\ref{two-equidistr-fig}, where each of $x_1,\dots,x_6$ takes one of two values: shaded or unshaded.\end{lem}

\begin{proof} Let $\pi \in S_n$. Define $\pi' \in S_n$ to be the permutation obtained from $\pi$ by replacing each entry $i$ with $n+1-i$ for $i \in [n-1]$ (that is, we keep the position of $n$ fixed and complement the remaining elements). Clearly, this map establishes the left equivalence in Figure~\ref{two-equidistr-fig}, as it exchanges occurrences of the corresponding mesh patterns.

The right equivalence in Figure~\ref{two-equidistr-fig} follows by reversing the elements of $\pi$ to the left of~$n$.\end{proof}

\begin{figure}[!t]
\begin{center}

\begin{tabular}{ccc}

\begin{tikzpicture}[scale=0.5]
  \tikzset{
    grid/.style={ draw, step=1cm, black!100, thick },
    graycell/.style={ fill=gray!50, draw=none, minimum width=1cm, minimum height=1cm, anchor=south west }
  }

  \begin{scope}
    \fill[graycell] (0,2) rectangle (1,3);
    \fill[graycell] (1,2) rectangle (2,3);
    \fill[graycell] (2,2) rectangle (3,3);

    \draw[grid] (0,0) grid (3,3);

    \node at (0.5, 1.5) {$x_1$};
    \node at (0.5, 0.5) {$x_2$};
    \node at (1.5, 1.5) {$x_3$};
    \node at (1.5, 0.5) {$x_4$};
    \node at (2.5, 1.5) {$x_5$};
    \node at (2.5, 0.5) {$x_6$};

    \filldraw (1,1) circle (4pt);
    \filldraw (2,2) circle (4pt);
  \end{scope}

  \begin{scope}[xshift=5cm]
    \fill[graycell] (0,2) rectangle (1,3);
    \fill[graycell] (1,2) rectangle (2,3);
    \fill[graycell] (2,2) rectangle (3,3);

    \draw[grid] (0,0) grid (3,3);

    \node at (0.5, 1.5) {$x_2$};
    \node at (0.5, 0.5) {$x_1$};
    \node at (1.5, 1.5) {$x_4$};
    \node at (1.5, 0.5) {$x_3$};
    \node at (2.5, 1.5) {$x_6$};
    \node at (2.5, 0.5) {$x_5$};

    \filldraw (1,1) circle (4pt);
    \filldraw (2,2) circle (4pt);
  \end{scope}

  \node at (4,1.5) {$\sim$};

\end{tikzpicture}

&

\ \ \ \ \ \

&

\begin{tikzpicture}[scale=0.5]
  \tikzset{
    grid/.style={ draw, step=1cm, black!100, thick },
    graycell/.style={ fill=gray!50, draw=none, minimum width=1cm, minimum height=1cm, anchor=south west }
  }

  \begin{scope}
    \fill[graycell] (0,2) rectangle (1,3);
    \fill[graycell] (1,2) rectangle (2,3);
    \fill[graycell] (2,2) rectangle (3,3);

    \draw[grid] (0,0) grid (3,3);

    \node at (0.5, 1.5) {$x_1$};
    \node at (0.5, 0.5) {$x_2$};
    \node at (1.5, 1.5) {$x_3$};
    \node at (1.5, 0.5) {$x_4$};
    \node at (2.5, 1.5) {};
    \node at (2.5, 0.5) {};

    \filldraw (1,1) circle (4pt);
    \filldraw (2,2) circle (4pt);
  \end{scope}

  \begin{scope}[xshift=5cm]
    \fill[graycell] (0,2) rectangle (1,3);
    \fill[graycell] (1,2) rectangle (2,3);
    \fill[graycell] (2,2) rectangle (3,3);

    \draw[grid] (0,0) grid (3,3);

    \node at (0.5, 1.5) {$x_3$};
    \node at (0.5, 0.5) {$x_4$};
    \node at (1.5, 1.5) {$x_1$};
    \node at (1.5, 0.5) {$x_2$};
    \node at (2.5, 1.5) {};
    \node at (2.5, 0.5) {};

    \filldraw (1,1) circle (4pt);
    \filldraw (2,2) circle (4pt);
  \end{scope}

  \node at (4,1.5) {$\sim$};

\end{tikzpicture}

\end{tabular}
\caption{Equidistributions in Lemma~\ref{two-equidistr-lemmas}.}\label{two-equidistr-fig}
\end{center}
\end{figure}
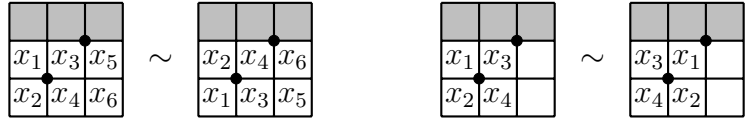

\begin{thm}\label{thm:class1} The patterns in Class $1$ are equivalent. The distribution on $S_n$, for $n\geq 1$, is given by $$n!-\sum_{i=1}^{n-1}(i-1)!(n-i-1)!+q\sum_{i=1}^{n-1}(i-1)!(n-i-1)!.$$   \end{thm}

\begin{proof}
The equivalence follows from an application of Lemma~\ref{two-equidistr-lemmas} to the patterns $\pattern{scale = 0.6}{2}{1/1,2/2}{0/0,0/1,0/2,1/0,1/2,2/1,2/2}$, $\pattern{scale = 0.6}{2}{1/1,2/2}{0/0,0/2,1/1,1/0,1/2,2/1,2/2}$, $\pattern{scale = 0.6}{2}{1/1,2/2}{0/0,0/1,0/2,1/1,1/2,2/0,2/2}$, $\pattern{scale = 0.6}{2}{1/1,2/2}{0/1,0/2,1/0,1/1,1/2,2/0,2/2}$. The distribution of the pattern $\pattern{scale = 0.6}{2}{1/1,2/2}{0/0,0/1,0/2,1/1,1/2,2/0,2/1}$ was determined in~\cite[Thm 2.8]{KZ2019}.
\end{proof}

\begin{thm}\label{thm:class2} The patterns in Class $2$ are equivalent. The distribution on $S_n$, for $n\geq 1$, is given by $(n-1)!+q(n!-(n-1)!)$. 
\end{thm}

\begin{proof} Applying Lemma~\ref{two-equidistr-lemmas} to the patterns 
$\pattern{scale = 0.6}{2}{1/1,2/2}{0/0,0/1,0/2,1/1,1/2,2/1,2/2}$, 
$\pattern{scale = 0.6}{2}{1/1,2/2}{0/1,0/2,1/1,1/0,1/2,2/1,2/2}$, and 
$\pattern{scale = 0.6}{2}{1/1,2/2}{0/0,0/2,1/0,1/1,1/2,2/0,2/2}$, 
as well as to the patterns
$\pattern{scale = 0.6}{2}{1/1,2/2}{0/1,0/2,1/1,1/2,2/0,2/1,2/2}$ and
$\pattern{scale = 0.6}{2}{1/1,2/2}{0/0,0/2,1/0,1/2,2/0,2/1,2/2}$, it only remains to show that 
$\pattern{scale = 0.6}{2}{1/1,2/2}{0/0,0/1,0/2,1/1,1/2,2/1,2/2}\sim\pattern{scale = 0.6}{2}{1/1,2/2}{0/1,0/2,1/1,1/2,2/0,2/1,2/2}$. This equivalence follows from a simple bijection on $S_n$ that sends permutations of the form $(n-1)A n B$ with a unique occurrence of the first pattern to permutations of the form $A(n-1)B n$ with a unique occurrence of the second pattern. Finally, $\pi \in S_n$ contains a unique occurrence of the pattern $\pattern{scale = 0.6}{2}{1/1,2/2}{0/1,0/2,1/1,1/2,2/0,2/1,2/2}$ if and only if $\pi$ ends with $n$, which yields $(n-1)!$ such permutations.
\end{proof}

\begin{thm}\label{thm:class25}
The patterns in Class $25$ are equivalent. For a pattern $p$ in this class, $s_{n,n-1}(p)=(n-1)!$ and, for $n\geq 2$ and $0\leq k\leq n-2$, 
\begin{equation}\label{thm:class25-distr}
s_{n,k}(p)=\frac{n!}{(k+1)(k+2)}.
\end{equation}
\end{thm}

\begin{proof}
The equivalence of the patterns in Class 25 follows from Lemma~\ref{two-equidistr-lemmas} applied to the patterns $\pattern{scale = 0.6}{2}{1/1,2/2}{0/2,1/2,2/1,2/2}$ and $\pattern{scale = 0.6}{2}{1/1,2/2}{0/2,1/2,2/0,2/2}$. Let $p=\pattern{scale = 0.6}{2}{1/1,2/2}{0/2,1/2,2/1,2/2}$.

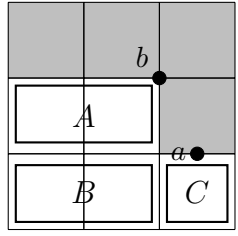
\begin{figure}[!t]
\begin{center}
	\begin{tikzpicture}[scale=1, baseline=(current bounding box.center)]
	\foreach \x/\y in {0/2,1/2,2/1,2/2}		    
	\fill[gray!50] (\x,\y) rectangle +(1,1);
	\draw (0.01,0.01) grid (2+0.99,2+0.99);
	\filldraw (2.5,1) circle (2.5pt) node[left] {\small{$a$}};
	\filldraw (2,2) circle (2.5pt) node[above left] {\small{$b$}};

    \tikzset{
    grid/.style={ draw, step=1cm, black!100, thin },
    graycell/.style={ fill=gray!50, draw=none, minimum width=1cm, minimum height=1cm, anchor=south west }
  }

   \draw[grid] (0,0) grid (3,3);

  \draw[thick] (0.1, 1.15) rectangle (1.9, 1.9);
  \node at (1.0, 1.5) {$A$};
  \draw[thick] (0.1, 0.1) rectangle (1.9, 0.85);
  \node at (1.0, 0.5) {$B$};
  \draw[thick] (2.1, 0.1) rectangle (2.9, 0.85);
  \node at (2.5, 0.5) {$C$};

\end{tikzpicture}
\end{center}
\caption{Related to the proof of Theorem~\ref{thm:class25}.}\label{thm:class25-fig}
\end{figure}

Referring to Figure~\ref{thm:class25-fig}, let $\pi\in S_n$ and let $b$ be its largest element. 
Assume that there are $k$ elements in block $A$ and $i$ elements in block $B$, where $0\le i \le n-k-2$; the remaining elements lie in block $C$.

If $C=\varepsilon$, then $\pi$ clearly contains $n-1$ occurrences of $p$, and there are $(n-1)!$ such permutations. Hence $s_{n,n-1}=(n-1)!$. 
On the other hand, if $C\neq \varepsilon$, let $a=n-k-1$ be the largest element in $C$. We note that $a$ can be inserted into $C$ in $n-k-1-i$ ways, and there are $\binom{i+k}{i}$ ways to choose the positions of the $k$ elements of $A$ in $\pi$. 
Furthermore, the number of ways to arrange the elements in $B$ and $C$ is $(n-k-2)!$, while the number of ways to arrange the $k$ elements of $A$ in their fixed positions is $k!$. 
Thus, for $0\le k \le n-2$,
\[
s_{n,k}(p)=k!(n-k-2)!\sum_{i=0}^{n-k-2} \binom{i+k}{i}\,(n-k-1-i).
\]
After some algebraic manipulations, we obtain~\eqref{thm:class25-distr}.
\end{proof}

\begin{thm}\label{thm:class31}
The patterns in Class $31$ are equivalent. The distribution on $S_n$, for $n\geq 1$, is given by 
$$n!-\sum_{i=0}^{n-2}i!(n-i-1)!+q\sum_{i=0}^{n-2}i!(n-i-1)!.$$\end{thm}

\begin{proof}
The equivalence of the patterns in Class 31 follows from Lemma~\ref{two-equidistr-lemmas} applied to the patterns $\pattern{scale = 0.6}{2}{1/1,2/2}{0/0,0/2,1/0,1/2,2/1,2/2}$ and $\pattern{scale = 0.6}{2}{1/1,2/2}{0/1,0/2,1/1,1/2,2/0,2/2}$. The distribution is found in \cite[Thm 2.7]{KZ2019}.
\end{proof}

\begin{thm}\label{thm:class32}
The patterns in Class $32$ are equivalent. The distribution on $S_n$, for $n\geq 1$, is given by 
$$n!-\sum_{i=1}^{n-1}\frac{(n-1)!}{i}+q\sum_{i=1}^{n-1}\frac{(n-1)!}{i}.$$\end{thm}

\begin{proof}
The equivalence of the patterns in Class 32 follows from Lemma~\ref{two-equidistr-lemmas}. The distribution of the pattern $\pattern{scale = 0.6}{2}{1/1,2/2}{0/0,0/1,0/2,1/2,2/0,2/2}$ is found in \cite[Thm 2.6]{KZ2019}.
\end{proof}

\begin{thm}\label{thm:class33}
The patterns in Class~$33$ are equivalent. Moreover, for a pattern $p$ in this class, 
$s_{0,0} = s_0(1,0) = 1$, and for $n \ge 2$, $s_{n,0} = n!/2$ and $s_{n,1} = n!/2$. 
Equivalently,
\[
F(x,q) = \frac{(1-q)(1+x)}{2} + \frac{1+q}{2} F(x).
\]
\end{thm}

\begin{proof} The distribution for the pattern $\pattern{scale = 0.6}{2}{1/1,2/2}{0/0,0/1,0/2,2/0,2/1,2/2}$ is given in  \cite[Thm 2.2]{KZ2019}. On the other hand, the pattern $\pattern{scale=0.6}{2}{1/1,2/2}{0/2,1/2,2/2,0/1,1/1,2/1}$ appears in a permutation (only once) if and only if the element $n-1$ preceeds the element $n$, which gives the same distribution.  So, $\pattern{scale = 0.6}{2}{1/1,2/2}{0/0,0/1,0/2,2/0,2/1,2/2}\sim\pattern{scale=0.6}{2}{1/1,2/2}{0/2,1/2,2/2,0/1,1/1,2/1}$.
\end{proof}

\begin{thm}\label{thm:class45}
The patterns in Class~$45$ are equivalent. The distribution on $S_n$, for $n \geq 1$, is given by 
$(n-1)! + q(n-1)(n-1)!$.
\end{thm}

\begin{proof}
The equivalence of the patterns 
$\pattern{scale = 0.6}{2}{1/1,2/2}{0/0,0/1,0/2,1/2,2/2}$ and 
$\pattern{scale = 0.6}{2}{1/1,2/2}{0/2,1/0,1/1,1/2,2/2}$, as well as 
$\pattern{scale = 0.6}{2}{1/1,2/2}{0/1,0/2,1/1,1/2,2/2}$ and 
$\pattern{scale = 0.6}{2}{1/1,2/2}{0/0,0/2,1/0,1/2,2/2}$, 
follows from Lemma~\ref{two-equidistr-lemmas}. The pattern 
$\pattern{scale = 0.6}{2}{1/1,2/2}{0/0,0/1,0/2,1/2,2/2}$ 
can occur at most once, and any permutation starting with the largest element avoids this pattern, which gives the desired enumeration. The same holds for the pattern 
$\pattern{scale = 0.6}{2}{1/1,2/2}{0/1,0/2,1/1,1/2,2/2}$; in particular, the unique occurrence is formed by the largest element and the largest element to its left. This completes the proof.
\end{proof}

\begin{thm}\label{thm:class53}
The patterns in Class~$53$ are equivalent. Moreover, for a pattern $p$ in this class, 
$s_{0,0} = s_0(1,0) = 1$, and for $n \ge 2$, $s_{n,0} = n!-(n-2)!$ and $s_{n,1} = (n-2)!$.
\end{thm}

\begin{proof}
The equivalence of the patterns in Class 53 follows from Lemma~\ref{two-equidistr-lemmas}.  The distribution is given by \cite[Thm 2.5]{KZ2019}.
\end{proof}

\begin{thm}\label{thm:class60} The patterns in Class $60$ are equivalent. \end{thm}

\begin{proof}
The equivalence follows from Lemma~\ref{two-equidistr-lemmas} applied to  the patterns
$\pattern{scale = 0.6}{2}{1/1,2/2}{0/0,0/2,1/1,1/2,2/2}$ and 
$\pattern{scale = 0.6}{2}{1/1,2/2}{0/1,0/2,1/0,1/0,1/2,2/2}$. 
\end{proof}

\begin{thm}\label{thm:class62} The patterns in Class $62$ are equivalent. Also, for $n\geq 1$, there are $(n-1)!(n-H_{n-1})$ $n$-permutations avoiding any pattern in this class, where $H_n = \sum_{k=1}^{n} \frac{1}{k}$ is the $n$-th Harmonic number. \end{thm}

\begin{proof}
The equivalence follows from Lemma~\ref{two-equidistr-lemmas} applied to  
$\pattern{scale=0.6}{2}{1/1,2/2}{0/2,1/2,2/2,2/1,0/0}$,  $\pattern{scale=0.6}{2}{1/1,2/2}{0/2,1/2,2/2,2/1,1/0}$,  $\pattern{scale=0.6}{2}{1/1,2/2}{0/2,1/2,2/2,0/1,2/0}$, and $\pattern{scale=0.6}{2}{1/1,2/2}{0/2,1/2,2/2,1/1,2/0}$. 

Next, we enumerate the avoiders of
$\pattern{scale=0.6}{2}{1/1,2/2}{0/2,1/2,2/2,2/1,0/0}$.
Let $\pi = \pi_1 \cdots \pi_n \in S_n$.
If $\pi_1 = n$, then $\pi$ cannot contain an occurrence of
$\pattern{scale=0.6}{2}{1/1,2/2}{0/2,1/2,2/2,2/1,0/0}$.
On the other hand, if $\pi_1 \neq n$ and $\pi$ contains an occurrence of the pattern, then it is easy to see that $\pi_1 n$ must form such an occurrence. In this case, the elements $\pi_1+1, \pi_1+2, \ldots, n-1$ must all lie between $\pi_1$ and $n$ in $\pi$. The number of choices for such a permutation is $\sum_{i=0}^{n-2} i!(n-i-2)! {n-1 \choose i+1}$,
where $(n-i-2)!$ counts the number of ways to arrange the elements smaller than $\pi_1$, and the binomial coefficient counts the choices of positions for the elements larger than $\pi_1$ in $\pi$. Therefore, the number of $n$-permutations with $\pi_1 \neq n$ that avoid
$\pattern{scale=0.6}{2}{1/1,2/2}{0/2,1/2,2/2,2/1,0/0}$
is
$$
n! - (n-1)! - \sum_{i=0}^{n-2} i!(n-i-2)! {n-1 \choose i+1}.
$$
Adding the $(n-1)!$ avoiders with $\pi_1 = n$, we obtain the desired result.
\end{proof}

\begin{thm}\label{thm:class63} The patterns in Class $63$ are equivalent. For $n\geq 2$, there are $n!/2$ $n$-permutations avoiding any pattern in this class.\end{thm}

\begin{proof}
The equivalence follows from Lemma~\ref{two-equidistr-lemmas} applied to  
$\pattern{scale=0.6}{2}{1/1,2/2}{0/2,1/2,2/2,0/1,2/1}$, 
$\pattern{scale=0.6}{2}{1/1,2/2}{0/2,1/2,2/2,1/1,2/1}$, 
$\pattern{scale=0.6}{2}{1/1,2/2}{0/2,1/2,2/2,1/0,2/0}$, and 
$\pattern{scale=0.6}{2}{1/1,2/2}{0/2,1/2,2/2,0/0,2/0}$. 

Clearly, an $n$-permutation avoids the pattern 
$\pattern{scale=0.6}{2}{1/1,2/2}{0/2,1/2,2/2,1/1,2/1}$ if and only if the element $n-1$ appears to the right of $n$. Hence, there are $n!/2$ permutations of length $n$ that avoid any pattern in this class.
\end{proof}

Only part of the proof of the next theorem admits an “easily explainable equidistribution”; nevertheless, we present the proof in full in this section rather than splitting it into separate parts. Also note that Theorem~\ref{thm:class15} below gives yet another connection to the unsigned Stirling number of the first kind.

\begin{thm}\label{thm:class68} The patterns in Class $68$ are equivalent. For any pattern $p$ in the class, 
\begin{equation}\label{class68-rec}
s_{n,k}(p)=s_{n-1,k-1}(p)+(n-1)s_{n-1,k}(p)
\end{equation}
with the initial conditions $s_{n,0}(p)=(n-1)!$ for $n\geq 1$ and $s_{0,0}(p)=1$, which shows that $s_{n,k}(p)=C(n,k+1)$, the unsigned Stirling number of the first kind. The g.f. for $s_{n,k}(p)$ is given by
$$\sum_{k=0}^{n-1}s_{n,k}(p)x^k=\prod_{i=1}^{n-1}(x+i).$$ 
\end{thm}

\begin{proof}
The equivalence of $\pattern{scale = 0.6}{2}{1/1,2/2}{0/0,0/1,1/0,1/1}$ and $\pattern{scale = 0.6}{2}{1/1,2/2}{0/1,1/1,1/2,2/1}$, together with the enumeration~(\ref{class68-rec}), is established in \cite[Thm.~4.1]{KZ2019}. The equivalence of $\pattern{scale = 0.6}{2}{1/1,2/2}{0/0,0/1,1/0,1/1}$ and $\pattern{scale = 0.6}{2}{1/1,2/2}{0/1,0/2,1/1,1/2}$, together with the enumeration~(\ref{class68-rec}), is established in \cite[Thm.~2.4]{LK2025}. Moreover, the equivalence of  the patterns $\pattern{scale = 0.6}{2}{1/1,2/2}{0/0,0/2,1/2,2/2}$, 
$\pattern{scale = 0.6}{2}{1/1,2/2}{0/1,0/2,1/2,2/2}$,
$\pattern{scale = 0.6}{2}{1/1,2/2}{1/1,0/2,1/2,2/2}$, and
$\pattern{scale = 0.6}{2}{1/1,2/2}{1/0,0/2,1/2,2/2}$  follows from Lemma~\ref{two-equidistr-lemmas}. 

We complete the proof by showing that the distribution of the patterns
$\pattern{scale = 0.6}{2}{1/1,2/2}{0/2,1/1,1/2,2/2}$,
$\pattern{scale = 0.6}{2}{1/1,2/2}{0/1,0/2,2/1,2/2}$, and
$\pattern{scale = 0.6}{2}{1/1,2/2}{0/1,0/2,1/1,2/1}$
is given by~(\ref{class68-rec}). In what follows, let $\pi=\pi_1\cdots \pi_n \in S_n$ and $\pi' \in S_{n-1}$.\\[-3mm]

\noindent
{\bf The pattern $\pattern{scale = 0.6}{2}{1/1,2/2}{0/2,1/1,1/2,2/2}$.}   If $\pi_1 \neq n$, then $\pi_1 n$ forms an occurrence of the pattern, whereas if $\pi_1 = n$, then $\pi$ avoids the pattern. It follows that $s_{n,0}(p) = (n-1)!$. Now construct $\pi$ by inserting the new minimum element into $\pi'$. Clearly, the number of occurrences of $p$ in $\pi'$ cannot decrease. Moreover, the only insertion that increases the number of occurrences of $p$ (by exactly one) is placing the new minimum element immediately to the left of the maximum element in $\pi'$ (this new occurrence is formed by the elements 1 and $n$ in $\pi$). The recurrence~(\ref{class68-rec}) now follows. \\[-3mm]

\noindent
{\bf The pattern $\pattern{scale = 0.6}{2}{1/1,2/2}{0/1,0/2,2/1,2/2}$.}  Clearly, if $\pi_1 = n$, then $\pi$ avoids $p$. Suppose $\pi_1 \neq n$, and let $x \neq n$ be a left-to-right maximum. If $x = n-1$, then $xn$ is an occurrence of $p$ in $\pi$; otherwise, $xy$ is an occurrence of $p$, where $y$ is the rightmost element of $\pi$ such that $y > x$. It follows that $s_{n,0}(p) = (n-1)!$.  Now construct $\pi$ by inserting the new minimum element into $\pi'$. Clearly, the number of occurrences of $p$ in $\pi'$ cannot decrease. Moreover, the only insertion that increases the number of occurrences of $p$ (by exactly one) is placing the new minimum element in the leftmost position (this new occurrence is formed by the leftmost and rightmost elements in $\pi$). The recurrence~(\ref{class68-rec}) now follows.\\[-3mm]

\noindent
{\bf The pattern $\pattern{scale = 0.6}{2}{1/1,2/2}{0/1,0/2,1/1,2/1}$.}  Clearly, if $\pi_1 = n$, then $\pi$ avoids $p$. Suppose $\pi_1 \neq n$ then $\pi_1(\pi_1+1)$ is an occurrence of $p$ in $\pi$.  It follows that $s_{n,0}(p) = (n-1)!$. Now construct $\pi$ by inserting the new minimum element into $\pi'$. Clearly, the number of occurrences of $p$ in $\pi'$ cannot decrease. Moreover, the only insertion that increases the number of occurrences of $p$ (by exactly one) is placing the new minimum element in the leftmost position (this new occurrence is formed by the elements 1 and 2 in $\pi$). The recurrence~(\ref{class68-rec}) now follows.\\[-3mm]

The proof is completed.
\end{proof}

\begin{thm}\label{thm:class77}
The patterns in Class $77$ are equivalent. Moreover, 
    \[
    A(x) = (1-x)F(x) + \frac{xF(x)}{1+xF(x)},
    \]
    \[
    F(x,q) = (1-x)F(x) + \frac{xF(x)}{1+x(1-q)F(x)}.
    \]
\end{thm}

\begin{proof}
The equivalence of the patterns in Class 77 follows from Lemma~\ref{two-equidistr-lemmas} applied to the patterns $\pattern{scale = 0.6}{2}{1/1,2/2}{0/0,0/2,1/1,1/2,2/0,2/1,2/2}$ and $\pattern{scale = 0.6}{2}{1/1,2/2}{0/1,0/2,1/0,1/2,2/0,2/1,2/2}$. The distribution of $\pattern{scale = 0.6}{2}{1/1,2/2}{0/1,0/2,1/0,1/2,2/0,2/1,2/2}$ is found in \cite[Thm 3.2]{KZ2019}.
\end{proof}

\begin{thm}\label{thm:class78}
The patterns in Class~$78$ are equivalent. Moreover,
\begin{align*}		
A(x) & =(1-x)F(x)+x, \\
F(x,q) & =(1-x)F(x)+x+x\sum_{n=1}^{\infty}\prod_{i=0}^{n-1}(q+i)x^n.
\end{align*}
\end{thm}

\begin{proof} The equivalence of the patterns in Class 78 follows from Lemma~\ref{two-equidistr-lemmas} applied to the patterns $\pattern{scale=0.6}{2}{1/1,2/2}{0/2,1/2,2/2,0/1,2/1,2/0}$, $\pattern{scale=0.6}{2}{1/1,2/2}{0/2,1/2,2/2,1/1,2/1,2/0}$, and $\pattern{scale=0.6}{2}{1/1,2/2}{0/2,1/2,2/2,2/1,0/0,2/0}$. The distribution of  $\pattern{scale=0.6}{2}{1/1,2/2}{0/2,1/2,2/2,1/1,2/1,2/0}$ is given by~\cite[Cor 2.5]{KZZ}.  
\end{proof}

Only part of the proof of the next theorem has an “easily explainable equidistribution”.

\begin{thm}\label{thm:class80}
The patterns in Class $80$ are equivalent. Moreover, 
\[
A(x) = \frac{(1+x)F(x)}{1+xF(x)}, \quad F(x,q) = \frac{(1 + x(1-q))F(x)}{1+x(1-q)F(x)}.
\]
\end{thm}

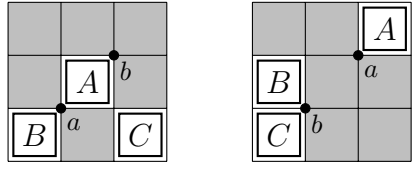
\begin{figure}
\begin{center}

\begin{tabular}{ccc}

\begin{tikzpicture}[scale=0.7]
  \tikzset{
    grid/.style={ draw, step=1cm, black!100, thin },
    graycell/.style={ fill=gray!50, draw=none, minimum width=1cm, minimum height=1cm, anchor=south west }
  }

  \fill[graycell] (0,2) rectangle (1,3);
  \fill[graycell] (1,2) rectangle (2,3);
  \fill[graycell] (2,2) rectangle (3,3);
  \fill[graycell] (2,1) rectangle (3,2);
  \fill[graycell] (0,1) rectangle (1,2);
  \fill[graycell] (1,0) rectangle (2,1);

  \draw[grid] (0,0) grid (3,3);

  \draw[thick] (1.1, 1.1) rectangle (1.9, 1.9);
  \draw[thick] (2.1, 0.1) rectangle (2.9, 0.9);
  \draw[thick] (0.1, 0.1) rectangle (0.9, 0.9);

  \node[anchor=center] at (1.5, 1.5) {$A$};
  \node[anchor=center] at (0.5, 0.5) {$B$};
  \node[anchor=center] at (2.5, 0.5) {$C$};

  \filldraw[black] (1,1) circle (2.5pt);
  \node[anchor=west] at (0.9,0.7) {\footnotesize{$a$}};
  \filldraw[black] (2,2) circle (2.5pt);
  \node[anchor=west] at (1.9,1.7) {\footnotesize{$b$}};
\end{tikzpicture}

&

\ \ \ 

&

\begin{tikzpicture}[scale=0.7]
  \tikzset{
    grid/.style={ draw, step=1cm, black!100, thin },
    graycell/.style={ fill=gray!50, draw=none, minimum width=1cm, minimum height=1cm, anchor=south west }
  }

  \fill[graycell] (0,2) rectangle (1,3);
  \fill[graycell] (1,2) rectangle (2,3);
  \fill[graycell] (1,1) rectangle (2,2);
  \fill[graycell] (2,1) rectangle (3,2);
  \fill[graycell] (2,0) rectangle (3,1);
  \fill[graycell] (1,0) rectangle (2,1);

  \draw[grid] (0,0) grid (3,3);

  \draw[thick] (2.1, 2.1) rectangle (2.9, 2.9);
  \node at (2.5, 2.5) {$A$};
  \draw[thick] (0.1, 1.1) rectangle (0.9, 1.9);
  \node at (0.5, 1.5) {$B$};
  \draw[thick] (0.1, 0.1) rectangle (0.9, 0.9);
  \node at (0.5, 0.5) {$C$};

  \filldraw[black] (1,1) circle (2.5pt);
  \node[anchor=west] at (0.9,0.7) {\footnotesize{$b$}};
  \filldraw[black] (2,2) circle (2.5pt);
  \node[anchor=west] at (1.9,1.7) {\footnotesize{$a$}};
\end{tikzpicture}

\end{tabular}

\caption{Related to the proof of Theorem~\ref{thm:class80}.}\label{thm:class80-fig}
\end{center}
\end{figure}

\begin{proof}
The equivalence of the patterns 
$\pattern{scale = 0.6}{2}{1/1,2/2}{0/1,0/2,1/0,1/2,2/0,2/2}$, 
$\pattern{scale = 0.6}{2}{1/1,2/2}{0/0,0/2,1/1,1/2,2/1,2/2}$,
$\pattern{scale = 0.6}{2}{1/1,2/2}{0/1,0/2,1/0,1/2,2/1,2/2}$, and
$\pattern{scale = 0.6}{2}{1/1,2/2}{0/0,0/2,1/1,1/2,2/0,2/2}$ follows from Lemma~\ref{two-equidistr-lemmas}. The equivalence of these patterns with  the pattern $\pattern{scale = 0.6}{2}{1/1,2/2}{0/2,1/0,1/1,1/2,2/0,2/1}$ follows from the same enumeration given below. In what follows, let $\pi=\pi_1\cdots \pi_n \in S_n$ and $\pi' \in S_{n-1}$.\\[-3mm]

\noindent
{\bf The pattern $\hspace{-1mm}\pattern{scale = 0.6}{2}{1/1,2/2}{0/1,0/2,1/0,1/2,2/1,2/2}$.}  We claim that
\begin{equation}\label{class80:A(x)-rec}
F(x)=A(x) + x(F(x)-1)\frac{F(x)}{1+xF(x)}.
\end{equation}
Indeed, each permutation (counted by $F(x)$ on the LHS) either avoids the pattern (and hence is counted by $A(x)$) or contains it. Among all occurrences $ab$ of the pattern in $\pi$ (where $b$ is the largest element), select the one with the largest possible $a$, as shown in the left-hand picture in Figure~\ref{thm:class80-fig}. Then the block $A$ must avoid the pattern $\pattern{scale=0.6}{1}{1/1}{0/1,1/0}$ (whose enumeration is given by Theorem~\ref{short}), while $BbC$ can be any non-empty permutation counted by $F(x)-1$. Note that $\pi$ is uniquely determined by $A$ and $BbC$, and the element $a$ contributes a factor of $x$. This yields~(\ref{class80:A(x)-rec}) and the formula for $A(x)$.  For the distribution, we have
\begin{equation}\label{class80:F(x,q)-rec}
F(x,q) = A(x) + xq\,(F(x,q)-1)\frac{F(x)}{1+xF(x)},
\end{equation}
whose derivation follows the same lines as in the case of avoidance. We note that each occurrence of the pattern in $BbC$ is also an occurrence in $\pi$, and the factor of $q$ corresponds to the distinguished occurrence $ab$ of the pattern. The formula for $F(x,q)$ follows from (\ref{class80:F(x,q)-rec}). \\[-3mm]

\noindent
{\bf The pattern $\pattern{scale = 0.6}{2}{1/1,2/2}{0/2,1/0,1/1,1/2,2/0,2/1}$.} This case is similar to the one already considered, and the recurrence relations~(\ref{class80:A(x)-rec}) and~(\ref{class80:F(x,q)-rec}) can be derived using the structure shown in the right-hand picture in Figure~\ref{thm:class80-fig}, where the corresponding blocks and elements are labeled by the same letters (in particular, $a$ is chosen to be as large as possible). We omit further details.
\end{proof}

\section{Equidistributions by generating functions}\label{gf-sec}

\begin{thm}\label{thm:class36}
The patterns in Class $36$ are equivalent. Moreover, 
\begin{equation}\label{formulas-class36}
A(x) = \frac{F(x)}{1+x(F(x)-1)}, \quad F(x,q) = \frac{F(x)}{1+(1-q)x(F(x)-1)}.
\end{equation}
\end{thm}

\begin{proof}
The distribution of $\pattern{scale=0.6}{2}{1/1,2/2}{0/0,0/1,1/1,1/2,2/1,2/2}$ is given in~\cite[Theorem~3.9]{KZ2019} and is stated in~(\ref{formulas-class36}). We show that the distribution of $\pattern{scale = 0.6}{2}{1/1,2/2}{0/0,0/1,1/1,1/2,2/0,2/1}$ is the same. We claim that
\begin{equation}\label{class36:A(x)-rec}
F(x)=A(x) + xA(x)(F(x)-1).
\end{equation}
Indeed, each permutation (counted by $F(x)$ on the LHS) either avoids the pattern (and hence is counted by $A(x)$) or contains it. Among all occurrences $ab$ of the pattern in $\pi$ ($b=a+1$), select the one with the smallest possible $a$, as shown in Figure~\ref{thm:class36-fig}. Then the block $A$ must avoid $\pattern{scale = 0.6}{2}{1/1,2/2}{0/0,0/1,1/1,1/2,2/0,2/1}$, while $BbC$ can be any non-empty permutation counted by $F(x)-1$. Note that $\pi$ is uniquely determined by $A$ and $BbC$, and the element $a$ contributes a factor of $x$. This yields~(\ref{class36:A(x)-rec}), and the formula for $A(x)$ given in (\ref{formulas-class36}).  For the distribution, we have
\begin{equation}\label{class36:F(x,q)-rec}
F(x,q) = A(x) + xqA(x)(F(x,q)-1),
\end{equation}
whose derivation follows the same lines as in the case of avoidance. We note that each occurrence of the pattern in $BbC$ is also an occurrence in $\pi$, and the factor of $q$ corresponds to the occurrence $ab$ of the pattern. The formula for $F(x,q)$ follows from (\ref{class36:F(x,q)-rec}), and it is given by (\ref{formulas-class36}). 
\end{proof}

\begin{figure}
\begin{center}
\begin{tikzpicture}[scale=0.7]
  \tikzset{
    grid/.style={ draw, step=1cm, black!100, thin },
    graycell/.style={ fill=gray!50, draw=none, minimum width=1cm, minimum height=1cm, anchor=south west }
  }

  \fill[graycell] (1,2) rectangle (2,3);
  \fill[graycell] (0,1) rectangle (1,2);
  \fill[graycell] (1,1) rectangle (2,2);
  \fill[graycell] (2,1) rectangle (3,2);
  \fill[graycell] (2,0) rectangle (3,1);
  \fill[graycell] (0,0) rectangle (1,1);

  \draw[grid] (0,0) grid (3,3);

  \draw[thick] (0.1, 2.1) rectangle (0.9, 2.9);
  \node at (0.5, 2.5) {$B$};
  \draw[thick] (2.1, 2.1) rectangle (2.9, 2.9);
  \node at (2.5, 2.5) {$C$};
  \draw[thick] (1.1, 0.1) rectangle (1.9, 0.9);
  \node at (1.5, 0.5) {$A$};

  \filldraw[black] (1,1) circle (2.5pt);
  \node[anchor=west] at (0.9,1.2) {\footnotesize{$a$}};
  \filldraw[black] (2,2) circle (2.5pt);
  \node[anchor=west] at (1.9,1.7) {\footnotesize{$b$}};
\end{tikzpicture}
\end{center}
\caption{Related to the proof of Theorem~\ref{thm:class36}.}\label{thm:class36-fig}
\end{figure}
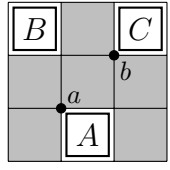

\section{Equidistributions by recurrence relations}\label{recurrence-sec}

The following theorem provides another connection to the unsigned Stirling numbers of the first kind (such a connection was already observed in Theorem~\ref{thm:class68}).

\begin{thm}\label{thm:class15} The patterns in Class $15$ are equivalent. For any pattern $p$ in the class, 
\begin{equation}\label{class15-rec}
s_{n,k}(p)=s_{n-1,k}(p)+(n-1)s_{n-1,k-1}(p)
\end{equation}
with the initial conditions $s_{0,0}(p)=s_{1,0}=1$ and $s_{1,1}(p)=0$, which shows that $s_{n,k}(p)=C(n,n-k)$, the unsigned Stirling number of the first kind. The bivariate exponential g.f. for $s_{n,k}(p)$ is given by
$$\sum_{n\geq 0}\sum_{k\geq 0}\frac{x^n}{n!}q^{k} = (1-xq)^{-1/q}.$$ 
\end{thm}

\begin{proof}
It easy to see that for both patterns, $\pattern{scale=0.6}{2}{1/1,2/2}{1/2,2/2}$ and  $\pattern{scale=0.6}{2}{1/1,2/2}{1/2,1/1}$, decreasing permutations are the only one avoiding them. 

Now, for the pattern  $\pattern{scale=0.6}{2}{1/1,2/2}{1/2,2/2}$, inserting the new minimal element to create a larger permutation never destroys any existing occurrence of the pattern, does not change the number of occurrences if inserted at the end, and introduces precisely one new occurrence of the pattern (formed by the new minimum element and the closest right-to-left maximum to its right) if inserted in a position other than the rightmost one. This explains (\ref{class15-rec}).

However, for the pattern $\pattern{scale=0.6}{2}{1/1,2/2}{1/2,1/1}$, exactly the same occurs when inserting the new minimum element, except that the new occurrence of the pattern is formed by the new minimum element and the element next to it on its right. Thus,  $\pattern{scale=0.6}{2}{1/1,2/2}{1/2,2/2}\sim\pattern{scale=0.6}{2}{1/1,2/2}{1/2,1/1}$, and the g.f.\ can be obtained using standard algebraic manipulations.
\end{proof}

\section{Equidistributions by bijections}\label{bijections-sec}

In this section, we present four bijections. Each bijection $f$ is in fact an involution, that is, $f(f(\pi)) = \pi$, and proves a stronger statement than mere equidistribution, namely the joint equidistribution of the patterns involved, since occurrences of the patterns in question are mapped directly to one another.


\begin{thm}\label{thm:class46} The patterns in Class $46$ are equivalent.
\end{thm}

\begin{proof}
We will explain bijectively that  $\pattern{scale=0.6}{2}{1/1,2/2}{1/2,2/2,0/1,1/1,1/0} \sim \hspace{-1mm}\pattern{scale=0.6}{2}{1/1,2/2}{1/2,2/2,1/1,0/0,1/0}$.

Let $\pi = \pi_1 \pi_2 \cdots \pi_n \in S_n$, and let $r_1 > r_2 > \cdots > r_\ell$ be the sequence of right-to-left maxima of $\pi$. Let $R = \{r_1, r_2, \dots, r_\ell\}$. Note that if $ab$ is an occurrence of either pattern, then $b = r_i$ for some $i$, and $a$ is immediately to the left of $r_i$. Also, each $r_i$ can be involved in at most one occurrence of $\pattern{scale=0.6}{2}{1/1,2/2}{1/2,2/2,0/1,1/1,1/0}$ or $\hspace{-1mm}\pattern{scale=0.6}{2}{1/1,2/2}{1/2,2/2,1/1,0/0,1/0}$.

The transformation $f: \mathcal{S}_n \to \mathcal{S}_n$ is obtained by considering $r_i$, for $i = \ell, \ell - 1, \ldots, 1$, one at a time. For each $r_i$, let $x_i$ be the element immediately to the left of $r_i$, if it exists. If $x_i \in R$, or if $r_i = \pi_1$, we do nothing. Otherwise, we define the set $S_i$ as:
\[
S_i := \{ y \mid y \text{ is to the left of } r_i, \text{ and } y < r_i \}.
\]
Note that by definition, $x_i \in S_i$. Let the elements of $S_i$ be ordered as $a_1 < a_2 < \cdots < a_j$, where $j = |S_i| \geq 1$. We perform the following steps:

\begin{itemize}
    \item  If $x_i r_i$ is not an occurrence of $\pattern{scale=0.6}{2}{1/1,2/2}{1/2,2/2,0/1,1/1,1/0}$ or $\hspace{-1mm}\pattern{scale=0.6}{2}{1/1,2/2}{1/2,2/2,1/1,0/0,1/0}$, or if it is involved in occurrences of both these patterns, we leave $\pi$ unchanged and proceed to $r_{i-1}$.
    
    \item If $x_ir_i$ is an occurrence of $\pattern{scale=0.6}{2}{1/1,2/2}{1/2,2/2,0/1,1/1,1/0}$ but not of $\hspace{-1mm}\pattern{scale=0.6}{2}{1/1,2/2}{1/2,2/2,1/1,0/0,1/0}$, then $a_j=x_i$, and we modify the elements of $S_i$ in $\pi$ by the following cyclic shift:
    \[
    x_i \to a_1, \quad a_1 \to a_2, \quad a_2 \to a_3, \quad \dots, \quad a_{j-2} \to a_{j-1},\quad a_{j-1} \to x_i.
    \]
After this shift, $a_1 r_i$ becomes an occurrence of the pattern $\hspace{-1mm}\pattern{scale=0.6}{2}{1/1,2/2}{1/2,2/2,1/1,0/0,1/0}$, and we lost the occurrence $x_ir_i$ of the pattern $\pattern{scale=0.6}{2}{1/1,2/2}{1/2,2/2,0/1,1/1,1/0}$ in $\pi$. 

    \item If $x_ir_i$ is an occurrence of $\hspace{-1mm}\pattern{scale=0.6}{2}{1/1,2/2}{1/2,2/2,1/1,0/0,1/0}$ but not $\pattern{scale=0.6}{2}{1/1,2/2}{1/2,2/2,0/1,1/1,1/0}$, then $a_1=x_i$, and we modify the elements of $S_i$ in $\pi$ by the substitution:
    \[
    x_i \to a_j, \quad a_j \to a_{j-1}, \quad a_{j-1} \to a_{j-2},  \quad \dots, \quad  a_{2} \to a_{1}, \quad a_1 \to x_i
    \]
This transformation maps the occurrence $x_ir_i$ of $\hspace{-1mm}\pattern{scale=0.6}{2}{1/1,2/2}{1/2,2/2,1/1,0/0,1/0}$ to an occurrence $a_jr_i$ of $\pattern{scale=0.6}{2}{1/1,2/2}{1/2,2/2,0/1,1/1,1/0}$. 
\end{itemize}

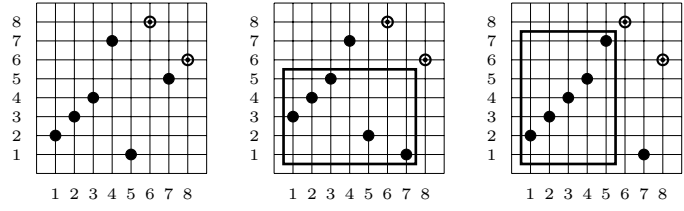
\begin{figure}[t]  
\begin{center}  

\begin{tabular}{ccc}
\begin{tikzpicture}[scale=0.25] 
\tikzset{    
  grid/.style={ draw, step=1cm, black!100, very thin }, 
  cell/.style={ draw, anchor=center, text centered },  
  graycell/.style={ fill=gray!50, draw=none, minimum width=1cm, minimum height=1cm, anchor=south west } 
}  

\draw[grid] (0,0) grid (9,9);

\foreach \x in {1,...,8} \node[anchor=north,font=\tiny] at (\x,-0.2) {\x};
\foreach \y in {1,...,8} \node[anchor=east,font=\tiny] at (-0.2,\y) {\y};

\filldraw[black] (1,2) circle (8pt);
\filldraw[black] (2,3) circle (8pt);
\filldraw[black] (3,4) circle (8pt);
\filldraw[black] (4,7) circle (8pt);
\filldraw[black] (5,1) circle (8pt);
\draw[thick] (6,8) circle (8pt); \draw[thick] (6,8) circle (2pt);
\draw[thick] (8,6) circle (8pt); \draw[thick] (8,6) circle (2pt);
\filldraw[black] (7,5) circle (8pt);

\end{tikzpicture}

&

\begin{tikzpicture}[scale=0.25] 
\tikzset{    
  grid/.style={ draw, step=1cm, black!100, very thin }, 
  cell/.style={ draw, anchor=center, text centered },  
  graycell/.style={ fill=gray!50, draw=none, minimum width=1cm, minimum height=1cm, anchor=south west } 
}  

\draw[grid] (0,0) grid (9,9);

\foreach \x in {1,...,8} \node[anchor=north,font=\tiny] at (\x,-0.2) {\x};
\foreach \y in {1,...,8} \node[anchor=east,font=\tiny] at (-0.2,\y) {\y};

\draw[line width=1pt] (0.5,0.5) rectangle (7.5,5.5);

\draw[thick] (6,8) circle (8pt); \draw[thick] (6,8) circle (2pt);
\draw[thick] (8,6) circle (8pt); \draw[thick] (8,6) circle (2pt);
\filldraw[black] (1,3) circle (8pt);
\filldraw[black] (2,4) circle (8pt);
\filldraw[black] (3,5) circle (8pt);
\filldraw[black] (4,7) circle (8pt);
\filldraw[black] (5,2) circle (8pt);
\filldraw[black] (7,1) circle (8pt);

\end{tikzpicture}
&

\begin{tikzpicture}[scale=0.25] 
\tikzset{    
  grid/.style={ draw, step=1cm, black!100, very thin }, 
  cell/.style={ draw, anchor=center, text centered },  
  graycell/.style={ fill=gray!50, draw=none, minimum width=1cm, minimum height=1cm, anchor=south west } 
}  

\draw[grid] (0,0) grid (9,9);

\foreach \x in {1,...,8} \node[anchor=north,font=\tiny] at (\x,-0.2) {\x};
\foreach \y in {1,...,8} \node[anchor=east,font=\tiny] at (-0.2,\y) {\y};

\draw[line width=1pt] (0.5,0.5) rectangle (5.5,7.5);

\draw[thick] (6,8) circle (8pt); \draw[thick] (6,8) circle (2pt);
\draw[thick] (8,6) circle (8pt); \draw[thick] (8,6) circle (2pt);
\filldraw[black] (1,2) circle (8pt);
\filldraw[black] (2,3) circle (8pt);
\filldraw[black] (3,4) circle (8pt);
\filldraw[black] (4,5) circle (8pt);
\filldraw[black] (5,7) circle (8pt);
\filldraw[black] (7,1) circle (8pt);

\end{tikzpicture}
\end{tabular}
\caption{The steps of implementing the mapping $f$ in the proof of Theorem~\ref{thm:class46} showing that $f(23471856)=23457816$.}\label{thm:class46-example}
\end{center}
\end{figure}

The steps of $f$ are illustrated in the example in Figure~\ref{thm:class46-example}. Note that in that example, the occurrence 56 of $\pattern{scale=0.6}{2}{1/1,2/2}{1/2,2/2,0/1,1/1,1/0}$ (resp., occurrence 18 of $\pattern{scale=0.6}{2}{1/1,2/2}{1/2,2/2,1/1,0/0,1/0}$) becomes the occurrence 16 of $\pattern{scale=0.6}{2}{1/1,2/2}{1/2,2/2,1/1,0/0,1/0}$ (resp., occurrence 78 of $\pattern{scale=0.6}{2}{1/1,2/2}{1/2,2/2,0/1,1/1,1/0}$).

We need to justify that a cyclic shift at each step cannot introduce or destroy any occurrences of the patterns other than those already discussed. This will show that $f$ is an involution since the cyclic shifts described above are inverses of each other.

First note that the transformation $f$ acts strictly on the values within the set $S_i$. 
\begin{itemize}
\item In the case when $x_ir_i$ is an occurrence of $\pattern{scale=0.6}{2}{1/1,2/2}{1/2,2/2,0/1,1/1,1/0}$ but not of $\pattern{scale=0.6}{2}{1/1,2/2}{1/2,2/2,1/1,0/0,1/0}$, there must be a unique element $a_i \in S_i$ such that $ a_1 \leq a_i < x_i$.  After modification, $x_i$ becomes the smallest element in $S_i$, $a_1r_i$ is an occurrence of $\pattern{scale=0.6}{2}{1/1,2/2}{1/2,2/2,0/0,1/1,1/0}$ but not of $\pattern{scale=0.6}{2}{1/1,2/2}{1/2,2/2,0/1,1/1,1/0}$, and there is no occurrence of $\pattern{scale=0.6}{2}{1/1,2/2}{1/2,2/2,0/1,1/1,1/0}$ involving $r_i$ because of the element $a_{i+1}$. 

Now, if $x_j r_j$ is an occurrence of the pattern $\pattern{scale=0.6}{2}{1/1,2/2}{1/2,2/2,0/1,1/1,1/0}$ (resp., $\pattern{scale=0.6}{2}{1/1,2/2}{1/2,2/2,1/1,0/0,1/0}$) to the right of $r_i$, then there are no elements between $x_j$ and $r_j$ (resp., less than $x_j$) to the left of $x_j$. Since the cyclic shift does not change the set $S_i$ (it only reorders the elements of $S_i$ in $\pi$), $x_j r_j$ remains an occurrence of the same pattern.

Finally, suppose that $x_j r_j$ is an occurrence of either of the patterns to the left of $r_i$, that is, $r_j > r_i$. If $x_j > r_i$, then clearly after the cyclic shift $x_j r_j$ remains an occurrence of the same pattern. On the other hand, if $x_j < r_i$, then we must have $x_j < x_i$ (otherwise, $x_i r_i$ cannot be an occurrence of $\pattern{scale=0.6}{2}{1/1,2/2}{1/2,2/2,0/1,1/1,1/0}$). In this case, the cyclic shift moves $x_j$ and all elements to its left, if any, upward, and it does not introduce any elements between $x_j$ and $r_j$ to the left of $x_j$. Moreover, no elements less than $x_j$ can appear to its left if there were none in $\pi$. Hence, $x_j r_j$ remains an occurrence of the same pattern.

\item The case when $x_i r_i$ is an occurrence of $\pattern{scale=0.6}{2}{1/1,2/2}{1/2,2/2,1/1,0/0,1/0}$ but not of $\pattern{scale=0.6}{2}{1/1,2/2}{1/2,2/2,0/1,1/1,1/0}$ can be analyzed similarly to the previous case and is therefore omitted.
\end{itemize}
The theorem is proved.
\end{proof}


\begin{thm}\label{thm:class49} The patterns in Class $49$ are equivalent.
\end{thm}

\begin{proof}
We will explain bijectively that  $\pattern{scale=0.6}{2}{1/1,2/2}{1/2,2/2,0/1,1/1} \sim \hspace{-1mm}\pattern{scale=0.6}{2}{1/1,2/2}{1/2,2/2,0/0,1/0}$.

Let $\pi = \pi_1 \pi_2 \cdots \pi_n \in S_n$, and write $\pi = A_1 r_1 A_2 r_2 \cdots A_\ell r_\ell$, where $r_1 > r_2 > \cdots > r_\ell$ is the sequence of right-to-left maxima of $\pi$. Note that if $ab$ is an occurrence of either pattern, then $b = r_i$ for some $i$, and $a$ must lie in $A_i$. Clearly, each $r_i$ can be involved in at most one occurrence of  $\pattern{scale=0.6}{2}{1/1,2/2}{1/2,2/2,0/1,1/1}$ or $\pattern{scale=0.6}{2}{1/1,2/2}{1/2,2/2,0/0,1/0}$.

The transformation $f: \mathcal{S}_n \to \mathcal{S}_n$ is obtained by considering $r_i$, for $i = \ell, \ell - 1, \ldots, 1$, one at a time. If $r_i$, is not involved in an occurrence of either pattern, or if $r_i$ is involved in occurrence of both patterns, do not change anything and proceed with $r_{i-1}$. Otherwise, we define the set $M_i$ as:
\[
M_i := \{ y \mid y \text{ is to the left of } r_{i-1}, \text{ and } y < r_i \}.
\]
By definition, $M_1=\emptyset$. Let the elements of $M_i$ be ordered as $a_1 < a_2 < \cdots < a_j$, where $j = |M_i|$. We perform the following steps if $j \geq 1$ (if $j = 0$, we do nothing):

\begin{itemize}

    \item Suppose that $a r_i$ is an occurrence of $\pattern{scale=0.6}{2}{1/1,2/2}{1/2,2/2,0/1,1/1}$, but there is no occurrence $b r_i$ of $\pattern{scale=0.6}{2}{1/1,2/2}{1/2,2/2,0/0,1/0}$. We modify $\pi$ by performing the following substitutions:
    \[
    a \to a_1,\, a_1 \to a_2,\, a_2 \to a_3,\, \dots,\, a_{j-1} \to a_j,\, a_j \to a.
    \]
     Then $ar_i$ becomes an occurrence of  $\pattern{scale=0.6}{2}{1/1,2/2}{1/2,2/2,0/0,1/0}$.
    \item Suppose that $a{r_i}$ is an occurrence of $\pattern{scale=0.6}{2}{1/1,2/2}{1/2,2/2,0/0,1/0}$, but there is no occurrence $\pattern{scale=0.6}{2}{1/1,2/2}{1/2,2/2,0/1,1/1}$ involving $r_i$. We modify $\pi$ by the following substitution:
    \[
    a \to a_j,\, a_j \to a_{j-1},\, \dots,\, a_2 \to a_1,\, a_1 \to a.
    \]
    Then $ar_i$ becomes an occurrence of $\pattern{scale=0.6}{2}{1/1,2/2}{1/2,2/2,0/1,1/1}$.
\end{itemize}

The steps of $f$ are illustrated in the example in Figure~\ref{thm:class49-example}. Note that in that example, $\pi$ has two occurrence of $\pattern{scale=0.6}{2}{1/1,2/2}{1/2,2/2,0/1,1/1}$ (48 and 67)  and two occurrence of $\pattern{scale=0.6}{2}{1/1,2/2}{1/2,2/2,0/0,1/0}$ (15 and 28),  and $f(\pi)$ has two occurrence of $\pattern{scale=0.6}{2}{1/1,2/2}{1/2,2/2,0/0,1/0}$ (28 and 17) and two occurrence of $\pattern{scale=0.6}{2}{1/1,2/2}{1/2,2/2,0/1,1/1}$ (45 and 68).

\begin{figure}[t]  
\begin{center}  

\begin{tabular}{cccc}

\begin{tikzpicture}[scale=0.25] 
\tikzset{    
  grid/.style={ draw, step=1cm, black!100, very thin }, 
  cell/.style={ draw, anchor=center, text centered },  
  graycell/.style={ fill=gray!50, draw=none, minimum width=1cm, minimum height=1cm, anchor=south west } 
}  

\draw[grid] (0,0) grid (9,9);

\foreach \x in {1,...,8} \node[anchor=north,font=\tiny] at (\x,-0.2) {\x};
\foreach \y in {1,...,8} \node[anchor=east,font=\tiny] at (-0.2,\y) {\y};

\draw[line width=1pt] (0.5,0.5) rectangle (4.5,4.5);
\draw[line width=1pt] (6.5,0.5) rectangle (7.5,4.5);

\draw[thick] (8,5) circle (8pt); \draw[thick] (8,5) circle (2pt);
\draw[thick] (5,7) circle (8pt); \draw[thick] (5,7) circle (2pt);
\draw[thick] (3,8) circle (8pt); \draw[thick] (3,8) circle (2pt);
\filldraw[black] (1,2) circle (8pt);
\filldraw[black] (2,4) circle (8pt);
\filldraw[black] (4,6) circle (8pt);
\filldraw[black] (6,3) circle (8pt);
\filldraw[black] (7,1) circle (8pt);

\end{tikzpicture}
&

\begin{tikzpicture}[scale=0.25] 
\tikzset{    
  grid/.style={ draw, step=1cm, black!100, very thin }, 
  cell/.style={ draw, anchor=center, text centered },  
  graycell/.style={ fill=gray!50, draw=none, minimum width=1cm, minimum height=1cm, anchor=south west } 
}  

\draw[grid] (0,0) grid (9,9);

\foreach \x in {1,...,8} \node[anchor=north,font=\tiny] at (\x,-0.2) {\x};
\foreach \y in {1,...,8} \node[anchor=east,font=\tiny] at (-0.2,\y) {\y};

\draw[line width=1pt] (0.5,0.5) rectangle (4.5,4.5);
\draw[line width=1pt] (6.5,0.5) rectangle (7.5,4.5);

\draw[thick] (8,5) circle (8pt); \draw[thick] (8,5) circle (2pt);
\draw[thick] (5,7) circle (8pt); \draw[thick] (5,7) circle (2pt);
\draw[thick] (3,8) circle (8pt); \draw[thick] (3,8) circle (2pt);
\filldraw[black] (1,1) circle (8pt);
\filldraw[black] (2,2) circle (8pt);
\filldraw[black] (4,6) circle (8pt);
\filldraw[black] (6,3) circle (8pt);
\filldraw[black] (7,4) circle (8pt);

\node[anchor=west] at ([xshift=5pt]current bounding box.east) {$=$};

\end{tikzpicture}
&

\hspace{-0.6cm}
\begin{tikzpicture}[scale=0.25] 
\tikzset{    
  grid/.style={ draw, step=1cm, black!100, very thin }, 
  cell/.style={ draw, anchor=center, text centered },  
  graycell/.style={ fill=gray!50, draw=none, minimum width=1cm, minimum height=1cm, anchor=south west } 
}  

\draw[grid] (0,0) grid (9,9);

\foreach \x in {1,...,8} \node[anchor=north,font=\tiny] at (\x,-0.2) {\x};
\foreach \y in {1,...,8} \node[anchor=east,font=\tiny] at (-0.2,\y) {\y};

\draw[line width=1pt] (0.5,0.5) rectangle (2.5,6.5);
\draw[line width=1pt] (3.5,0.5) rectangle (4.5,6.5);

\draw[thick] (8,5) circle (8pt); \draw[thick] (8,5) circle (2pt);
\draw[thick] (5,7) circle (8pt); \draw[thick] (5,7) circle (2pt);
\draw[thick] (3,8) circle (8pt); \draw[thick] (3,8) circle (2pt);
\filldraw[black] (1,1) circle (8pt);
\filldraw[black] (2,2) circle (8pt);
\filldraw[black] (4,6) circle (8pt);
\filldraw[black] (6,3) circle (8pt);
\filldraw[black] (7,4) circle (8pt);

\end{tikzpicture}
&

\begin{tikzpicture}[scale=0.25] 
\tikzset{    
  grid/.style={ draw, step=1cm, black!100, very thin }, 
  cell/.style={ draw, anchor=center, text centered },  
  graycell/.style={ fill=gray!50, draw=none, minimum width=1cm, minimum height=1cm, anchor=south west } 
}  

\draw[grid] (0,0) grid (9,9);

\foreach \x in {1,...,8} \node[anchor=north,font=\tiny] at (\x,-0.2) {\x};
\foreach \y in {1,...,8} \node[anchor=east,font=\tiny] at (-0.2,\y) {\y};

\draw[line width=1pt] (0.5,0.5) rectangle (2.5,6.5);
\draw[line width=1pt] (3.5,0.5) rectangle (4.5,6.5);

\draw[thick] (8,5) circle (8pt); \draw[thick] (8,5) circle (2pt);
\draw[thick] (5,7) circle (8pt); \draw[thick] (5,7) circle (2pt);
\draw[thick] (3,8) circle (8pt); \draw[thick] (3,8) circle (2pt);
\filldraw[black] (1,2) circle (8pt);
\filldraw[black] (2,6) circle (8pt);
\filldraw[black] (4,1) circle (8pt);
\filldraw[black] (6,3) circle (8pt);
\filldraw[black] (7,4) circle (8pt);

\end{tikzpicture}
\end{tabular}
\caption{The steps in implementing the mapping $f$ in the proof of Theorem~\ref{thm:class49}, showing that $f(24867315)=26817345$. The middle two diagrams represent the same permutation, and the rectangles indicate the regions in which the elements are moved by the cyclic shifts.}
\label{thm:class49-example}
\end{center}
\end{figure}
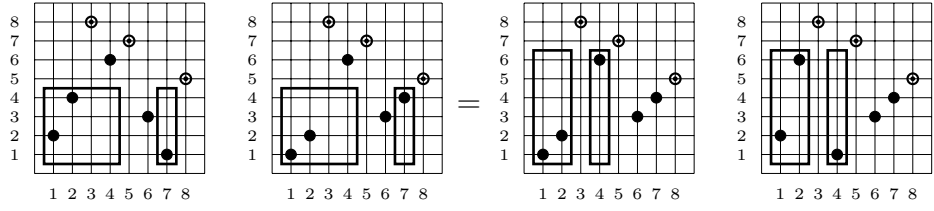

One can justify that a cyclic shift at each step cannot introduce or destroy any occurrences of the patterns, other than those already discussed, in an analogous way to the proof of Theorem~\ref{thm:class46}. The operations are very similar; the only difference is that, in Theorem~\ref{thm:class46}, the cyclic shift is applied to a larger block extending to the region between $r_{i-1}$ and $r_i$. We omit further details. The proof is complete.
\end{proof}


\begin{thm}\label{thm:class73} The patterns in Class $73$ are equivalent. For any pattern $p$ in the class, 
	\begin{align}\label{thm:class73-rec}
		s_{n,k}(p) = \sum_{i=1}^{n-1} \left[ i \cdot s_{n-1,k-1,i}(p) + (n-i) \cdot s_{n-1,k,i}(p) \right]
\end{align}
with  the initial conditions $s_{2,1}(p) = 1$, $s_{n,0}(p) = 1$ for all $n \in \mathbb{N}$, and $s_{0,k}(p) = 0$ for all  $k \ge 1$, where $s_{n,k,\ell}(p)$ is the number of $n$-permutations with $k$ occurrences of $p$ and $\ell$ right-to-left maxima satisfying
	\begin{align}\label{thm:class73-rec-2}
		s_{n,k,\ell}(p) = s_{n-1,k,\ell-1}(p) + \ell \cdot s_{n-1,k-1,\ell}(p) + (n-1-\ell) \cdot s_{n-1,k,\ell}(p)
	\end{align}
with the initial conditions $s_{0,0,0}(p) = 1$, $s_{n,k,0}(p) = 0$  for all $n,k \in \mathbb{N}$, $s_{0,k,\ell}(p) = 0$ if  $(k,\ell) \neq (0,0)$, and $s_{1,0,1}(p) = 1$.
\end{thm}

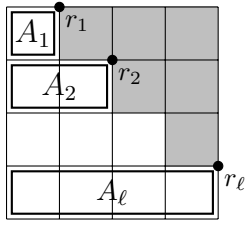
\begin{figure}
\begin{center}
\begin{tikzpicture}[scale=0.7]
  \tikzset{
    grid/.style={ draw, step=1cm, black!100, thin },
    graycell/.style={ fill=gray!50, draw=none, minimum width=1cm, minimum height=1cm, anchor=south west }
  }

  \fill[graycell] (1,3) rectangle (2,4);
  \fill[graycell] (2,3) rectangle (3,4);
  \fill[graycell] (3,3) rectangle (4,4);
  \fill[graycell] (2,2) rectangle (3,3);
  \fill[graycell] (3,2) rectangle (4,3);
  \fill[graycell] (3,1) rectangle (4,2);

  \draw[grid] (0,0) grid (4,4);

  \draw[thick] (0.1, 3.1) rectangle (0.9, 3.9);
  \node at (0.5, 3.5) {$A_1$};
  \draw[thick] (0.1, 2.1) rectangle (1.9, 2.9);
  \node at (1.0, 2.5) {$A_2$};
  \draw[thick] (0.1, 0.1) rectangle (3.9, 0.9);
  \node at (2.0, 0.5) {\small{$A_\ell$}};

  \filldraw[black] (1,4) circle (2.5pt);
  \node[anchor=west] at (0.9,3.7) {\footnotesize{$r_1$}};
  \filldraw[black] (2,3) circle (2.5pt);
  \node[anchor=west] at (1.9,2.7) {\footnotesize{$r_2$}};
  \filldraw[black] (4,1) circle (2.5pt);
  \node[anchor=west] at (3.9,0.7) {\footnotesize{$r_\ell$}};
\end{tikzpicture}
\caption{Related to the proofs of Theorems~\ref{thm:class73} and~\ref{thm:class75}.}\label{thm:class73-fig-1}
\end{center}
\end{figure}

\begin{proof} 
We begin with explaining bijectively that $\pattern{scale=0.6}{2}{1/1,2/2}{2/2,1/1,2/1}\sim\hspace{-1mm} \pattern{scale=0.6}{2}{1/1,2/2}{2/2,0/1,2/1}$. 

Let \(\pi = \pi_1 \pi_2 \cdots \pi_n \in S_n\) and \(r_1 > r_2 > \cdots > r_\ell\) be the sequence of right-to-left maxima in \(\pi\), which define the areas \(A_1, A_2, \ldots, A_\ell\) as shown in Figure~\ref{thm:class73-fig-1}. Keeping the values and positions of \(r_i\)'s, we modify \(\pi\) by reversing the elements in each block \(A_j\), \(1 \leq j \leq \ell\), while keeping the positions of the elements in \(A_j\). Let $\sigma$ be the resulting permutation. That is, if \(\pi_{i_1} \pi_{i_2} \ldots \pi_{i_m}\), \(1 \leq i_1 < i_2 < \cdots < i_m\) is a subsequence of \(\pi\) constituting the block \(A_j\), then the elements in $\sigma$  in positions \(i_1, i_2, \ldots, i_m\) will be, respectively, \(\pi_{i_m}, \pi_{i_{m-1}}, \ldots, \pi_{i_1}\). See Figure~\ref{thm:class73-example} for an example.

Note that in the example in  Figure~\ref{thm:class73-example}, the occurrences of $\pattern{scale=0.6}{2}{1/1,2/2}{2/2,1/1,2/1}$ (7(10) and 9(10) in the top block, and 15, 35, and 45 in the bottom block) and the occurrences of $\pattern{scale=0.6}{2}{1/1,2/2}{2/2,0/1,2/1}$ (8(10) and 9(10) in the top block, and 25 and 45 in the bottom block) in $\pi$ are mapped to one another. We will next prove that this is not a coincidence; under the bijection, if $ab$ is an occurrence of $\pattern{scale=0.6}{2}{1/1,2/2}{2/2,1/1,2/1}$ in $\pi$ then $ab$ is an occurrence of $\pattern{scale=0.6}{2}{1/1,2/2}{2/2,0/1,2/1}$ in $\sigma$, and vice versa. 

Note that if \( ab \) is an occurrence of either pattern, then \( b = r_i \) for some \( i \). Due to the shading of square (2,1), it follows that \( a \) must belong to \( A_i \) for that \( i \). Moreover, since there is no shading in squares (0,0), (0,2), (1,0), (1,2), and (2,0) in both patterns, the blocks \( A_j \) are independent in the sense that reversing the elements in any block cannot affect (create or destroy an occurrence of either of the patterns) another block. But then, we can consider each block independently, and we are done since clearly reversing elements inside the block interchanges occurrences of \( \pattern{scale=0.6}{1}{1/1}{0/1} \) and \( \pattern{scale=0.6}{1}{1/1}{1/0} \), and hence the respective occurrences of \( \pattern{scale=0.6}{2}{1/1,2/2}{2/2,1/1,2/1} \) and \( \pattern{scale=0.6}{2}{1/1,2/2}{2/2,0/1,2/1} \) in \( \pi \).

\begin{figure}[t]  
\begin{center}  

\begin{tabular}{cc}

\begin{tikzpicture}[scale=0.25] 
\tikzset{    
  grid/.style={ draw, step=1cm, black!100, very thin }, 
  cell/.style={ draw, anchor=center, text centered },  
  graycell/.style={ fill=gray!40, draw=none, minimum width=1cm, minimum height=1cm, anchor=south west } 
}  

\draw[grid] (0,0) grid (11,11);

\foreach \x in {1,...,10} \node[anchor=north,font=\tiny] at (\x,-0.2) {\x};
\foreach \y in {1,...,10} \node[anchor=east,font=\tiny] at (-0.2,\y) {\y};

\draw[line width=1pt] (0.3,6.3) rectangle (4.7,9.7);
\draw[line width=1pt] (0.3,0.3) rectangle (9.7,4.7);

\filldraw[black] (1,8) circle (8pt);
\filldraw[black] (2,2) circle (8pt);
\filldraw[black] (3,9) circle (8pt);
\filldraw[black] (4,7) circle (8pt);
\draw[thick] (5,10) circle (8pt); \draw[thick] (5,10) circle (2pt);
\draw[thick] (6,6) circle (8pt); \draw[thick] (6,6) circle (2pt);
\draw[thick] (10,5) circle (8pt); \draw[thick] (10,5) circle (2pt);
\filldraw[black] (7,4) circle (8pt);
\filldraw[black] (8,3) circle (8pt);
\filldraw[black] (9,1) circle (8pt);

\end{tikzpicture}

&

\begin{tikzpicture}[scale=0.25] 
\tikzset{    
  grid/.style={ draw, step=1cm, black!100, very thin }, 
  cell/.style={ draw, anchor=center, text centered },  
  graycell/.style={ fill=gray!40, draw=none, minimum width=1cm, minimum height=1cm, anchor=south west } 
}  

\draw[grid] (0,0) grid (11,11);

\foreach \x in {1,...,10} \node[anchor=north,font=\tiny] at (\x,-0.2) {\x};
\foreach \y in {1,...,10} \node[anchor=east,font=\tiny] at (-0.2,\y) {\y};

\draw[line width=1pt] (0.3,6.3) rectangle (4.7,9.7);
\draw[line width=1pt] (0.3,0.3) rectangle (9.7,4.7);

\draw[thick] (5,10) circle (8pt); \draw[thick] (5,10) circle (2pt);
\draw[thick] (10,5) circle (8pt); \draw[thick] (10,5) circle (2pt);
\draw[thick] (6,6) circle (8pt); \draw[thick] (6,6) circle (2pt);
\filldraw[black] (1,7) circle (8pt);
\filldraw[black] (3,9) circle (8pt);
\filldraw[black] (4,8) circle (8pt);
\filldraw[black] (2,1) circle (8pt);
\filldraw[black] (7,3) circle (8pt);
\filldraw[black] (8,4) circle (8pt);
\filldraw[black] (9,2) circle (8pt);

\end{tikzpicture}
\end{tabular}
\end{center}
\caption{The permutation $\pi=8297(10)64315$ goes to $\sigma=7198(10)63425$ under the bijection in the proof of Theorem~\ref{thm:class73}.}\label{thm:class73-example}
\end{figure}
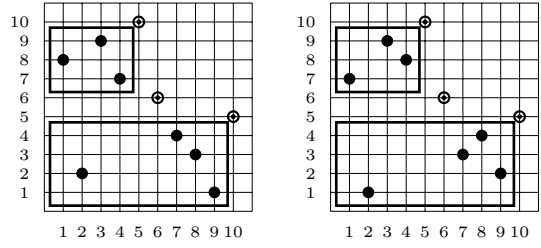

We now find the distribution of the pattern \( p = \pattern{scale=0.6}{2}{1/1,2/2}{1/1,1/2,2/2} \), which is equivalent to \( \pattern{scale=0.6}{2}{1/1,2/2}{2/2,1/1,2/1} \) by taking the inverse.  The initial conditions for $s_{n,k}(p)$ and $s_{n,k,\ell}(p)$ are easy to check. To explain (\ref{thm:class73-rec-2}), we consider generating $\pi\in S_n$ by inserting the new smallest element 1 in a permutation $\pi'\in S_{n-1}$. Note that if \( ab \) is an occurrence of \( p \) in $\pi'$, then \( b = r_i \) for some \( i \), where $r_i$ is a right-to-left maximum in $\pi'$. In (\ref{thm:class73-rec-2}):
\begin{itemize}
\item $s_{n-1,k,\ell-1}(p)$ corresponds to inserting 1 in the rightmost position in $\pi'$, thereby creating a new right-to-left maximum. Clearly, no new occurrence of $p$ is created, and no existing occurrence of $p$ in $\pi'$ is affected. 
\item $\ell \cdot s_{n-1,k-1,\ell}(p)$ counts the possibilities of picking one of $r_i$'s in $\pi'$ (in $\ell$ ways) and inserting 1 just before it, that preserves the number of right-to-left maxima, creates an extra occurrence of $p$, and does not affect any other occurrences of $p$ in $\pi'$.
\item \((n-1-\ell) \cdot s_{n-1,k,\ell}(p)\) gives the remaining possibilities to insert 1: not at the end of \(\pi'\) and not next to an \(r_i\) to the left. In each case, because of the shaded square (1,1), no new occurrence of \( p \) can be created (since \( 1r_i \) cannot be an occurrence), and no other occurrence of \( p \) in \(\pi'\) is affected.
\end{itemize}
The recurrence \eqref{thm:class73-rec-2} is explained. The recurrence \eqref{thm:class73-rec} now follows by observing that inserting the new minimal element immediately to the left of a right-to-left maximum creates a new occurrence of \( p \), while inserting it in any other place does not change the number of occurrences of \( p \); clearly, the number \( i \) of right-to-left maxima in $\pi'\in S_{n-1}$ is between 1 and \( n-1 \). Our proof is complete.
\end{proof}


\begin{thm}\label{thm:class75} The patterns in Class $75$ are equivalent.
\end{thm}

\begin{proof}
The equivalence 
$\pattern{scale=0.6}{2}{1/1,2/2}{0/1,1/1,2/2} \sim \hspace{-1mm} \pattern{scale=0.6}{2}{1/1,2/2}{0/0,1/0,2/2}$, 
and in fact the joint equidistribution of these patterns, follows from the observation that any permutation contains the same number of occurrences of them. Indeed, since the square $(2,2)$ is shaded, the second element in any occurrence of either pattern must be a right-to-left maximum. Now fix a right-to-left maximum $r$. If there are no elements to the left of $r$ that are smaller than $r$, then $r$ does not contribute any occurrences of either pattern. If there is exactly one element $x<r$ to the left of $r$, then $xr$ is an occurrence of both patterns. Finally, if there are at least two elements to the left of $r$ that are smaller than $r$, let $a$ and $b$ be the minimum and maximum among them, respectively. Then $ar$ (resp., $br$) is an occurrence of 
$\hspace{-1mm}\pattern{scale=0.6}{2}{1/1,2/2}{0/0,1/0,2/2}$ 
(resp., 
$\pattern{scale=0.6}{2}{1/1,2/2}{0/1,1/1,2/2}$), 
and $r$ cannot be involved in any other occurrences of these patterns, which completes the proof.

Next, we will explain bijectively that  $\pattern{scale=0.6}{2}{1/1,2/2}{1/2,2/2,0/1}\sim \hspace{-1mm}\pattern{scale=0.6}{2}{1/1,2/2}{1/2,2/2,0/0}$.

Let $\pi = \pi_1 \pi_2 \cdots \pi_n \in S_n$, and let $r_1 > r_2 > \cdots > r_\ell$ be the sequence of right-to-left maxima in $\pi$, which define the areas $A_1, A_2, \ldots, A_\ell$ as shown in Figure~\ref{thm:class73-fig-1}. We define the set $M_i$ as
\[
M_i := \{ y \mid y \text{ is to the left of } r_{i-1} \text{ and } y < r_i \}.
\]
By definition, $M_1 = \emptyset$. Let the elements of $M_i$ be ordered as $a_1 < a_2 < \cdots < a_j$, where $j = |M_i|$. We also define the set $B_i := A_i \setminus M_i$, so that $B_i$ consists of all elements in $\pi$ between $r_{i-1}$ and $r_i$. Note that if $ab$ is an occurrence of either pattern, then $b = r_i$ for some $i$, and $a$ must lie in $B_i$. Moreover, each $r_i$ can be involved in multiple occurrences of the patterns  $\pattern{scale=0.6}{2}{1/1,2/2}{1/2,2/2,0/1}$ and $\pattern{scale=0.6}{2}{1/1,2/2}{1/2,2/2,0/0}$.

The transformation $f: \mathcal{S}_n \to \mathcal{S}_n$ is obtained by considering $r_i$, for $i = \ell, \ell - 1, \ldots, 1$, one at a time. 
 For each $r_i$, we perform the following two steps sequentially:
                   
\begin{itemize}
    \item \textbf{Step 1 (First Complementation):} We apply a complementation to the elements in $A_i$. Specifically, we extract the values of these elements, reverse their relative order (i.e., mapping the minimum to the maximum, the second minimum to the second maximum, etc.), and assign them back to their original positions.
    
    \item \textbf{Step 2 (Second Complementation):} We apply a second complementation to the elements in $M_i$.
\end{itemize}

By sequentially applying Step 1 and Step 2 for each $r_i,r_{i-1}, \dots, r_1$, we obtain the transformed permutation $f(\pi)$.

In the example in Figure~\ref{thm:class75-example}, we consider $\pi = 23784516\in S_8$. We apply the procedure as follows: $\pi = \uline{23}78\uline{451}6 
\xrightarrow{\text{complement } 1,2,3,4,5} 
\uline{43}782156 
\xrightarrow{\text{complement } 3,4} 
\uline{347}82156 
\xrightarrow{\text{complement } 3,4,7} 
74382156 = f(\pi)$. Note that the occurrences $28, 38, 78, 46, 56$ of 
$\pattern{scale=0.6}{2}{1/1,2/2}{1/2,2/2,0/1}$ in $\pi$ are replaced by the occurrences 
$78, 48, 38, 26, 16$ of 
$\pattern{scale=0.6}{2}{1/1,2/2}{1/2,2/2,0/0}$ in $f(\pi)$. 
Similarly, the occurrences $28$ and $16$ of 
$\pattern{scale=0.6}{2}{1/1,2/2}{1/2,2/2,0/0}$ in $\pi$ are replaced by the occurrences 
$78$ and $58$ of 
$\pattern{scale=0.6}{2}{1/1,2/2}{1/2,2/2,0/1}$ in $f(\pi)$.
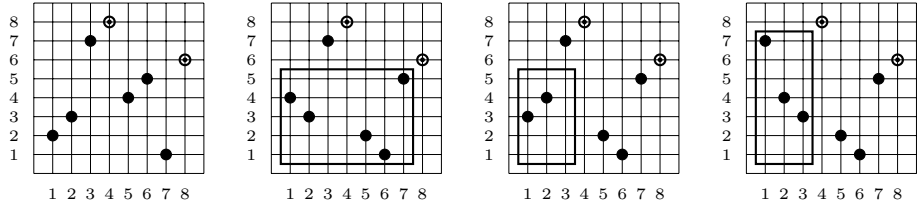
\begin{figure}[t]  
\begin{center}  

\begin{tabular}{cccc}

\begin{tikzpicture}[scale=0.25] 
\tikzset{    
  grid/.style={ draw, step=1cm, black!100, very thin }, 
  cell/.style={ draw, anchor=center, text centered },  
  graycell/.style={ fill=gray!50, draw=none, minimum width=1cm, minimum height=1cm, anchor=south west } 
}  

\draw[grid] (0,0) grid (9,9);

\foreach \x in {1,...,8} \node[anchor=north,font=\tiny] at (\x,-0.2) {\x};
\foreach \y in {1,...,8} \node[anchor=east,font=\tiny] at (-0.2,\y) {\y};

\filldraw[black] (1,2) circle (8pt);
\filldraw[black] (2,3) circle (8pt);
\filldraw[black] (3,7) circle (8pt);
\draw[thick] (4,8) circle (8pt); \draw[thick] (4,8) circle (2pt);
\filldraw[black] (5,4) circle (8pt);
\filldraw[black] (6,5) circle (8pt);
\filldraw[black] (7,1) circle (8pt);
\draw[thick] (8,6) circle (8pt); \draw[thick] (8,6) circle (2pt);

\end{tikzpicture}

&

\begin{tikzpicture}[scale=0.25] 
\tikzset{    
  grid/.style={ draw, step=1cm, black!100, very thin }, 
  cell/.style={ draw, anchor=center, text centered },  
  graycell/.style={ fill=gray!50, draw=none, minimum width=1cm, minimum height=1cm, anchor=south west } 
}  

\draw[grid] (0,0) grid (9,9);
\draw[thick] (0.5, 0.5) rectangle (7.5, 5.5);
\foreach \x in {1,...,8} \node[anchor=north,font=\tiny] at (\x,-0.2) {\x};
\foreach \y in {1,...,8} \node[anchor=east,font=\tiny] at (-0.2,\y) {\y};

\draw[line width=1.5pt] (4,8) rectangle (4,8);

\draw[thick] (4,8) circle (8pt); \draw[thick] (4,8) circle (2pt);
\draw[thick] (8,6) circle (8pt); \draw[thick] (8,6) circle (2pt);
\filldraw[black] (7,5) circle (8pt);
\filldraw[black] (6,1) circle (8pt);
\filldraw[black] (5,2) circle (8pt);
\filldraw[black] (3,7) circle (8pt);
\filldraw[black] (1,4) circle (8pt);
\filldraw[black] (2,3) circle (8pt);

\end{tikzpicture}

&

\begin{tikzpicture}[scale=0.25] 
\tikzset{    
  grid/.style={ draw, step=1cm, black!100, very thin }, 
  cell/.style={ draw, anchor=center, text centered },  
  graycell/.style={ fill=gray!50, draw=none, minimum width=1cm, minimum height=1cm, anchor=south west } 
}  

\draw[grid] (0,0) grid (9,9);
\draw[thick] (0.5, 0.5) rectangle (3.5, 5.5);
\foreach \x in {1,...,8} \node[anchor=north,font=\tiny] at (\x,-0.2) {\x};
\foreach \y in {1,...,8} \node[anchor=east,font=\tiny] at (-0.2,\y) {\y};

\draw[line width=1.5pt] (4,8) rectangle (4,8);

\draw[thick] (4,8) circle (8pt); \draw[thick] (4,8) circle (2pt);
\draw[thick] (8,6) circle (8pt); \draw[thick] (8,6) circle (2pt);
\filldraw[black] (7,5) circle (8pt);
\filldraw[black] (6,1) circle (8pt);
\filldraw[black] (5,2) circle (8pt);
\filldraw[black] (3,7) circle (8pt);
\filldraw[black] (1,3) circle (8pt);
\filldraw[black] (2,4) circle (8pt);

\end{tikzpicture}

&

\begin{tikzpicture}[scale=0.25] 
\tikzset{    
  grid/.style={ draw, step=1cm, black!100, very thin }, 
  cell/.style={ draw, anchor=center, text centered },  
  graycell/.style={ fill=gray!50, draw=none, minimum width=1cm, minimum height=1cm, anchor=south west } 
}  

\draw[grid] (0,0) grid (9,9);
\draw[thick] (0.5, 0.5) rectangle (3.5, 7.5);

\foreach \x in {1,...,8} \node[anchor=north,font=\tiny] at (\x,-0.2) {\x};
\foreach \y in {1,...,8} \node[anchor=east,font=\tiny] at (-0.2,\y) {\y};

\draw[thick] (4,8) circle (8pt); \draw[thick] (4,8) circle (2pt);
\draw[thick] (8,6) circle (8pt); \draw[thick] (8,6) circle (2pt);
\filldraw[black] (7,5) circle (8pt);
\filldraw[black] (6,1) circle (8pt);
\filldraw[black] (5,2) circle (8pt);
\filldraw[black] (1,7) circle (8pt);
\filldraw[black] (2,4) circle (8pt);
\filldraw[black] (3,3) circle (8pt);

\end{tikzpicture}
\end{tabular}
\caption{The permutation $\pi = 23784516$ is mapped to $f(\pi) = 74382156$ under the bijection in the proof of Theorem~\ref{thm:class75}. The rectangles indicate the areas where complementation occurs.}
\label{thm:class75-example}
\end{center}
\end{figure}

We now justify that $f(\pi) = \sigma_1 \sigma_2 \cdots \sigma_n$ only swaps the occurrences of 
$\pattern{scale=0.6}{2}{1/1,2/2}{1/2,2/2,0/1}$ and 
$\pattern{scale=0.6}{2}{1/1,2/2}{1/2,2/2,0/0}$ involving the set $B_i$.  If $\pi_j r_i$ is an occurrence of 
$\pattern{scale=0.6}{2}{1/1,2/2}{1/2,2/2,0/1}$, then there is no element to the left of $\pi_j$ whose value lies between $\pi_j$ and $r_i$. After the first complementation in $A_i$, the empty region above $\pi_j$ to its left is transferred to the empty region below $\sigma_j$ to its left, thereby turning the occurrence $\pi_j r_i$ of 
$\pattern{scale=0.6}{2}{1/1,2/2}{1/2,2/2,0/1}$ into an occurrence $\sigma_j r_i$ of 
$\pattern{scale=0.6}{2}{1/1,2/2}{1/2,2/2,0/0}$. Similarly, any occurrence of 
$\pattern{scale=0.6}{2}{1/1,2/2}{1/2,2/2,0/0}$ involving elements in $B_i$ is transformed into an occurrence of 
$\pattern{scale=0.6}{2}{1/1,2/2}{1/2,2/2,0/1}$.

On the other hand, the second complementation ensures that the relative order of the elements in $M_i$ in $\sigma$ is the same as that in $M_i$ in $\pi$. This implies that any occurrence of either pattern in $\pi$ to the left of $r_{i-1}$ remains an occurrence in $\sigma$ to the left of $r_{i-1}$, and that no new occurrences of the patterns are created in this region.

Finally, it is clear that the complementation operations do not affect any occurrence of the patterns involving $r_j$ for $j > i$, which completes the proof.
\end{proof}

\section{Additional enumeration of avoiders}\label{avoiders-enum-sec}

In this section, we mainly enumerate avoidance classes corresponding to some equivalence classes arising from trivial bijections, which are presented in Table~\ref{tab-remaining}. However, we begin with two theorems that are not related to Table~\ref{tab-remaining}.

\begin{thm}\label{thm:class47} The g.f.\ for the number of permutations avoiding any pattern in Class $47$ is 
$$A(x)=\frac{F(x)}{1+x\sum_{n\geq 1}n!H_{n}x^n},$$ where $H_n = \sum_{k=1}^{n} \frac{1}{k}$ is the $n$-th Harmonic number. \end{thm}

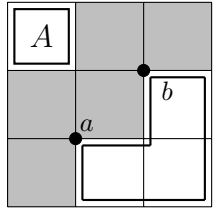
\begin{figure}
\begin{center}
\begin{tikzpicture}[scale=0.9]
  \tikzset{
    grid/.style={ draw, step=1cm, black!100, thin },
    graycell/.style={ fill=gray!50, draw=none, minimum width=1cm, minimum height=1cm, anchor=south west }
  }

  \fill[graycell] (1,2) rectangle (2,3);
  \fill[graycell] (2,2) rectangle (3,3);
  \fill[graycell] (0,1) rectangle (1,2);
  \fill[graycell] (1,1) rectangle (2,2);
  \fill[graycell] (0,0) rectangle (1,1);

  \draw[grid] (0,0) grid (3,3);

  \draw[thick] (0.1, 2.1) rectangle (0.9, 2.9);
  \node at (0.5, 2.5) {$A$};
  \draw[thick] (1.1,0.9) -- (1.1,0.1);
  \draw[thick] (1.1,0.1) -- (2.9,0.1);
  \draw[thick] (2.9,0.1) -- (2.9,1.9);
  \draw[thick] (2.9,1.9) -- (2.1,1.9);
  \draw[thick] (2.1,1.9) -- (2.1,0.9);
  \draw[thick] (2.1,0.9) -- (1.1,0.9);

  \filldraw[black] (1,1) circle (2.5pt);
  \node[anchor=west] at (0.9,1.2) {\footnotesize{$a$}};
  \filldraw[black] (2,2) circle (2.5pt);
  \node[anchor=west] at (2.1,1.7) {\footnotesize{$b$}};
\end{tikzpicture}
\end{center}
\caption{Related to the proof of Theorem~\ref{thm:class47}.}\label{thm:class47-fig-2}
\end{figure}

\begin{proof}
We enumerate permutations avoiding the pattern $\pattern{scale = 0.6}{2}{1/1,2/2}{0/0,0/1,1/1,1/2,2/2}$ in Class 47. We claim that
\begin{equation}\label{class47:A(x)-rec}
F(x)=A(x) + A(x)\sum_{n\geq 2}(n-1)!H_{n-1}x^n.
\end{equation}
Indeed, each permutation (counted by $F(x)$ on the LHS) either avoids the pattern (and hence is counted by $A(x)$) or contains it. Among all occurrences $ab$ of the pattern in $\pi$, select the one with the largest possible $b$, as shown in Figure~\ref{thm:class47-fig-2}. Then the block $A$ must avoid $\pattern{scale = 0.6}{2}{1/1,2/2}{0/0,0/1,1/1,1/2,2/2}$, while the rest of the permutation is an arbitrary permutation with two left-to-right maxima (given by the elements $a$ and $b$). It is well known~\cite{StanleyEC1} that permutations with $k$ left-to-right maxima are in bijection with permutations with $k$ cycles, so these permutations are counted by $c(n,2)$, the unsigned Stirling numbers of the first kind. Their ordinary g.f.\ is $\sum_{n\geq 2}(n-1)!H_{n-1}x^n$. This yields~\eqref{class47:A(x)-rec}, and hence the formula for $A(x)$.
\end{proof}

\begin{thm}\label{thm:class54} The g.f.\ for the number of permutations avoiding any pattern in Class $54$ is 
$$A(x)=F(x)-1-x-\frac{xF(x)}{1+xF(x)}\sum_{n\ge 1} n!H_n x^n,$$ where $H_n = \sum_{k=1}^{n} \frac{1}{k}$ is the $n$-th Harmonic number. \end{thm}

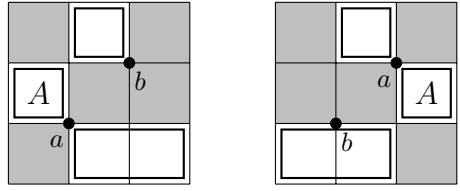
\begin{figure}
\begin{center}
\begin{tabular}{ccc}
\begin{tikzpicture}[scale=0.8]
  \tikzset{
    grid/.style={ draw, step=1cm, black!100, thin },
    graycell/.style={ fill=gray!50, draw=none, minimum width=1cm, minimum height=1cm, anchor=south west }
  }

  \fill[graycell] (0,2) rectangle (1,3);
  \fill[graycell] (1,1) rectangle (2,2);
  \fill[graycell] (2,1) rectangle (3,2);
  \fill[graycell] (2,2) rectangle (3,3);
  \fill[graycell] (0,0) rectangle (1,1);

  \draw[grid] (0,0) grid (3,3);

  \draw[thick] (0.1, 1.1) rectangle (0.9, 1.9);
  \node at (0.5, 1.5) {$A$};
  \draw[thick] (1.1, 2.1) rectangle (1.9, 2.9);
  \draw[thick] (1.1, 0.1) rectangle (2.9, 0.9);

  \filldraw[black] (1,1) circle (2.5pt);
  \node[anchor=west] at (0.5,0.7) {\footnotesize{$a$}};
  \filldraw[black] (2,2) circle (2.5pt);
  \node[anchor=west] at (1.9,1.7) {\footnotesize{$b$}};
\end{tikzpicture}

&

\ \ \

&

\begin{tikzpicture}[scale=0.8]
  \tikzset{
    grid/.style={ draw, step=1cm, black!100, thin },
    graycell/.style={ fill=gray!50, draw=none, minimum width=1cm, minimum height=1cm, anchor=south west }
  }

  \fill[graycell] (0,2) rectangle (1,3);
  \fill[graycell] (0,1) rectangle (1,2);
  \fill[graycell] (1,1) rectangle (2,2);
  \fill[graycell] (2,2) rectangle (3,3);
  \fill[graycell] (2,0) rectangle (3,1);

  \draw[grid] (0,0) grid (3,3);

  \draw[thick] (2.1, 1.1) rectangle (2.9, 1.9);
  \node at (2.5, 1.5) {$A$};
  \draw[thick] (1.1, 2.1) rectangle (1.9, 2.9);
  \draw[thick] (0.1, 0.1) rectangle (1.9, 0.9);

  \filldraw[black] (1,1) circle (2.5pt);
  \node[anchor=west] at (0.9,0.7) {\footnotesize{$b$}};
  \filldraw[black] (2,2) circle (2.5pt);
  \node[anchor=west] at (1.5,1.7) {\footnotesize{$a$}};
\end{tikzpicture}
\end{tabular}
\end{center}
\caption{Related to the proof of Theorem~\ref{thm:class54}.}
\end{figure}

\begin{proof}
We first enumerate permutations avoiding the pattern $p=\pattern{scale = 0.6}{2}{1/1,2/2}{0/0,0/2,1/1,2/1,2/2}$ in Class 54. We claim that for $n\geq 2$,
\begin{equation}\label{class54:A(x)-rec}
s_{n,0}(p)=n! - \sum_{i=0}^{n-2}[x^i]\frac{F(x)}{1+xF(x)}\sum_{k=1}^{n-1-i}k \cdot c(n-1-i,k).
\end{equation}
Here, $[x^i]$ denotes the coefficient of $x^i$ in the subsequent generating function, and $c(n,k)$ is the unsigned Stirling number of the first kind, counting permutations of length $n$ with $k$ right-to-left maxima.

We compute the number of $n$-permutations containing $p$ and then subtract this from $n!$. Suppose that $ab$ is an occurrence of $p$ in $\pi \in S_n$, as shown in the left-hand diagram of Figure~\ref{thm:class54}, where we assume that $a$ is chosen as small as possible. Then block $A$ must avoid the pattern $\pattern{scale = 0.6}{1}{1/1}{0/0,1/1}$, and if it contains $i$ elements, then by Theorem~\ref{short} it can be formed in $[x^i]\frac{F(x)}{1+xF(x)}$ ways. Since $b$ is a right-to-left maximum in $\pi$, we can think of constructing $\pi$ by appending the block $Aa$ from the left, immediately below $b$, to a permutation of length $n-i-1$ with $k$ right-to-left maxima (where $b$ is one of them). This construction is clearly reversible, so every such $\pi$ is generated exactly once. Since $b$ can be chosen in $k$ ways, this completes the justification of (\ref{class54:A(x)-rec}), and the formula for $A(x)$ follows from algebraic manipulations.

As for the pattern $\pattern{scale = 0.6}{2}{1/1,2/2}{0/1,0/2,1/1,2/0,2/2}$, which is also in Class 54 but is not trivially Wilf-equivalent to the pattern $p$, essentially the same argument can be repeated using the right-hand diagram in Figure~\ref{thm:class54}, with $ba$ now representing an occurrence of $\pattern{scale = 0.6}{2}{1/1,2/2}{0/1,0/2,1/1,2/0,2/2}$ and $b$ playing the role of a left-to-right maximum. We therefore omit the further details.
\end{proof}

\begin{thm}\label{thm:avoidance-1} For a pattern $p$ in Classes $15$, $16$, $17$, $18$, $38$, $39$, $49$, $50$, $51$, $72$--$75$, $87$, $94$, $95$, and $102$, $s_{n,0}(p)=1$ for $n\geq 0$. \end{thm}

\begin{proof}
It is easy to see that only decreasing permutations avoid the respective representatives of the classes 
$\pattern{scale=0.6}{2}{1/1,2/2}{1/1,2/2}$, 
$\pattern{scale=0.6}{2}{1/1,2/2}{0/1,2/2}$, 
$\pattern{scale=0.6}{2}{1/1,2/2}{1/1}$, 
$\pattern{scale=0.6}{2}{1/1,2/2}{0/0}$, 
$\pattern{scale=0.6}{2}{1/1,2/2}{1/0,1/1,1/2,2/2}$, 
$\pattern{scale=0.6}{2}{1/1,2/2}{0/1,1/1,1/2}$, 
$\pattern{scale=0.6}{2}{1/1,2/2}{0/1,2/1}$, and 
$\pattern{scale=0.6}{2}{1/1,2/2}{1/2}$.
\end{proof}

\begin{thm}\label{thm:avoidance-(n-1)!} For a pattern $p$ in Classes $4$, $6$--$12$, $42$, $43$, $45$, $64$--$66$, $68$, $89$, $90$, $92$, $93$, $100$, and $101$, $s_{n,0}(p)=(n-1)!$ for $n\geq 1$. \end{thm}

\begin{proof}
Avoidance for these classes was established in \cite{Hilmarsson2015Wilf}, where various symmetries and the {\em Shading Lemma} were used to obtain all classes enumerated by $(n-1)!$. Also, proofs for Classes 45 and 68 can be found in Theorem~\ref{thm:class45} and \cite{KZ2019}, respectively, while proofs for Classes 8, 10 and 64 appear in \cite[Prop.~16]{Hilmarsson2015Wilf}.
\end{proof}

\begin{thm}\label{thm:avoidance-(n!/2)} For a pattern $p$ in Classes $25$, $28$, $33$, $63$, $85$, and $104$, $s_{n,0}(p)=n!/2$ for $n\geq 2$. \end{thm}

\begin{proof}
Avoidance for these classes was established in \cite{Hilmarsson2015Wilf}, where various symmetries and the Shading Lemma were used to obtain all classes enumerated by $n!/2$.  Also, the distributions for Classes 25, 33, and 63 were established by us in Theorems~\ref{thm:class25}, \ref{thm:class33}, and \ref{thm:class63}, respectively.
\end{proof}

\begin{thm}\label{thm:avoidance-32-62-84}
For a pattern $p$ in Classes $32$, $62$, and $84$, and $n\geq 2$,
\[
s_{n,0}(p)=n!-\sum_{i=1}^{n-1}\frac{(n-1)!}{i}.
\]
\end{thm}

\begin{proof}
Avoidance for these classes was established in \cite[Prop.~24]{Hilmarsson2015Wilf}. The avoidance for Class~62 is also given in Theorem~\ref{thm:class62}, and the distribution for Class~32 is given in Theorem~\ref{thm:class32}.
\end{proof}

\begin{thm}\label{thm:avoidance-24-61-86-105}
For a pattern $p$ in Classes $24$, $61$, $86$, and $105$ and $n\geq 1$,
\[
s_{n,0}(p)=\sum_{k=0}^{n}(-1)^k(n-k+1)\frac{n!}{k!}.
\]
\end{thm}

\begin{proof}
Avoidance for these classes was established in \cite{Hilmarsson2015Wilf,Parv2009}. The distribution for Class~61 is given in Theorem~\ref{dis-patterns-61}.
\end{proof}

\begin{thm}\label{thm:avoidance-48-67}
For a pattern $p$ in Classes $48$ and $67$, we have $s_{n,0}(p)=n\cdot s_{n-1,0}(p)-s_{n-2,0}(p)$, $s_{-1,0}(p)=0$, and $s_{0,0}(p)=1$.
\end{thm}

\begin{proof}
Avoidance for these classes was established in \cite[Prop. 30 and 31]{Hilmarsson2015Wilf}. 
\end{proof}

\begin{thm}\label{thm:avoidance-19-21-80-97-103}
For a pattern $p$ in Classes $19$, $21$, $80$, $97$, and $103$, the g.f.\ is
\[
A(x) = \frac{(1+x)F(x)}{1+xF(x)}.
\]
\end{thm}

\begin{proof}
Avoidance for these classes was established in \cite[Prop.~23]{Hilmarsson2015Wilf}, where various symmetries and the Shading Lemma were used to obtain all classes enumerated by this sequence. The generating function was given by us in Theorem~\ref{thm:class80}, along with the full distribution for Class~80. The distribution for Class~103 was given in \cite{KZ2019}.
\end{proof}

\section{Concluding remarks}\label{conclud-sec}

In this paper, we have developed a near-complete classification of equidistribution equivalence classes for mesh patterns of length~2, showing that their number lies between 105 and 108. Based on our results, we conjecture that the exact number is 105.

\begin{conj}
There are exactly $105$ equivalence classes with respect to distribution for mesh patterns of length~$2$.
\end{conj}

As a consequence, we improved the known upper bound on the number of Wilf-classes from 56 to 49, reducing both the determination of the exact number of Wilf-classes and the complete classification of equidistribution classes to only three remaining conjectural cases listed in Table~\ref{tab-conject}, one of which is the following new conjecture:

\begin{conj}
Any two patterns belonging to the union of the following two sets are equivalent: $\left\{
\pattern{scale=0.6}{2}{1/1,2/2}{0/2,2/2,1/1,2/1,0/0},
\pattern{scale=0.6}{2}{1/1,2/2}{0/2,2/2,1/1,0/0,1/0},  
\pattern{scale=0.6}{2}{1/1,2/2}{1/2,2/2,1/1,0/0,2/0},  
\pattern{scale=0.6}{2}{1/1,2/2}{2/2,0/1,1/1,0/0,2/0}
\right\}$
and $\left\{
\pattern{scale=0.6}{2}{1/1,2/2}{0/2,2/2,0/1,1/1,2/0}, 
\pattern{scale=0.6}{2}{1/1,2/2}{0/2,1/2,1/1,0/0,2/0},  
\pattern{scale=0.6}{2}{1/1,2/2}{0/2,1/1,2/1,0/0,2/0},  
\pattern{scale=0.6}{2}{1/1,2/2}{0/2,2/2,1/1,1/0,2/0}
\right\}$.
\end{conj}

All bijective proofs presented in this paper establish the stronger property of joint equidistribution for the patterns under consideration. Nevertheless, a broader systematic study of joint equidistributions for mesh patterns of length~2 remains open and appears to be a promising direction for future research.

Finally, Table~\ref{tab:patterns_ref} highlights 19 remaining open avoidance problems and 72 open distribution problems. Further investigation of these unresolved cases would be both interesting and valuable for advancing our understanding of mesh pattern avoidance and distribution.

\section*{\bf Acknowledgments}

The work of the first and third authors was supported by the Tianjin Municipal Natural Science Foundation (No. 25JCYBJC00430).

\appendix
\section{Equivalence classes explained by trivial bijections}\label{trivial-appendix}
In Table~\ref{tab-remaining}, we present the ``trivial'' classes, namely equivalence classes that can be explained by simple bijections (symmetries).

\begin{table}[!t]
    \renewcommand{\arraystretch}{1.4}
    \begin{center}
        \begin{tabular}{|c|l||c|l||c|l|}
            \hline
            \multicolumn{1}{|c|}{\footnotesize{nr.}} & \multicolumn{1}{c||}{\footnotesize{trivial}} &
            \multicolumn{1}{c|}{\footnotesize{nr.}} & \multicolumn{1}{c||}{\footnotesize{trivial}} &
            \multicolumn{1}{c|}{\footnotesize{nr.}} & \multicolumn{1}{c|}{\footnotesize{trivial}} \\
            \hline
            \footnotesize{3} & \hspace{-4mm} $\pattern{scale=0.5}{2}{1/1,2/2}{1/2, 2/2, 0/1, 1/1, 2/1, 0/0, 1/0}$
            & \footnotesize{4}   & \hspace{-4mm} $\pattern{scale=0.5}{2}{1/1,2/2}{0/2, 1/1, 2/1}\hspace{-1mm}\pattern{scale=0.5}{2}{1/1,2/2}{0/2, 1/1, 1/0}\hspace{-1mm}\pattern{scale=0.5}{2}{1/1,2/2}{1/2, 1/1, 2/0}\hspace{-1mm}\pattern{scale=0.5}{2}{1/1,2/2}{0/1, 1/1, 2/0}$
            & \footnotesize{5}   & \hspace{-4mm} $\pattern{scale=0.5}{2}{1/1,2/2}{1/2, 1/1, 2/1}\hspace{-1mm}\pattern{scale=0.5}{2}{1/1,2/2}{0/1, 1/1, 1/0}$ \\ \hline
            \footnotesize{6} & \hspace{-4mm} $\pattern{scale=0.5}{2}{1/1,2/2}{0/2, 1/2, 1/1}\hspace{-1mm}\pattern{scale=0.5}{2}{1/1,2/2}{0/2, 0/1, 1/1}\hspace{-1mm}\pattern{scale=0.5}{2}{1/1,2/2}{1/1, 2/1, 2/0}\hspace{-1mm}\pattern{scale=0.5}{2}{1/1,2/2}{1/1, 1/0, 2/0}$
            & \footnotesize{8}   & \hspace{-4mm} $\pattern{scale=0.5}{2}{1/1,2/2}{0/2, 1/2, 2/2}\hspace{-1mm}\pattern{scale=0.5}{2}{1/1,2/2}{0/2, 0/1, 0/0}\hspace{-1mm}\pattern{scale=0.5}{2}{1/1,2/2}{2/2, 2/1, 2/0}\hspace{-1mm}\pattern{scale=0.5}{2}{1/1,2/2}{0/0, 1/0, 2/0}$
            & \footnotesize{9}   & \hspace{-4mm} $\pattern{scale=0.5}{2}{1/1,2/2}{0/2, 2/2, 0/0}\hspace{-1mm}\pattern{scale=0.5}{2}{1/1,2/2}{2/2, 0/0, 2/0}$ \\ \hline
            \footnotesize{10} & \hspace{-4mm} $\pattern{scale=0.5}{2}{1/1,2/2}{0/2, 2/2, 1/1}\hspace{-1mm}\pattern{scale=0.5}{2}{1/1,2/2}{0/2, 1/1, 0/0}\hspace{-1mm}\pattern{scale=0.5}{2}{1/1,2/2}{2/2, 1/1, 2/0}\hspace{-1mm}\pattern{scale=0.5}{2}{1/1,2/2}{1/1, 0/0, 2/0}$
            & \footnotesize{11}  & \hspace{-4mm} $\pattern{scale=0.5}{2}{1/1,2/2}{0/2, 2/1, 0/0}\hspace{-1mm}\pattern{scale=0.5}{2}{1/1,2/2}{0/2, 2/2, 1/0}\hspace{-1mm}\pattern{scale=0.5}{2}{1/1,2/2}{2/2, 0/1, 2/0}\hspace{-1mm}\pattern{scale=0.5}{2}{1/1,2/2}{1/2, 0/0, 2/0}$
            & \footnotesize{12}  & \hspace{-4mm} $\pattern{scale=0.5}{2}{1/1,2/2}{0/2, 2/2, 0/1}\hspace{-1mm}\pattern{scale=0.5}{2}{1/1,2/2}{0/2, 1/2, 0/0}\hspace{-1mm}\pattern{scale=0.5}{2}{1/1,2/2}{2/1, 0/0, 2/0}\hspace{-1mm}\pattern{scale=0.5}{2}{1/1,2/2}{2/2, 1/0, 2/0}$ \\ \hline
            \footnotesize{13} & \hspace{-4mm} $\pattern{scale=0.5}{2}{1/1,2/2}{2/2, 1/1, 0/0}$
            & \footnotesize{14}  & \hspace{-4mm} $\pattern{scale=0.5}{2}{1/1,2/2}{1/2, 2/1, 0/0}\hspace{-1mm}\pattern{scale=0.5}{2}{1/1,2/2}{2/2, 0/1, 1/0}$
            & \footnotesize{16}  & \hspace{-4mm} $\pattern{scale=0.5}{2}{1/1,2/2}{2/2, 1/1}\hspace{-1mm}\pattern{scale=0.5}{2}{1/1,2/2}{1/1, 0/0}$ \\ \hline
            \footnotesize{17} & \hspace{-4mm} $\pattern{scale=0.5}{2}{1/1,2/2}{2/2, 0/1}\hspace{-1mm}\pattern{scale=0.5}{2}{1/1,2/2}{1/2, 0/0}\hspace{-1mm}\pattern{scale=0.5}{2}{1/1,2/2}{2/1, 0/0}\hspace{-1mm}\pattern{scale=0.5}{2}{1/1,2/2}{2/2, 1/0}$
            & \footnotesize{18}  & \hspace{-4mm} $\pattern{scale=0.5}{2}{1/1,2/2}{2/2, 0/0}$
            & \footnotesize{19}  & \hspace{-4mm} $\pattern{scale=0.5}{2}{1/1,2/2}{1/2, 0/2, 2/0,2/1,1/1}\hspace{-1mm}\pattern{scale=0.5}{2}{1/1,2/2}{0/2, 1/1, 2/0,1/0,0/1}$ \\ \hline
            \footnotesize{20} & \hspace{-4mm} $\pattern{scale=0.5}{2}{1/1,2/2}{1/2,0/2,2/0,2/1,2/2}\hspace{-1mm}\pattern{scale=0.5}{2}{1/1,2/2}{1/0,0/1,2/0,0/0,0/2}$
            & \footnotesize{21}  & \hspace{-4mm} $\pattern{scale=0.5}{2}{1/1,2/2}{0/2,2/0,2/1,1/2,0/0}\hspace{-1mm}\pattern{scale=0.5}{2}{1/1,2/2}{2/0,0/1,1/0,0/2,2/2}$
            & \footnotesize{22}  & \hspace{-4mm} $\pattern{scale=0.5}{2}{1/1,2/2}{0/2,2/2,1/1,0/0,2/0}$ \\ \hline
            \footnotesize{23} & \hspace{-4mm} $\pattern{scale=0.5}{2}{1/1,2/2}{0/2,1/1,1/2,2/0}\hspace{-1mm}\pattern{scale=0.5}{2}{1/1,2/2}{0/1,0/2,1/1,2/0}\hspace{-1mm}\pattern{scale=0.5}{2}{1/1,2/2}{0/2,1/1,2/0,2/1}\hspace{-1mm}\pattern{scale=0.5}{2}{1/1,2/2}{0/2,1/0,1/1,2/0}$
            & \footnotesize{24}  & \hspace{-4mm} $\pattern{scale=0.5}{2}{1/1,2/2}{0/2,1/2,1/1,2/1}\hspace{-1mm}\pattern{scale=0.5}{2}{1/1,2/2}{0/2,0/1,1/1,1/0}\hspace{-1mm}\pattern{scale=0.5}{2}{1/1,2/2}{1/2,1/1,2/1,2/0}\hspace{-1mm}\pattern{scale=0.5}{2}{1/1,2/2}{0/1,1/1,1/0,2/0}$
            & \footnotesize{26}  & \hspace{-4mm} $\pattern{scale=0.5}{2}{1/1,2/2}{0/2,2/2,1/1,0/0}\hspace{-1mm}\pattern{scale=0.5}{2}{1/1,2/2}{2/2,1/1,0/0,2/0}$ \\ \hline
            \footnotesize{27} & \hspace{-4mm} $\pattern{scale=0.5}{2}{1/1,2/2}{0/2,2/2,1/1,2/0}\hspace{-1mm}\pattern{scale=0.5}{2}{1/1,2/2}{0/2,1/1,0/0,2/0}$
            & \footnotesize{28}  & \hspace{-4mm} $\pattern{scale=0.5}{2}{1/1,2/2}{0/2,2/2,0/1,2/0}\hspace{-1mm}\pattern{scale=0.5}{2}{1/1,2/2}{0/2,1/2,0/0,2/0}\hspace{-1mm}\pattern{scale=0.5}{2}{1/1,2/2}{0/2,2/1,0/0,2/0}\hspace{-1mm}\pattern{scale=0.5}{2}{1/1,2/2}{0/2,2/2,1/0,2/0}$
            & \footnotesize{29}  & \hspace{-4mm} $\pattern{scale=0.5}{2}{1/1,2/2}{0/2,2/2,0/0,2/0}$ \\ \hline
            \footnotesize{30} & \hspace{-4mm} $\pattern{scale=0.5}{2}{1/1,2/2}{0/2,1/2,2/1,0/0}\hspace{-1mm}\pattern{scale=0.5}{2}{1/1,2/2}{0/2,2/2,0/1,1/0}\hspace{-1mm}\pattern{scale=0.5}{2}{1/1,2/2}{1/2,2/1,0/0,2/0}\hspace{-1mm}\pattern{scale=0.5}{2}{1/1,2/2}{2/2,0/1,1/0,2/0}$
            & \footnotesize{34}  & \hspace{-4mm} $\pattern{scale=0.5}{2}{1/1,2/2}{0/2,2/2,1/1,2/1,0/0,1/0}\hspace{-1mm}\pattern{scale=0.5}{2}{1/1,2/2}{1/2,2/2,0/1,1/1,0/0,2/0}$
            & \footnotesize{35}  & \hspace{-4mm} $\pattern{scale=0.5}{2}{1/1,2/2}{1/2,2/2,0/1,2/1,0/0,1/0}$ \\ \hline
            \footnotesize{37} & \hspace{-4mm} $\pattern{scale=0.5}{2}{1/1,2/2}{1/2,2/2,0/1,1/1,2/1,1/0}\hspace{-1mm}\pattern{scale=0.5}{2}{1/1,2/2}{1/2,0/1,1/1,2/1,0/0,1/0}$
            & \footnotesize{38}  & \hspace{-4mm} $\pattern{scale=0.5}{2}{1/1,2/2}{1/1}$
            & \footnotesize{39}  & \hspace{-4mm} $\pattern{scale=0.5}{2}{1/1,2/2}{2/2}\hspace{-1mm}\pattern{scale=0.5}{2}{1/1,2/2}{0/0}$ \\ \hline
            \footnotesize{40} & \hspace{-4mm} $\pattern{scale=0.5}{2}{1/1,2/2}{0/2,1/1,2/0}$
            & \footnotesize{41}  & \hspace{-4mm} $\pattern{scale=0.5}{2}{1/1,2/2}{2/2,0/2,2/0}\hspace{-1mm}\pattern{scale=0.5}{2}{1/1,2/2}{0/2,0/0,2/0}$
            & \footnotesize{42}  & \hspace{-4mm} $\pattern{scale=0.5}{2}{1/1,2/2}{0/2,1/1}\hspace{-1mm}\pattern{scale=0.5}{2}{1/1,2/2}{1/1,2/0}$ \\ \hline
            \footnotesize{43} & \hspace{-4mm} $\pattern{scale=0.5}{2}{1/1,2/2}{0/2,2/2}\hspace{-1mm}\pattern{scale=0.5}{2}{1/1,2/2}{0/2,0/0}\hspace{-1mm}\pattern{scale=0.5}{2}{1/1,2/2}{2/2,2/0}\hspace{-1mm}\pattern{scale=0.5}{2}{1/1,2/2}{0/0,2/0}$
            & \footnotesize{44}  & \hspace{-4mm} $\pattern{scale=0.5}{2}{1/1,2/2}{0/2,2/2,2/1,0/0,1/0}\hspace{-1mm}\pattern{scale=0.5}{2}{1/1,2/2}{1/2,2/2,0/1,0/0,2/0}$
            & \footnotesize{50}  & \hspace{-4mm} $\pattern{scale=0.5}{2}{1/1,2/2}{2/2,0/1,1/1,2/1}\hspace{-1mm}\pattern{scale=0.5}{2}{1/1,2/2}{0/1,1/1,2/1,0/0}\hspace{-1mm}\pattern{scale=0.5}{2}{1/1,2/2}{1/2,2/2,1/1,1/0}\hspace{-1mm}\pattern{scale=0.5}{2}{1/1,2/2}{1/2,1/1,0/0,1/0}$ \\ \hline
            \footnotesize{51} & \hspace{-4mm} $\pattern{scale=0.5}{2}{1/1,2/2}{1/2,2/2,0/1,0/0}\hspace{-1mm}\pattern{scale=0.5}{2}{1/1,2/2}{2/2,2/1,0/0,1/0}$
            & \footnotesize{52}  & \hspace{-4mm} $\pattern{scale=0.5}{2}{1/1,2/2}{0/2,1/2,2/2,0/1,1/1,2/1,0/0,1/0,2/0}$
            & \footnotesize{55}  & \hspace{-4mm} $\pattern{scale=0.5}{2}{1/1,2/2}{1/2,0/2,1/1,2/1,0/0}\hspace{-1mm}\pattern{scale=0.5}{2}{1/1,2/2}{0/2,2/2,0/1,1/1,1/0}\hspace{-1mm}\pattern{scale=0.5}{2}{1/1,2/2}{1/2,1/1,2/1,0/0,2/0}\hspace{-1mm}\pattern{scale=0.5}{2}{1/1,2/2}{2/2,0/1,1/1,1/0,2/0}$ \\ \hline
            \footnotesize{57} & \hspace{-4mm} $\pattern{scale=0.5}{2}{1/1,2/2}{1/2,2/2,0/1,2/1,1/0}\hspace{-1mm}\pattern{scale=0.5}{2}{1/1,2/2}{1/2,0/1,2/1,0/0,1/0}$
            & \footnotesize{58}  & \hspace{-4mm} $\pattern{scale=0.5}{2}{1/1,2/2}{0/2,0/1,1/1,2/1,2/0}\hspace{-1mm}\pattern{scale=0.5}{2}{1/1,2/2}{0/2,1/2,1/1,1/0,2/0}$
            & \footnotesize{59}  & \hspace{-4mm} $\pattern{scale=0.5}{2}{1/1,2/2}{0/2,1/2,1/1,2/1,1/0}\hspace{-1mm}\pattern{scale=0.5}{2}{1/1,2/2}{0/2,0/1,1/1,2/1,1/0}\hspace{-1mm}\pattern{scale=0.5}{2}{1/1,2/2}{1/2,0/1,1/1,2/1,2/0}\hspace{-1mm}\pattern{scale=0.5}{2}{1/1,2/2}{1/2,0/1,1/1,1/0,2/0}$ \\ \hline
            \footnotesize{64} & \hspace{-4mm} $\pattern{scale=0.5}{2}{1/1,2/2}{0/2,1/1,2/1,0/0}\hspace{-1mm}\pattern{scale=0.5}{2}{1/1,2/2}{0/2,2/2,1/1,1/0}\hspace{-1mm}\pattern{scale=0.5}{2}{1/1,2/2}{2/2,0/1,1/1,2/0}\hspace{-1mm}\pattern{scale=0.5}{2}{1/1,2/2}{1/2,1/1,0/0,2/0}$
            & \footnotesize{65}  & \hspace{-4mm} $\pattern{scale=0.5}{2}{1/1,2/2}{0/2,2/2,0/1,1/1}\hspace{-1mm}\pattern{scale=0.5}{2}{1/1,2/2}{0/2,1/2,1/1,0/0}\hspace{-1mm}\pattern{scale=0.5}{2}{1/1,2/2}{1/1,2/1,0/0,2/0}\hspace{-1mm}\pattern{scale=0.5}{2}{1/1,2/2}{2/2,1/1,1/0,2/0}$
            & \footnotesize{66}  & \hspace{-4mm} $\pattern{scale=0.5}{2}{1/1,2/2}{0/2,2/2,2/1,0/0}\hspace{-1mm}\pattern{scale=0.5}{2}{1/1,2/2}{0/2,2/2,0/0,1/0}\hspace{-1mm}\pattern{scale=0.5}{2}{1/1,2/2}{1/2,2/2,0/0,2/0}\hspace{-1mm}\pattern{scale=0.5}{2}{1/1,2/2}{2/2,0/1,0/0,2/0}$ \\ \hline
            \footnotesize{72} & \hspace{-4mm} $\pattern{scale=0.5}{2}{1/1,2/2}{1/2,0/1,1/1}\hspace{-1mm}\pattern{scale=0.5}{2}{1/1,2/2}{1/1,2/1,1/0}$
            & \footnotesize{74}  & \hspace{-4mm} $\pattern{scale=0.5}{2}{1/1,2/2}{0/1,1/1,2/1}\hspace{-1mm}\pattern{scale=0.5}{2}{1/1,2/2}{1/2,1/1,1/0}$
            & \footnotesize{76}  & \hspace{-4mm} $\pattern{scale=0.5}{2}{1/1,2/2}{0/2,1/2,0/1,1/1,2/1,1/0,2/0}$ \\ \hline
            \footnotesize{79} & \hspace{-4mm} $\pattern{scale=0.5}{2}{1/1,2/2}{0/2,1/2,1/1,2/1,0/0,2/0}\hspace{-1mm}\pattern{scale=0.5}{2}{1/1,2/2}{0/2,2/2,0/1,1/1,1/0,2/0}$
            & \footnotesize{81}  & \hspace{-4mm} $\pattern{scale=0.5}{2}{1/1,2/2}{0/2,1/2,0/1,1/1,2/1,1/0}\hspace{-1mm}\pattern{scale=0.5}{2}{1/1,2/2}{1/2,0/1,1/1,2/1,1/0,2/0}$
            & \footnotesize{82}  & \hspace{-4mm} $\pattern{scale=0.5}{2}{1/1,2/2}{0/2,1/2,2/1,1/0}\hspace{-1mm}\pattern{scale=0.5}{2}{1/1,2/2}{0/2,0/1,2/1,1/0}\hspace{-1mm}\pattern{scale=0.5}{2}{1/1,2/2}{1/2,0/1,2/1,2/0}\hspace{-1mm}\pattern{scale=0.5}{2}{1/1,2/2}{1/2,0/1,1/0,2/0}$ \\ \hline
            \footnotesize{83} & \hspace{-4mm} $\pattern{scale=0.5}{2}{1/1,2/2}{1/2,0/1,2/1,1/0}$
            & \footnotesize{84}  & \hspace{-4mm} $\pattern{scale=0.5}{2}{1/1,2/2}{0/2,1/2,0/1,2/0}\hspace{-1mm}\pattern{scale=0.5}{2}{1/1,2/2}{0/2,2/1,1/0,2/0}$
            & \footnotesize{85}  & \hspace{-4mm} $\pattern{scale=0.5}{2}{1/1,2/2}{0/2,0/1,2/1,2/0}\hspace{-1mm}\pattern{scale=0.5}{2}{1/1,2/2}{0/2,1/2,1/0,2/0}$ \\ \hline
            \footnotesize{86} & \hspace{-4mm} $\pattern{scale=0.5}{2}{1/1,2/2}{0/2,1/2,0/1,2/1}\hspace{-1mm}\pattern{scale=0.5}{2}{1/1,2/2}{0/2,1/2,0/1,1/0}\hspace{-1mm}\pattern{scale=0.5}{2}{1/1,2/2}{1/2,2/1,1/0,2/0}\hspace{-1mm}\pattern{scale=0.5}{2}{1/1,2/2}{0/1,2/1,1/0,2/0}$
            & \footnotesize{87}  & \hspace{-4mm}\pattern{scale=0.5}{2}{1/1,2/2}{}
            & \footnotesize{88}  & \hspace{-4mm} $\pattern{scale=0.5}{2}{1/1,2/2}{0/2,2/0}$ \\ \hline
            \footnotesize{89} & \hspace{-4mm} $\pattern{scale=0.5}{2}{1/1,2/2}{0/2}\hspace{-1mm}\pattern{scale=0.5}{2}{1/1,2/2}{2/0}$
            & \footnotesize{90}  & \hspace{-4mm} $\pattern{scale=0.5}{2}{1/1,2/2}{1/2,0/1,2/1}\hspace{-1mm}\pattern{scale=0.5}{2}{1/1,2/2}{1/2,0/1,1/0}\hspace{-1mm}\pattern{scale=0.5}{2}{1/1,2/2}{1/2,2/1,1/0}\hspace{-1mm}\pattern{scale=0.5}{2}{1/1,2/2}{0/1,2/1,1/0}$
            & \footnotesize{91}  & \hspace{-4mm} $\pattern{scale=0.5}{2}{1/1,2/2}{0/2,2/1,1/0}\hspace{-1mm}\pattern{scale=0.5}{2}{1/1,2/2}{1/2,0/1,2/0}$ \\ \hline
            \footnotesize{92} & \hspace{-4mm} $\pattern{scale=0.5}{2}{1/1,2/2}{0/2,0/1,2/1}\hspace{-1mm}\pattern{scale=0.5}{2}{1/1,2/2}{0/2,1/2,1/0}\hspace{-1mm}\pattern{scale=0.5}{2}{1/1,2/2}{0/1,2/1,2/0}\hspace{-1mm}\pattern{scale=0.5}{2}{1/1,2/2}{1/2,1/0,2/0}$
            & \footnotesize{93}  & \hspace{-4mm} $\pattern{scale=0.5}{2}{1/1,2/2}{0/2,1/2,0/1}\hspace{-1mm}\pattern{scale=0.5}{2}{1/1,2/2}{2/1,1/0,2/0}$
            & \footnotesize{94}  & \hspace{-4mm} $\pattern{scale=0.5}{2}{1/1,2/2}{1/2,0/1}\hspace{-1mm}\pattern{scale=0.5}{2}{1/1,2/2}{2/1,1/0}$ \\ \hline
            \footnotesize{95} & \hspace{-4mm} $\pattern{scale=0.5}{2}{1/1,2/2}{0/1,2/1}\hspace{-1mm}\pattern{scale=0.5}{2}{1/1,2/2}{1/2,1/0}$
            & \footnotesize{96}  & \hspace{-4mm} $\pattern{scale=0.5}{2}{1/1,2/2}{0/2,1/2,0/1,2/1,1/0,2/0}$
            & \footnotesize{97}  & \hspace{-4mm} $\pattern{scale=0.5}{2}{1/1,2/2}{0/2,1/2,0/1,2/1,2/0}\hspace{-1mm}\pattern{scale=0.5}{2}{1/1,2/2}{0/2,1/2,0/1,1/0,2/0}\hspace{-1mm}\pattern{scale=0.5}{2}{1/1,2/2}{0/2,1/2,2/1,1/0,2/0}\hspace{-1mm}\pattern{scale=0.5}{2}{1/1,2/2}{0/2,0/1,2/1,1/0,2/0}$ \\ \hline
            \footnotesize{98} & \hspace{-4mm} $\pattern{scale=0.5}{2}{1/1,2/2}{0/2,1/2,0/1,2/1,1/0}\hspace{-1mm}\pattern{scale=0.5}{2}{1/1,2/2}{1/2,0/1,2/1,1/0,2/0}$
            & \footnotesize{99}  & \hspace{-4mm} $\pattern{scale=0.5}{2}{1/1,2/2}{1/2,2/1}\hspace{-1mm}\pattern{scale=0.5}{2}{1/1,2/2}{0/1,1/0}$
            & \footnotesize{100} & \hspace{-4mm} $\pattern{scale=0.5}{2}{1/1,2/2}{0/2,2/1}\hspace{-1mm}\pattern{scale=0.5}{2}{1/1,2/2}{0/2,1/0}\hspace{-1mm}\pattern{scale=0.5}{2}{1/1,2/2}{1/2,2/0}\hspace{-1mm}\pattern{scale=0.5}{2}{1/1,2/2}{0/1,2/0}$ \\ \hline
            \footnotesize{101} & \hspace{-4mm} $\pattern{scale=0.5}{2}{1/1,2/2}{0/2,1/2}\hspace{-1mm}\pattern{scale=0.5}{2}{1/1,2/2}{0/2,0/1}\hspace{-1mm}\pattern{scale=0.5}{2}{1/1,2/2}{2/1,2/0}\hspace{-1mm}\pattern{scale=0.5}{2}{1/1,2/2}{1/0,2/0}$
            & \footnotesize{102} & \hspace{-4mm} $\pattern{scale=0.5}{2}{1/1,2/2}{1/2}\hspace{-1mm}\pattern{scale=0.5}{2}{1/1,2/2}{0/1}\hspace{-1mm}\pattern{scale=0.5}{2}{1/1,2/2}{2/1}\hspace{-1mm}\pattern{scale=0.5}{2}{1/1,2/2}{1/0}$
            & \footnotesize{103} & \hspace{-4mm} $\pattern{scale=0.5}{2}{1/1,2/2}{0/2,1/2,2/1,2/0}\hspace{-1mm}\pattern{scale=0.5}{2}{1/1,2/2}{0/2,0/1,1/0,2/0}$ \\ \hline
            \footnotesize{104} & \hspace{-4mm} $\pattern{scale=0.5}{2}{1/1,2/2}{0/2,1/2,2/0}\hspace{-1mm}\pattern{scale=0.5}{2}{1/1,2/2}{0/2,0/1,2/0}\hspace{-1mm}\pattern{scale=0.5}{2}{1/1,2/2}{0/2,2/1,2/0}\hspace{-1mm}\pattern{scale=0.5}{2}{1/1,2/2}{0/2,1/0,2/0}$
            & \footnotesize{105} & \hspace{-4mm} $\pattern{scale=0.5}{2}{1/1,2/2}{0/2,1/2,2/1}\hspace{-1mm}\pattern{scale=0.5}{2}{1/1,2/2}{0/2,0/1,1/0}\hspace{-1mm}\pattern{scale=0.5}{2}{1/1,2/2}{1/2,2/1,2/0}\hspace{-1mm}\pattern{scale=0.5}{2}{1/1,2/2}{0/1,1/0,2/0}$
            & & \\ \hline
        \end{tabular}
    \end{center}
    \vspace{-0.5cm}
    \caption{The equivalence classes explained by trivial bijections.}
    \label{tab-remaining}
\end{table}

\section{Distributions of all classes}\label{all-distr-appendix}

The table below gives the distribution for each class of patterns on $n$-permutations for $1 \leq n \leq 7$.

\begin{longtable}{c|p{0.80\textwidth}}
\hline
\endfirsthead 

\hline
\endhead 

\hline
\endfoot 

\hline
\endlastfoot 

\tiny 1
&
\tiny $1, q + 1, 2q + 4, 5q + 19, 16q + 104, 64q + 656, 312q + 4728$ \\
\hline
\tiny 2
&
\tiny $1, q + 1, 2q + 4, 6q + 18, 24q + 96, 120q + 600, 720q + 4320$ \\
\hline
\tiny 3
&
\tiny $1, q + 1, 2q + 4, q^2 + 3q + 20, 3q^2 + 10q + 107, q^3 + 6q^2 + 49q + 664, 4q^3 + 19q^2 + 262q + 4755$ \\
\hline
\tiny 4
&
\tiny $1, q + 1, 2q^2 + 2q + 2, 5q^3 + 6q^2 + 7q + 6, 14q^4 + 22q^3 + 28q^2 + 32q + 24$, \\[-2mm]
& \tiny $42q^5 + 88q^4 + 128q^3 + 164q^2 + 178q + 120, 132q^6 + 364q^5 + 632q^4 + 896q^3 + 1132q^2 + 1164q + 720$ \\
\hline
\tiny 5
&
\tiny $1, q + 1, 2q^2 + 2q + 2, 5q^3 + 6q^2 + 8q + 5, 14q^4 + 23q^3 + 34q^2 + 32q + 17,$ \\[-2mm]
& \tiny $42q^5 + 99q^4 + 160q^3 + 190q^2 + 158q + 71, 132q^6 + 444q^5 + 825q^4 + 1156q^3 + 1226q^2 + 900q + 357$ \\
\hline
\tiny 6
&
\tiny $1, q + 1, 2q^2 + 2q + 2, 5q^3 + 7q^2 + 6q + 6, 15q^4 + 27q^3 + 30q^2 + 24q + 24,$ \\[-2mm] 
& \tiny $52q^5 + 116q^4 + 156q^3 + 156q^2 + 120q + 120, 203q^6 + 548q^5 + 867q^4 + 1022q^3 + 960q^2 + 720q + 720$ \\
\hline
\tiny 7
&
\tiny $1, q + 1, 2q^2 + 2q + 2, 6q^3 + 5q^2 + 7q + 6, 24q^4 + 16q^3 + 25q^2 + 31q + 24,$ \\[-2mm] 
& \tiny $120q^5 + 64q^4 + 103q^3 + 146q^2 + 167q + 120, 720q^6 + 312q^5 + 498q^4 + 742q^3 + 986q^2 + 1062q + 720$ \\
\hline
\tiny 8
&
\tiny $1, q + 1, 2q^2 + 2q + 2, 6q^3 + 6q^2 + 6q + 6, 24q^4 + 24q^3 + 24q^2 + 24q + 24,$ \\[-2mm]
& \tiny $120q^5 + 120q^4 + 120q^3 + 120q^2 + 120q + 120, 720q^6 + 720q^5 + 720q^4 + 720q^3 + 720q^2 + 720q + 720$ \\
\hline
\tiny 9
&
\tiny $1, q + 1, 2q^2 + 2q + 2, q^4 + 2q^3 + 10q^2 + 5q + 6, 2q^6 + 11q^4 + 19q^3 + 48q^2 + 16q + 24, q^9 + 2q^8 + 28q^6 + 4q^5 + 98q^4 + 149q^3 + 254q^2 + 64q + 120, 2q^{12} + 2q^{10} + 20q^9 + 41q^8 + 4q^7 + 306q^6 + 81q^5 + 863q^4 + 1169q^3 + 1520q^2 + 312q + 720$ \\
\hline
\tiny 10
&
\tiny $1, q + 1, 2q^2 + 2q + 2, q^4 + 2q^3 + 9q^2 + 6q + 6, 2q^6 + 8q^4 + 19q^3 + 42q^2 + 25q + 24, q^9 + 2q^8 + 17q^6 + 7q^5 + 69q^4 + 138q^3 + 234q^2 + 132q + 120, 2q^{12} + 2q^{10} + 11q^9 + 24q^8 + 11q^7 + 145q^6 + 106q^5 + 596q^4 + 1053q^3 + 1531q^2 + 839q + 720$ \\
\hline
\tiny 11
&
\tiny $1, q + 1, 2q^2 + 2q + 2, q^4 + 3q^3 + 7q^2 + 7q + 6, 2q^6 + 2q^5 + 9q^4 + 18q^3 + 33q^2 + 32q + 24, q^9 + 3q^8 + 4q^7 + 18q^6 + 20q^5 + 70q^4 + 113q^3 + 193q^2 + 178q + 120, 2q^{12} + 2q^{11} + 6q^{10} + 16q^9 + 37q^8 + 49q^7 + 153q^6 + 194q^5 + 552q^4 + 817q^3 + 1328q^2 + 1164q + 720$ \\
\hline
\tiny 12
&
\tiny $1, q + 1, 2q^2 + 2q + 2, q^4 + 3q^3 + 8q^2 + 6q + 6, 2q^6 + 2q^5 + 11q^4 + 21q^3 + 36q^2 + 24q + 24, q^9 + 3q^8 + 4q^7 + 25q^6 + 27q^5 + 90q^4 + 138q^3 + 192q^2 + 120q + 120, 2q^{12} + 2q^{11} + 6q^{10} + 22q^9 + 54q^8 + 68q^7 + 256q^6 + 282q^5 + 730q^4 + 978q^3 + 1200q^2 + 720q + 720$ \\
\hline
\tiny 13
&
\tiny $1, q + 1, 2q^2 + 2q + 2, q^4 + 4q^3 + 5q^2 + 7q + 7, 2q^6 + 4q^5 + 10q^4 + 14q^3 + 25q^2 + 30q + 35, q^9 + 4q^8 + 9q^7 + 16q^6 + 30q^5 + 53q^4 + 89q^3 + 131q^2 + 169q + 218, 2q^{12} + 4q^{11} + 12q^{10} + 22q^9 + 40q^8 + 66q^7 + 131q^6 + 216q^5 + 376q^4 + 560q^3 + 865q^2 + 1148q + 1598$ \\
\hline
\tiny 14
&
\tiny $1, q + 1, 2q^2 + 2q + 2, q^4 + 4q^3 + 5q^2 + 7q + 7, 2q^6 + 4q^5 + 10q^4 + 14q^3 + 25q^2 + 30q + 35, q^9 + 4q^8 + 9q^7 + 16q^6 + 31q^5 + 52q^4 + 87q^3 + 133q^2 + 170q + 217, 2q^{12} + 4q^{11} + 12q^{10} + 22q^9 + 41q^8 + 70q^7 + 132q^6 + 209q^5 + 366q^4 + 557q^3 + 881q^2 + 1158q + 1586$ \\
\hline
\tiny 15
&
\tiny $1, q + 1, 2q^2 + 3q + 1, 6q^3 + 11q^2 + 6q + 1, 24q^4 + 50q^3 + 35q^2 + 10q + 1, 120q^5 + 274q^4 + 225q^3 + 85q^2 + 15q + 1, 720q^6 + 1764q^5 + 1624q^4 + 735q^3 + 175q^2 + 21q + 1$ \\
\hline
\tiny 16
&
\tiny $1, q + 1, 2q^2 + 3q + 1, q^4 + 4q^3 + 11q^2 + 7q + 1, 2q^6 + 4q^5 + 15q^4 + 38q^3 + 43q^2 + 17q + 1, q^9 + 4q^8 + 9q^7 + 31q^6 + 67q^5 + 161q^4 + 228q^3 + 167q^2 + 51q + 1, 2q^{12} + 4q^{11} + 12q^{10} + 32q^9 + 83q^8 + 184q^7 + 421q^6 + 799q^5 + 1284q^4 + 1266q^3 + 747q^2 + 205q + 1$ \\
\hline
\tiny 17
&
\tiny $1, q + 1, 2q^2 + 3q + 1, q^4 + 4q^3 + 11q^2 + 7q + 1, 2q^6 + 4q^5 + 16q^4 + 36q^3 + 44q^2 + 17q + 1, q^9 + 4q^8 + 9q^7 + 36q^6 + 65q^5 + 159q^4 + 218q^3 + 176q^2 + 51q + 1, 2q^{12} + 4q^{11} + 12q^{10} + 37q^9 + 97q^8 + 187q^7 + 456q^6 + 742q^5 + 1244q^4 + 1243q^3 + 810q^2 + 205q + 1$ \\
\hline
\tiny 18
&
\tiny $1, q + 1, 2q^2 + 3q + 1, q^4 + 4q^3 + 11q^2 + 7q + 1, 2q^6 + 4q^5 + 17q^4 + 34q^3 + 45q^2 + 17q + 1, q^9 + 4q^8 + 9q^7 + 40q^6 + 66q^5 + 156q^4 + 205q^3 + 187q^2 + 51q + 1, 2q^{12} + 4q^{11} + 12q^{10} + 41q^9 + 108q^8 + 193q^7 + 510q^6 + 688q^5 + 1178q^4 + 1197q^3 + 901q^2 + 205q + 1$ \\
\hline
\tiny 19
&
\tiny $1, q + 1, 2q^2 + 4, 3q^3 + 3q^2 + q + 17, 4q^4 + 10q^3 + 12q^2 + 3q + 91, 5q^5 + 23q^4 + 49q^3 + 55q^2 + 14q + 574,$ \\[-2mm]
& \tiny $ 6q^6 + 44q^5 + 146q^4 + 291q^3 + 303q^2 + 77q + 4173$ \\
\hline
\tiny 20
&
\tiny $1, q + 1, 2q^2 + 4, 6q^3 + 18, 24q^4 + 96, 120q^5 + 600, 720q^6 + 4320$ \\
\hline
\tiny 21
&
\tiny $1, q + 1, 2q^2 + 4, q^4 + 2q^3 + 2q^2 + 2q + 17, 2q^6 + 4q^4 + 5q^3 + 12q^2 + 6q + 91, q^9 + 2q^8 + 7q^6 + 2q^5 + 19q^4 + 34q^3 + 52q^2 + 29q + 574, 2q^{12} + 2q^{10} + 5q^9 + 11q^8 + 44q^6 + 14q^5 + 129q^4 + 209q^3 + 291q^2 + 160q + 4173$ \\
\hline
\tiny 22
&
\tiny $1, q + 1, 2q^2 + 4, q^4 + 2q^3 + 6q + 15, 2q^6 + 2q^4 + 16q^2 + 20q + 80, q^9 + 2q^8 + 2q^5 + 8q^4 + 24q^3 + 51q^2 + 131q + 501, 2q^{12} + 2q^{10} + 20q^6 + 4q^5 + 60q^4 + 72q^3 + 354q^2 + 860q + 3666$ \\
\hline
\tiny 23
&
\tiny $1, q + 1, 2q^2 + q + 3, 4q^3 + 4q^2 + 4q + 12, 8q^4 + 14q^3 + 16q^2 + 21q + 61, 16q^5 + 42q^4 + 71q^3 + 84q^2 + 131q + 376, 32q^6 + 114q^5 + 271q^4 + 430q^3 + 536q^2 + 942q + 2715$ \\
\hline
\tiny 24
&
\tiny $1, q + 1, 2q^2 + q + 3, 4q^3 + 4q^2 + 5q + 11, 8q^4 + 15q^3 + 20q^2 + 24q + 53, 16q^5 + 50q^4 + 86q^3 + 113q^2 + 146q + 309, 32q^6 + 152q^5 + 355q^4 + 575q^3 + 778q^2 + 1029q + 2119$ \\
\hline
\tiny 25
&
\tiny $1, q + 1, 2q^2 + q + 3, 6q^3 + 2q^2 + 4q + 12, 24q^4 + 6q^3 + 10q^2 + 20q + 60, 120q^5 + 24q^4 + 36q^3 + 60q^2 + 120q + 360, 720q^6 + 120q^5 + 168q^4 + 252q^3 + 420q^2 + 840q + 2520$ \\
\hline
\tiny 26
&
\tiny $1, q + 1, 2q^2 + q + 3, q^4 + 2q^3 + 4q^2 + 6q + 11, 2q^6 + 5q^4 + 7q^3 + 23q^2 + 28q + 55, q^9 + 2q^8 + 8q^6 + 4q^5 + 27q^4 + 47q^3 + 124q^2 + 170q + 337, 2q^{12} + 2q^{10} + 5q^9 + 11q^8 + 4q^7 + 54q^6 + 31q^5 + 182q^4 + 306q^3 + 814q^2 + 1192q + 2437$ \\
\hline
\tiny 27
&
\tiny $1, q + 1, 2q^2 + q + 3, q^4 + 2q^3 + 4q^2 + 6q + 11, 2q^6 + 5q^4 + 7q^3 + 24q^2 + 26q + 56, q^9 + 2q^8 + 9q^6 + 2q^5 + 27q^4 + 55q^3 + 123q^2 + 152q + 349, 2q^{12} + 2q^{10} + 6q^9 + 13q^8 + 59q^6 + 20q^5 + 200q^4 + 369q^3 + 778q^2 + 1031q + 2560$ \\
\hline
\tiny 28
&
\tiny $1, q + 1, 2q^2 + q + 3, q^4 + 2q^3 + 6q^2 + 3q + 12, 2q^6 + 7q^4 + 13q^3 + 24q^2 + 14q + 60, q^9 + 2q^8 + 16q^6 + 4q^5 + 48q^4 + 81q^3 + 126q^2 + 82q + 360, 2q^{12} + 2q^{10} + 11q^9 + 25q^8 + 4q^7 + 118q^6 + 58q^5 + 361q^4 + 557q^3 + 816q^2 + 566q + 2520$ \\
\hline
\tiny 29
&
\tiny $1, q + 1, 2q^2 + q + 3, q^4 + 2q^3 + 6q^2 + 4q + 11, 2q^6 + 7q^4 + 14q^3 + 26q^2 + 20q + 51, q^9 + 2q^8 + 16q^6 + 6q^5 + 52q^4 + 92q^3 + 145q^2 + 119q + 287, 2q^{12} + 2q^{10} + 11q^9 + 26q^8 + 8q^7 + 126q^6 + 88q^5 + 416q^4 + 664q^3 + 973q^2 + 823q + 1901$ \\
\hline
\tiny 30
&
\tiny $1, q + 1, 2q^2 + q + 3, q^4 + 3q^3 + 3q^2 + 5q + 12, 2q^6 + 2q^5 + 6q^4 + 8q^3 + 17q^2 + 23q + 62, q^9 + 3q^8 + 4q^7 + 9q^6 + 12q^5 + 31q^4 + 49q^3 + 90q^2 + 134q + 387, 2q^{12} + 2q^{11} + 6q^{10} + 10q^9 + 19q^8 + 25q^7 + 63q^6 + 83q^5 + 209q^4 + 306q^3 + 584q^2 + 912q + 2819$ \\
\hline
\tiny 31
&
\tiny $1, q + 1, 3q + 3, 10q + 14, 40q + 80, 192q + 528, 1092q + 3948$ \\
\hline
\tiny 32
&
\tiny $1, q + 1, 3q + 3, 11q + 13, 50q + 70, 274q + 446, 1764q + 3276$ \\
\hline
\tiny 33
&
\tiny $1, q + 1, 3q + 3, 12q + 12, 60q + 60, 360q + 360, 2520q + 2520$ \\
\hline
\tiny 34
&
\tiny $1, q + 1, 3q + 3, 9q + 15, 31q + 89, 126q + 594, 606q + 4434$ \\
\hline
\tiny 35
&
\tiny $1, q + 1, 3q + 3, q^2 + 7q + 16, 5q^2 + 21q + 94, q^3 + 17q^2 + 89q + 613, 7q^3 + 59q^2 + 467q + 4507$ \\
\hline
\tiny 36
&
\tiny $1, q + 1, 3q + 3, q^2 + 8q + 15, 5q^2 + 30q + 85, q^3 + 19q^2 + 151q + 549, 7q^3 + 81q^2 + 909q + 4043$ \\
\hline
\tiny 37
&
\tiny $1, q + 1, 3q + 3, q^2 + 9q + 14, 6q^2 + 38q + 76, q^3 + 32q^2 + 207q + 480, 10q^3 + 195q^2 + 1344q + 3491$ \\
\hline
\tiny 38
&
\tiny $1, q + 1, 3q^2 + 2q + 1, 2q^4 + 11q^3 + 7q^2 + 3q + 1, 5q^6 + 22q^5 + 48q^4 + 28q^3 + 12q^2 + 4q + 1, 3q^9 + 17q^8 + 86q^7 + 181q^6 + 232q^5 + 125q^4 + 52q^3 + 18q^2 + 5q + 1, 7q^{12} + 26q^{11} + 129q^{10} + 416q^9 + 978q^8 + 1324q^7 + 1202q^6 + 600q^5 + 242q^4 + 84q^3 + 25q^2 + 6q + 1$ \\
\hline
\tiny 39
&
\tiny $1, q + 1, 3q^2 + 2q + 1, 2q^4 + 11q^3 + 7q^2 + 3q + 1, 8q^6 + 16q^5 + 51q^4 + 28q^3 + 12q^2 + 4q + 1, 6q^9 + 38q^8 + 82q^7 + 120q^6 + 267q^5 + 131q^4 + 52q^3 + 18q^2 + 5q + 1, 30q^{12} + 60q^{11} + 272q^{10} + 484q^9 + 730q^8 + 836q^7 + 1579q^6 + 682q^5 + 251q^4 + 84q^3 + 25q^2 + 6q + 1$ \\
\hline
\tiny 40
&
\tiny $1, q + 1, 3q^2 + 3, 2q^4 + 5q^3 + 4q^2 + 2q + 11, 5q^6 + 2q^5 + 14q^4 + 18q^3 + 16q^2 + 12q + 53, 3q^9 + 7q^8 + 4q^7 + 30q^6 + 30q^5 + 71q^4 + 92q^3 + 91q^2 + 77q + 315, 7q^{12} + 2q^{11} + 11q^{10} + 26q^9 + 68q^8 + 62q^7 + 195q^6 + 210q^5 + 480q^4 + 582q^3 + 602q^2 + 578q + 2217$ \\
\hline
\tiny 41
&
\tiny $1, q + 1, 3q^2 + 3, 2q^4 + 7q^3 + 2q^2 + 2q + 11, 8q^6 + 29q^4 + 9q^3 + 11q^2 + 12q + 51, 6q^9 + 26q^8 + 22q^6 + 137q^5 + 39q^4 + 58q^3 + 70q^2 + 75q + 287, 30q^{12} + 122q^{10} + 30q^9 + 78q^8 + 76q^7 + 867q^6 + 174q^5 + 330q^4 + 398q^3 + 506q^2 + 528q + 1901$ \\
\hline
\tiny 42
&
\tiny $1, q + 1, 3q^2 + q + 2, 2q^4 + 7q^3 + 7q^2 + 2q + 6, 5q^6 + 7q^5 + 28q^4 + 26q^3 + 24q^2 + 6q + 24, 3q^9 + 9q^8 + 21q^7 + 67q^6 + 99q^5 + 146q^4 + 125q^3 + 106q^2 + 24q + 120, 7q^{12} + 6q^{11} + 32q^{10} + 77q^9 + 210q^8 + 322q^7 + 678q^6 + 696q^5 + 873q^4 + 727q^3 + 572q^2 + 120q + 720$ \\
\hline
\tiny 43
&
\tiny $1, q + 1, 3q^2 + q + 2, 2q^4 + 8q^3 + 6q^2 + 2q + 6, 8q^6 + 3q^5 + 39q^4 + 21q^3 + 19q^2 + 6q + 24, 6q^9 + 28q^8 + 13q^7 + 65q^6 + 163q^5 + 135q^4 + 86q^3 + 80q^2 + 24q + 120, 30q^{12} + 8q^{11} + 144q^{10} + 116q^9 + 285q^8 + 168q^7 + 1299q^6 + 601q^5 + 676q^4 + 455q^3 + 418q^2 + 120q + 720$ \\
\hline
\tiny 44
&
\tiny $1, q + 1, 4q + 2, 16q + 8, 73q + 47, 388q + 332, 2396q + 2644$ \\
\hline
\tiny 45
&
\tiny $1, q + 1, 4q + 2, 18q + 6, 96q + 24, 600q + 120, 4320q + 720$ \\
\hline
\tiny 46
&
\tiny $1, q + 1, 4q + 2, 2q^2 + 13q + 9, 16q^2 + 52q + 52, 5q^3 + 95q^2 + 279q + 341, 64q^3 + 588q^2 + 1848q + 2540$ \\
\hline
\tiny 47
&
\tiny $1, q + 1, 4q + 2, q^2 + 14q + 9, 7q^2 + 59q + 54, q^3 + 36q^2 + 313q + 370, 10q^3 + 185q^2 + 1996q + 2849$ \\
\hline
\tiny 48
&
\tiny $1, q + 1, 4q + 2, q^2 + 16q + 7, 9q^2 + 78q + 33, q^3 + 69q^2 + 459q + 191, 16q^3 + 552q^2 + 3168q + 1304$ \\
\hline
\tiny 49
&
\tiny $1, q + 1, 5q + 1, 2q^2 + 21q + 1, 22q^2 + 97q + 1, 5q^3 + 181q^2 + 533q + 1, 93q^3 + 1457q^2 + 3489q + 1$ \\
\hline
\tiny 50
&
\tiny $1, q + 1, 5q + 1, 3q^2 + 20q + 1, 35q^2 + 84q + 1, 15q^3 + 295q^2 + 409q + 1, 315q^3 + 2359q^2 + 2365q + 1$ \\
\hline
\tiny 51
&
\tiny $1, q + 1, 5q + 1, q^2 + 22q + 1, 9q^2 + 110q + 1, q^3 + 59q^2 + 659q + 1, 13q^3 + 371q^2 + 4655q + 1$ \\
\hline
\tiny 52
&
\tiny $1, q + 1, 6, 24, 120, 720, 5040$ \\
\hline
\tiny 53
&
\tiny $1, q + 1, q + 5, 2q + 22, 6q + 114, 24q + 696, 120q + 4920$ \\
\hline
\tiny 54
&
\tiny $1, q + 1, q^2 + 2q + 3, q^3 + 2q^2 + 9q + 12, q^4 + 2q^3 + 10q^2 + 43q + 64, q^5 + 2q^4 + 11q^3 + 48q^2 + 246q + 412,$ \\[-2mm]
& \tiny $ q^6 + 2q^5 + 12q^4 + 53q^3 + 275q^2 + 1623q + 3074$ \\
\hline
\tiny 55
&
\tiny $1, q + 1, q^2 + 2q + 3, q^3 + 3q^2 + 7q + 13, q^4 + 4q^3 + 11q^2 + 35q + 69, q^5 + 5q^4 + 16q^3 + 54q^2 + 207q + 437,$ \\[-2mm]
& \tiny $ q^6 + 6q^5 + 22q^4 + 79q^3 + 312q^2 + 1411q + 3209$ \\
\hline
\tiny 56
&
\tiny $1, q + 1, q^2 + 2q + 3, q^3 + 3q^2 + 7q + 13, q^4 + 4q^3 + 12q^2 + 32q + 71, q^5 + 5q^4 + 18q^3 + 58q^2 + 177q + 461,$ \\[-2mm]
& \tiny $ q^6 + 6q^5 + 25q^4 + 92q^3 + 327q^2 + 1142q + 3447$ \\
\hline
\tiny 57
&
\tiny $1, q + 1, q^2 + 2q + 3, q^3 + 3q^2 + 7q + 13, q^4 + 4q^3 + 12q^2 + 33q + 70, q^5 + 5q^4 + 18q^3 + 61q^2 + 187q + 448, $ \\[-2mm]
& \tiny $q^6 + 6q^5 + 25q^4 + 98q^3 + 363q^2 + 1240q + 3307$ \\
\hline
\tiny 58
&
\tiny $1, q + 1, q^2 + 2q + 3, q^3 + 3q^2 + 8q + 12, q^4 + 4q^3 + 14q^2 + 40q + 61, q^5 + 5q^4 + 21q^3 + 79q^2 + 238q + 376, $ \\[-2mm]
& \tiny $q^6 + 6q^5 + 29q^4 + 128q^3 + 519q^2 + 1642q + 2715$ \\
\hline
\tiny 59
&
\tiny $1, q + 1, q^2 + 2q + 3, q^3 + 3q^2 + 8q + 12, q^4 + 4q^3 + 15q^2 + 38q + 62, q^5 + 5q^4 + 24q^3 + 82q^2 + 223q + 385, $ \\[-2mm]
& \tiny $q^6 + 6q^5 + 35q^4 + 148q^3 + 535q^2 + 1526q + 2789$ \\
\hline
\tiny 60
&
\tiny $1, q + 1, q^2 + 2q + 3, q^3 + 3q^2 + 8q + 12, q^4 + 4q^3 + 15q^2 + 38q + 62, q^5 + 5q^4 + 24q^3 + 84q^2 + 219q + 387, $ \\[-2mm]
& \tiny $q^6 + 6q^5 + 35q^4 + 156q^3 + 549q^2 + 1474q + 2819$ \\
\hline
\tiny 61
&
\tiny $1, q + 1, q^2 + 2q + 3, q^3 + 3q^2 + 9q + 11, q^4 + 4q^3 + 18q^2 + 44q + 53, q^5 + 5q^4 + 30q^3 + 110q^2 + 265q + 309, $ \\[-2mm]
& \tiny $q^6 + 6q^5 + 45q^4 + 220q^3 + 795q^2 + 1854q + 2119$ \\
\hline
\tiny 62
&
\tiny $1, q + 1, q^2 + 2q + 3, q^3 + 4q^2 + 6q + 13, q^4 + 7q^3 + 17q^2 + 25q + 70, q^5 + 11q^4 + 44q^3 + 85q^2 + 133q + 446, $ \\[-2mm]
& \tiny $q^6 + 16q^5 + 98q^4 + 294q^3 + 501q^2 + 854q + 3276$ \\
\hline
\tiny 63
&
\tiny $1, q + 1, q^2 + 2q + 3, q^3 + 4q^2 + 7q + 12, q^4 + 7q^3 + 19q^2 + 33q + 60, q^5 + 11q^4 + 47q^3 + 109q^2 + 192q + 360, $ \\[-2mm]
& \tiny $q^6 + 16q^5 + 102q^4 + 344q^3 + 737q^2 + 1320q + 2520$ \\
\hline
\tiny 64
&
\tiny $1, q + 1, q^2 + 3q + 2, q^3 + 4q^2 + 13q + 6, q^4 + 5q^3 + 19q^2 + 71q + 24, q^5 + 6q^4 + 26q^3 + 106q^2 + 461q + 120, $ \\[-2mm]
& \tiny $q^6 + 7q^5 + 34q^4 + 150q^3 + 681q^2 + 3447q + 720$ \\
\hline
\tiny 65
&
\tiny $1, q + 1, q^2 + 3q + 2, q^3 + 5q^2 + 12q + 6, q^4 + 7q^3 + 28q^2 + 60q + 24, q^5 + 9q^4 + 50q^3 + 179q^2 + 361q + 120, $ \\[-2mm]
& \tiny $q^6 + 11q^5 + 78q^4 + 387q^3 + 1305q^2 + 2538q + 720$ \\
\hline
\tiny 66
&
\tiny $1, q + 1, q^2 + 3q + 2, q^3 + 5q^2 + 12q + 6, q^4 + 8q^3 + 26q^2 + 61q + 24, q^5 + 12q^4 + 57q^3 + 156q^2 + 374q + 120, $ \\[-2mm]
& \tiny $q^6 + 17q^5 + 116q^4 + 429q^3 + 1083q^2 + 2674q + 720$ \\
\hline
\tiny 67
&
\tiny $1, q + 1, q^2 + 3q + 2, q^3 + 6q^2 + 10q + 7, q^4 + 10q^3 + 31q^2 + 45q + 33, q^5 + 15q^4 + 75q^3 + 182q^2 + 256q + 191, $ \\[-2mm]
& \tiny $q^6 + 21q^5 + 155q^4 + 572q^3 + 1248q^2 + 1739q + 1304$ \\
\hline
\tiny 68
&
\tiny $1, q + 1, q^2 + 3q + 2, q^3 + 6q^2 + 11q + 6, q^4 + 10q^3 + 35q^2 + 50q + 24, q^5 + 15q^4 + 85q^3 + 225q^2 + 274q + 120, $ \\[-2mm]
& \tiny $q^6 + 21q^5 + 175q^4 + 735q^3 + 1624q^2 + 1764q + 720$ \\
\hline
\tiny 69
&
\tiny $1, q + 1, q^2 + 3q + 2, q^3 + 6q^2 + 9q + 8, q^4 + 10q^3 + 26q^2 + 41q + 42, q^5 + 15q^4 + 60q^3 + 136q^2 + 243q + 265, $ \\[-2mm]
& \tiny $q^6 + 21q^5 + 120q^4 + 365q^3 + 887q^2 + 1690q + 1956$ \\
\hline
\tiny 70
&
\tiny $1, q + 1, q^2 + 3q + 2, q^3 + 6q^2 + 9q + 8, q^4 + 10q^3 + 27q^2 + 39q + 43, q^5 + 15q^4 + 64q^3 + 140q^2 + 223q + 277, $ \\[-2mm]
& \tiny $q^6 + 21q^5 + 130q^4 + 408q^3 + 885q^2 + 1525q + 2070$ \\
\hline
\tiny 71
&
\tiny $1, q + 1, q^2 + 3q + 2, q^3 + 6q^2 + 9q + 8, q^4 + 10q^3 + 28q^2 + 37q + 44, q^5 + 15q^4 + 69q^3 + 143q^2 + 202q + 290, $ \\[-2mm]
& \tiny $q^6 + 21q^5 + 145q^4 + 464q^3 + 862q^2 + 1343q + 2204$ \\
\hline
\tiny 72
&
\tiny $1, q + 1, q^2 + 4q + 1, q^3 + 10q^2 + 12q + 1, q^4 + 20q^3 + 63q^2 + 35q + 1, q^5 + 35q^4 + 224q^3 + 348q^2 + 111q + 1, $ \\[-2mm]
& \tiny $q^6 + 56q^5 + 630q^4 + 2027q^3 + 1920q^2 + 405q + 1$ \\
\hline
\tiny 73
&
\tiny $1, q + 1, q^2 + 4q + 1, q^3 + 10q^2 + 12q + 1, q^4 + 21q^3 + 61q^2 + 36q + 1, q^5 + 41q^4 + 226q^3 + 326q^2 + 125q + 1, $ \\[-2mm]
& \tiny $q^6 + 78q^5 + 722q^4 + 1912q^3 + 1786q^2 + 540q + 1$ \\
\hline
\tiny 74
&
\tiny $1, q + 1, q^2 + 4q + 1, q^3 + 11q^2 + 11q + 1, q^4 + 26q^3 + 66q^2 + 26q + 1, q^5 + 57q^4 + 302q^3 + 302q^2 + 57q + 1, $ \\[-2mm]
& \tiny $q^6 + 120q^5 + 1191q^4 + 2416q^3 + 1191q^2 + 120q + 1$ \\
\hline
\tiny 75
&
\tiny $1, q + 1, q^2 + 4q + 1, q^3 + 9q^2 + 13q + 1, q^4 + 16q^3 + 56q^2 + 46q + 1, q^5 + 25q^4 + 160q^3 + 334q^2 + 199q + 1, $ \\[-2mm]
& \tiny $q^6 + 36q^5 + 365q^4 + 1408q^3 + 2157q^2 + 1072q + 1$ \\
\hline
\tiny 76
&
\tiny $1, q + 1, q^2 + 5, q^3 + 2q + 21, q^4 + 2q^2 + 8q + 109, q^5 + 2q^3 + 9q^2 + 35q + 673, q^6 + 2q^4 + 10q^3 + 38q^2 + 192q + 4797$ \\
\hline
\tiny 77
&
\tiny $1, q + 1, q^2 + 5, q^3 + 2q + 21, q^4 + 3q^2 + 6q + 110, q^5 + 4q^3 + 9q^2 + 29q + 677, q^6 + 5q^4 + 12q^3 + 45q^2 + 160q + 4817$ \\
\hline
\tiny 78
&
\tiny $1, q + 1, q^2 + q + 4, q^3 + 3q^2 + 2q + 18, q^4 + 6q^3 + 11q^2 + 6q + 96, q^5 + 10q^4 + 35q^3 + 50q^2 + 24q + 600, $ \\[-2mm]
& \tiny $q^6 + 15q^5 + 85q^4 + 225q^3 + 274q^2 + 120q + 4320$ \\
\hline
\tiny 79
&
\tiny $1, q + 1, q^2 + q + 4, q^3 + q^2 + 4q + 18, q^4 + q^3 + 5q^2 + 14q + 99, q^5 + q^4 + 6q^3 + 18q^2 + 63q + 631, $ \\[-2mm]
& \tiny $q^6 + q^5 + 7q^4 + 22q^3 + 84q^2 + 333q + 4592$ \\
\hline
\tiny 80
&
\tiny $1, q + 1, q^2 + q + 4, q^3 + q^2 + 5q + 17, q^4 + q^3 + 6q^2 + 21q + 91, q^5 + q^4 + 7q^3 + 25q^2 + 112q + 574, $ \\[-2mm]
& \tiny $q^6 + q^5 + 8q^4 + 29q^3 + 134q^2 + 694q + 4173$ \\
\hline
\tiny 81
&
\tiny $1, q + 1, q^2 + q + 4, q^3 + q^2 + 6q + 16, q^4 + q^3 + 8q^2 + 27q + 83, q^5 + q^4 + 10q^3 + 39q^2 + 157q + 512, $ \\[-2mm]
& \tiny $q^6 + q^5 + 12q^4 + 52q^3 + 246q^2 + 1057q + 3671$ \\
\hline
\tiny 82
&
\tiny $1, q + 1, q^3 + 2q + 3, q^6 + 2q^3 + 3q^2 + 6q + 12, q^{10} + 2q^6 + 2q^5 + 3q^4 + 10q^3 + 13q^2 + 27q + 62, q^{15} + 2q^{10} + 2q^9 + q^8 + 5q^7 + 9q^6 + 13q^5 + 15q^4 + 66q^3 + 66q^2 + 155q + 385, q^{21} + 2q^{15} + 2q^{14} + q^{13} + 3q^{12} + 3q^{11} + 9q^{10} + 14q^9 + 11q^8 + 33q^7 + 54q^6 + 83q^5 + 141q^4 + 423q^3 + 411q^2 + 1060q + 2789$ \\
\hline
\tiny 83
&
\tiny $1, q + 1, q^3 + 2q + 3, q^6 + 2q^3 + q^2 + 10q + 10, q^{10} + 2q^6 + 2q^4 + 11q^3 + 11q^2 + 43q + 50, q^{15} + 2q^{10} + 2q^7 + 13q^6 + 24q^4 + 50q^3 + 69q^2 + 269q + 290, q^{21} + 2q^{15} + 2q^{11} + 13q^{10} + 2q^9 + 26q^7 + 62q^6 + 13q^5 + 147q^4 + 338q^3 + 578q^2 + 1838q + 2018$ \\
\hline
\tiny 84
&
\tiny $1, q + 1, q^3 + 2q + 3, q^6 + 3q^3 + 2q^2 + 5q + 13, q^{10} + 4q^6 + 3q^5 + 11q^3 + 14q^2 + 17q + 70, q^{15} + 5q^{10} + 4q^9 + 3q^7 + 14q^6 + 24q^5 + 5q^4 + 55q^3 + 86q^2 + 77q + 446, q^{21} + 6q^{15} + 5q^{14} + 4q^{12} + 20q^{10} + 39q^9 + 6q^8 + 39q^7 + 58q^6 + 157q^5 + 75q^4 + 351q^3 + 560q^2 + 443q + 3276$ \\
\hline
\tiny 85
&
\tiny $1, q + 1, q^3 + 2q + 3, q^6 + 3q^3 + 2q^2 + 6q + 12, q^{10} + 4q^6 + 2q^5 + 2q^4 + 12q^3 + 14q^2 + 25q + 60, q^{15} + 5q^{10} + 2q^9 + 2q^8 + 4q^7 + 15q^6 + 20q^5 + 20q^4 + 65q^3 + 92q^2 + 134q + 360, q^{21} + 6q^{15} + 2q^{14} + 2q^{13} + 4q^{12} + 2q^{11} + 23q^{10} + 24q^9 + 21q^8 + 46q^7 + 85q^6 + 158q^5 + 179q^4 + 435q^3 + 658q^2 + 874q + 2520$ \\
\hline
\tiny 86
&
\tiny $1, q + 1, q^3 + 2q + 3, q^6 + 3q^3 + q^2 + 8q + 11, q^{10} + 4q^6 + q^5 + 2q^4 + 14q^3 + 7q^2 + 38q + 53, q^{15} + 5q^{10} + q^9 + 2q^8 + 2q^7 + 23q^6 + 7q^5 + 18q^4 + 81q^3 + 47q^2 + 224q + 309, q^{21} + 6q^{15} + q^{14} + 2q^{13} + 2q^{12} + 3q^{11} + 32q^{10} + 10q^9 + 18q^8 + 25q^7 + 155q^6 + 60q^5 + 146q^4 + 561q^3 + 349q^2 + 1550q + 2119$ \\
\hline
\tiny 87
&
\tiny $1, q + 1, q^3 + 2q^2 + 2q + 1, q^6 + 3q^5 + 5q^4 + 6q^3 + 5q^2 + 3q + 1, q^{10} + 4q^9 + 9q^8 + 15q^7 + 20q^6 + 22q^5 + 20q^4 + 15q^3 + 9q^2 + 4q + 1, q^{15} + 5q^{14} + 14q^{13} + 29q^{12} + 49q^{11} + 71q^{10} + 90q^9 + 101q^8 + 101q^7 + 90q^6 + 71q^5 + 49q^4 + 29q^3 + 14q^2 + 5q + 1, q^{21} + 6q^{20} + 20q^{19} + 49q^{18} + 98q^{17} + 169q^{16} + 259q^{15} + 359q^{14} + 455q^{13} + 531q^{12} + 573q^{11} + 573q^{10} + 531q^9 + 455q^8 + 359q^7 + 259q^6 + 169q^5 + 98q^4 + 49q^3 + 20q^2 + 6q + 1$ \\
\hline
\tiny 88
&
\tiny $1, q + 1, q^3 + 2q^2 + 3, q^6 + 3q^5 + q^4 + 6q^3 + 2q + 11, q^{10} + 4q^9 + 3q^8 + 9q^7 + 6q^6 + 4q^5 + 22q^4 + 2q^3 + 8q^2 + 10q + 51, q^{15} + 5q^{14} + 6q^{13} + 13q^{12} + 18q^{11} + 6q^{10} + 46q^9 + 26q^8 + 16q^7 + 22q^6 + 112q^5 + 20q^4 + 30q^3 + 53q^2 + 59q + 287, q^{21} + 6q^{20} + 10q^{19} + 19q^{18} + 36q^{17} + 17q^{16} + 83q^{15} + 72q^{14} + 40q^{13} + 116q^{12} + 193q^{11} + 144q^{10} + 72q^9 + 136q^8 + 182q^7 + 674q^6 + 98q^5 + 222q^4 + 234q^3 + 379q^2 + 405q + 1901$ \\
\hline
\tiny 89
&
\tiny $1, q + 1, q^3 + 2q^2 + q + 2, q^6 + 3q^5 + 3q^4 + 5q^3 + 4q^2 + 2q + 6, q^{10} + 4q^9 + 6q^8 + 10q^7 + 13q^6 + 10q^5 + 20q^4 + 14q^3 + 12q^2 + 6q + 24, q^{15} + 5q^{14} + 10q^{13} + 18q^{12} + 29q^{11} + 31q^{10} + 50q^9 + 58q^8 + 54q^7 + 62q^6 + 78q^5 + 78q^4 + 54q^3 + 48q^2 + 24q + 120, q^{21} + 6q^{20} + 15q^{19} + 30q^{18} + 55q^{17} + 74q^{16} + 113q^{15} + 160q^{14} + 176q^{13} + 228q^{12} + 268q^{11} + 332q^{10} + 324q^9 + 372q^8 + 264q^7 + 534q^6 + 360q^5 + 384q^4 + 264q^3 + 240q^2 + 120q + 720$ \\
\hline
\tiny 90
&
\tiny $1, q + 1, q^3 + 3q + 2, q^6 + 4q^3 + 3q^2 + 10q + 6, q^{10} + 5q^6 + q^5 + 7q^4 + 17q^3 + 21q^2 + 44q + 24, q^{15} + 6q^{10} + q^9 + 2q^8 + 8q^7 + 29q^6 + 10q^5 + 60q^4 + 105q^3 + 139q^2 + 239q + 120, q^{21} + 7q^{15} + q^{14} + 2q^{13} + 2q^{12} + 10q^{11} + 35q^{10} + 19q^9 + 22q^8 + 88q^7 + 197q^6 + 138q^5 + 477q^4 + 759q^3 + 1024q^2 + 1538q + 720$ \\
\hline
\tiny 91
&
\tiny $1, q + 1, q^3 + 3q + 2, q^6 + 4q^3 + 4q^2 + 7q + 8, q^{10} + 5q^6 + 3q^5 + 5q^4 + 15q^3 + 24q^2 + 26q + 41, q^{15} + 6q^{10} + 4q^9 + 9q^7 + 22q^6 + 26q^5 + 37q^4 + 98q^3 + 116q^2 + 150q + 251, q^{21} + 7q^{15} + 5q^{14} + 4q^{12} + 7q^{11} + 26q^{10} + 46q^9 + 14q^8 + 81q^7 + 137q^6 + 203q^5 + 273q^4 + 681q^3 + 704q^2 + 1042q + 1809$ \\
\hline
\tiny 92
&
\tiny $1, q + 1, q^3 + 3q + 2, q^6 + 5q^3 + 2q^2 + 10q + 6, q^{10} + 7q^6 + 2q^5 + 3q^4 + 27q^3 + 14q^2 + 42q + 24, q^{15} + 9q^{10} + 2q^9 + 3q^8 + 4q^7 + 48q^6 + 24q^5 + 37q^4 + 165q^3 + 91q^2 + 216q + 120, q^{21} + 11q^{15} + 2q^{14} + 3q^{13} + 4q^{12} + 4q^{11} + 74q^{10} + 30q^9 + 54q^8 + 60q^7 + 373q^6 + 246q^5 + 361q^4 + 1139q^3 + 638q^2 + 1320q + 720$ \\
\hline
\tiny 93
&
\tiny $1, q + 1, q^3 + 3q + 2, q^6 + 6q^3 + 11q + 6, q^{10} + 10q^6 + 35q^3 + 50q + 24, q^{15} + 15q^{10} + 85q^6 + 225q^3 + 274q + 120, q^{21} + 21q^{15} + 175q^{10} + 735q^6 + 1624q^3 + 1764q + 720$ \\
\hline
\tiny 94
&
\tiny $1, q + 1, q^3 + 4q + 1, q^6 + 7q^3 + 3q^2 + 12q + 1, q^{10} + 11q^6 + 9q^4 + 36q^3 + 27q^2 + 35q + 1, q^{15} + 16q^{10} + 13q^7 + 92q^6 + 122q^4 + 193q^3 + 171q^2 + 111q + 1, q^{21} + 22q^{15} + 18q^{11} + 176q^{10} + 16q^9 + 252q^7 + 731q^6 + 92q^5 + 1140q^4 + 1182q^3 + 1004q^2 + 405q + 1$ \\
\hline
\tiny 95
&
\tiny $1, q + 1, q^3 + 4q + 1, q^6 + 7q^3 + 4q^2 + 11q + 1, q^{10} + 10q^6 + 2q^5 + 14q^4 + 32q^3 + 34q^2 + 26q + 1, q^{15} + 13q^{10} + 2q^9 + 4q^8 + 24q^7 + 77q^6 + 26q^5 + 179q^4 + 156q^3 + 180q^2 + 57q + 1, q^{21} + 16q^{15} + 2q^{14} + 4q^{13} + 4q^{12} + 38q^{11} + 109q^{10} + 84q^9 + 71q^8 + 406q^7 + 643q^6 + 462q^5 + 1392q^4 + 919q^3 + 768q^2 + 120q + 1$ \\
\hline
\tiny 96
&
\tiny $1, q + 1, q^3 + 5, q^6 + 3q + 20, q^{10} + 4q^3 + 9q + 106, q^{15} + 5q^6 + 12q^3 + 45q + 657, q^{21} + 6q^{10} + 15q^6 + 62q^3 + 249q + 4707$ \\
\hline
\tiny 97
&
\tiny $1, q + 1, q^3 + q + 4, q^6 + q^3 + 2q^2 + 3q + 17, q^{10} + q^6 + 2q^5 + q^4 + 5q^3 + 9q^2 + 10q + 91, q^{15} + q^{10} + 2q^9 + q^8 + 3q^7 + 3q^6 + 10q^5 + 5q^4 + 31q^3 + 46q^2 + 43q + 574, q^{21} + q^{15} + 2q^{14} + q^{13} + 3q^{12} + q^{11} + 4q^{10} + 11q^9 + 4q^8 + 24q^7 + 13q^6 + 56q^5 + 48q^4 + 210q^3 + 265q^2 + 223q + 4173$ \\
\hline
\tiny 98
&
\tiny $1, q + 1, q^3 + q + 4, q^6 + q^3 + 7q + 15, q^{10} + q^6 + 8q^3 + 3q^2 + 27q + 80, q^{15} + q^{10} + 9q^6 + 7q^4 + 30q^3 + 12q^2 + 167q + 493, q^{21} + q^{15} + 10q^{10} + 8q^7 + 37q^6 + 27q^4 + 187q^3 + 109q^2 + 1095q + 3565$ \\
\hline
\tiny 99
&
\tiny $1, q + 1, q^3 + q^2 + 2q + 2, q^6 + q^5 + 2q^4 + 4q^3 + 3q^2 + 8q + 5, q^{10} + q^9 + 2q^8 + 4q^7 + 5q^6 + 10q^5 + 11q^4 + 16q^3 + 21q^2 + 32q + 17, q^{15} + q^{14} + 2q^{13} + 4q^{12} + 5q^{11} + 12q^{10} + 13q^9 + 20q^8 + 31q^7 + 53q^6 + 51q^5 + 67q^4 + 105q^3 + 126q^2 + 158q + 71, q^{21} + q^{20} + 2q^{19} + 4q^{18} + 5q^{17} + 12q^{16} + 15q^{15} + 22q^{14} + 35q^{13} + 61q^{12} + 63q^{11} + 103q^{10} + 147q^9 + 198q^8 + 302q^7 + 321q^6 + 351q^5 + 549q^4 + 731q^3 + 860q^2 + 900q + 357$ \\
\hline
\tiny 100
&
\tiny $1, q + 1, q^3 + q^2 + 2q + 2, q^6 + q^5 + 2q^4 + 4q^3 + 4q^2 + 6q + 6, q^{10} + q^9 + 2q^8 + 4q^7 + 6q^6 + 9q^5 + 13q^4 + 17q^3 + 19q^2 + 24q + 24, q^{15} + q^{14} + 2q^{13} + 4q^{12} + 6q^{11} + 11q^{10} + 16q^9 + 22q^8 + 32q^7 + 49q^6 + 65q^5 + 66q^4 + 102q^3 + 103q^2 + 120q + 120, q^{21} + q^{20} + 2q^{19} + 4q^{18} + 6q^{17} + 11q^{16} + 18q^{15} + 25q^{14} + 37q^{13} + 60q^{12} + 84q^{11} + 104q^{10} + 163q^9 + 203q^8 + 282q^7 + 391q^6 + 387q^5 + 487q^4 + 686q^3 + 648q^2 + 720q + 720$ \\
\hline
\tiny 101
&
\tiny $1, q + 1, q^3 + q^2 + 2q + 2, q^6 + q^5 + 2q^4 + 5q^3 + 3q^2 + 6q + 6, q^{10} + q^9 + 2q^8 + 5q^7 + 7q^6 + 10q^5 + 14q^4 + 20q^3 + 12q^2 + 24q + 24, q^{15} + q^{14} + 2q^{13} + 5q^{12} + 7q^{11} + 15q^{10} + 19q^9 + 30q^8 + 37q^7 + 59q^6 + 74q^5 + 70q^4 + 100q^3 + 60q^2 + 120q + 120, q^{21} + q^{20} + 2q^{19} + 5q^{18} + 7q^{17} + 15q^{16} + 25q^{15} + 36q^{14} + 49q^{13} + 89q^{12} + 116q^{11} + 160q^{10} + 214q^9 + 240q^8 + 342q^7 + 474q^6 + 444q^5 + 420q^4 + 600q^3 + 360q^2 + 720q + 720$ \\
\hline
\tiny 102
 &
\tiny $1, q + 1, q^3 + q^2 + 3q + 1, q^6 + q^5 + 3q^4 + 5q^3 + 7q^2 + 6q + 1, q^{10} + q^9 + 3q^8 + 5q^7 + 12q^6 + 11q^5 + 26q^4 + 25q^3 + 25q^2 + 10q + 1, q^{15} + q^{14} + 3q^{13} + 5q^{12} + 12q^{11} + 17q^{10} + 32q^9 + 43q^8 + 70q^7 + 92q^6 + 117q^5 + 136q^4 + 110q^3 + 65q^2 + 15q + 1, q^{21} + q^{20} + 3q^{19} + 5q^{18} + 12q^{17} + 17q^{16} + 39q^{15} + 50q^{14} + 91q^{13} + 127q^{12} + 222q^{11} + 255q^{10} + 411q^9 + 478q^8 + 687q^7 + 694q^6 + 784q^5 + 616q^4 + 385q^3 + 140q^2 + 21q + 1$ \\
\hline
\tiny 103
&
\tiny $1, q + 1, q^3 + q^2 + 4, q^6 + q^5 + 4q^3 + q + 17, q^{10} + q^9 + 4q^7 + q^5 + 17q^4 + q^3 + q^2 + 3q + 91, q^{15} + q^{14} + 4q^{12} + q^{10} + 17q^9 + q^8 + q^7 + 4q^6 + 92q^5 + 7q^3 + 3q^2 + 14q + 574, q^{21} + q^{20} + 4q^{18} + q^{16} + 17q^{15} + q^{14} + q^{13} + 4q^{12} + 92q^{11} + q^{10} + 8q^9 + 3q^8 + 18q^7 + 577q^6 + 4q^5 + 17q^4 + 26q^3 + 14q^2 + 77q + 4173$ \\
\hline
\tiny 104
&
\tiny $1, q + 1, q^3 + q^2 + q + 3, q^6 + q^5 + q^4 + 4q^3 + 2q^2 + 3q + 12, q^{10} + q^9 + q^8 + 4q^7 + 3q^6 + 5q^5 + 14q^4 + 7q^3 + 11q^2 + 13q + 60, q^{15} + q^{14} + q^{13} + 4q^{12} + 3q^{11} + 6q^{10} + 16q^9 + 9q^8 + 17q^7 + 22q^6 + 73q^5 + 27q^4 + 41q^3 + 67q^2 + 72q + 360, q^{21} + q^{20} + q^{19} + 4q^{18} + 3q^{17} + 6q^{16} + 17q^{15} + 11q^{14} + 19q^{13} + 28q^{12} + 81q^{11} + 38q^{10} + 72q^9 + 99q^8 + 117q^7 + 442q^6 + 143q^5 + 204q^4 + 286q^3 + 467q^2 + 480q + 2520$ \\
\hline
\tiny 105
&
\tiny $1, q + 1, q^3 + q^2 + q + 3, q^6 + q^5 + q^4 + 4q^3 + q^2 + 5q + 11, q^{10} + q^9 + q^8 + 4q^7 + 2q^6 + 6q^5 + 12q^4 + 9q^3 + 7q^2 + 24q + 53, q^{15} + q^{14} + q^{13} + 4q^{12} + 2q^{11} + 7q^{10} + 13q^9 + 10q^8 + 11q^7 + 32q^6 + 65q^5 + 20q^4 + 57q^3 + 41q^2 + 146q + 309, q^{21} + q^{20} + q^{19} + 4q^{18} + 2q^{17} + 7q^{16} + 14q^{15} + 11q^{14} + 12q^{13} + 36q^{12} + 67q^{11} + 34q^{10} + 77q^9 + 58q^8 + 187q^7 + 390q^6 + 143q^5 + 157q^4 + 391q^3 + 299q^2 + 1029q + 2119$ \\
\hline
\end{longtable}
\end{document}